\documentclass[a4paper]{amsart}
\usepackage{amssymb} 
\usepackage[utf8]{inputenc}
\usepackage[T1]{fontenc} 
\usepackage{lmodern}
\usepackage[final]{microtype} 
\usepackage[hidelinks,pagebackref]{hyperref} 
   \hypersetup{final} 

\usepackage{mathtools} 
\usepackage[shortlabels]{enumitem}

\usepackage{tikz-cd} 
\usetikzlibrary{decorations.markings} 
\tikzset{negated/.style={
        decoration={markings,
            mark= at position 0.5 with {
                \node[transform shape] (tempnode) {$\backslash$};
            }
        },
        postaction={decorate}
    }
}

\usepackage{newunicodechar} 
\newunicodechar{ﬁ}{fi}
\newunicodechar{ﬀ}{ff}

\newtheorem{theorem}{Theorem}[section]
\newtheorem{lemma}[theorem]{Lemma}
\newtheorem{corollary}[theorem]{Corollary} 
\newtheorem{proposition}[theorem]{Proposition} 

\theoremstyle{definition}
\newtheorem{assumption}{Assumption} 
\newtheorem{definition}[theorem]{Definition}

\theoremstyle{remark}
\newtheorem{remark}[theorem]{Remark}
\newtheorem{example}[theorem]{Example}

\numberwithin{equation}{section}

\newcommand*{\RR}{\mathbb{R}}
\newcommand*{\NN}{\mathbb{N}}
\newcommand{\calX}{\mathcal{X}}
\newcommand*{\calA}{\mathcal{A}}
\newcommand{\calB}{\mathcal{B}}

\newcommand*{\ds}{D_{\ref{prop:symmetric-group-moments}}}

\newcommand{\fs}{D_{\ref{prop:moments-Poisson}}}
\def\ii{{\bf i}}

\newcommand*{\opt}{\text{opt}}
\newcommand*{\HS}{\text{HS}}

\DeclareMathOperator{\EE}{\mathbb{E}} 
\newcommand*{\PP}{\mathbb{P}} 
\DeclareMathOperator{\Dom}{Dom} 

\newcommand*{\calE}{\mathcal{E}}

\DeclareMathOperator{\Ent}{Ent}	
\DeclareMathOperator{\Var}{Var}	
\DeclareMathOperator{\Cov}{Cov}	
\DeclareMathOperator{\Med}{Med}

\DeclareMathOperator{\supp}{supp}

\newcommand*{\ind}[1]{\mathbf{1}_{#1}}

\newcommand*{\abs}[1]{\left\vert #1 \right\vert}
\newcommand*{\norm}[1]{\left\Vert #1 \right\Vert}

\makeatletter
\newcommand{\addresseshere}{%
  \enddoc@text\let\enddoc@text\relax
}
\makeatother

\title[Modified log-Sobolev and Beckner inequalities]{Modified log-Sobolev inequalities, Beckner inequalities and moment estimates}

\author[R. Adamczak]{Rados{\l}aw Adamczak} %
\address{Institute of Mathematics, University of Warsaw, Banacha 2, 02--097 Warsaw, Poland.}
\email{R.Adamczak@mimuw.edu.pl}

\author[B. Polaczyk]{Bart{\l}omiej Polaczyk}
\address{Institute of Mathematics, University of Warsaw, Banacha 2, 02--097 Warsaw, Poland.}
\email{B.Polaczyk@mimuw.edu.pl}

\author[M. Strzelecki]{Micha{\l} Strzelecki}
\address{Institute of Mathematics, University of Warsaw, Banacha 2, 02--097 Warsaw, Poland.}
\email{M.Strzelecki@mimuw.edu.pl}

\thanks{Research partially supported by the National Science Centre, Poland, via the Sonata Bis grant no.\ 2015/18/E/ST1/00214 (RA, BP) and the Preludium grant no.\ 2015/19/N/ST1/00891 (MS)}

\date{October 25, 2020}

\makeatletter
\@namedef{subjclassname@2020}{\textup{2020} Mathematics Subject Classification}
\makeatother


\begin{document}

\begin{abstract}
We prove that in the context of general Markov semigroups Beckner inequalities with constants separated from zero as $p\to 1^+$ are equivalent to the modified log Sobolev inequality (previously only one implication was known to hold in this generality). Further, by adapting an argument by Boucheron et al. we derive Sobolev type moment estimates which hold under these functional inequalities.

We illustrate our results with applications to concentration of measure estimates (also of higher order, beyond the case of Lipschitz functions) for various stochastic models, including random permutations, zero-range processes, strong Rayleigh measures, exponential random graphs, and geometric functionals on the Poisson path space.

\end{abstract}

\maketitle

\tableofcontents


\section{Introduction}

\subsection{Motivation and informal presentation}
Functional inequalities are among the most popular tools for studying concentration of measure phenomena, with the Poincar\'e and log-Sobolev inequalities being arguably the most important examples.
The former implies subexponential concentration for Lipschitz functions (as shown by Gromov and Milman~\cite{MR708367}), the latter by the well known Herbst's argument (see, e.g., \cite{MR1849347,MR3185193}) implies subgaussian bounds.
In the work~\cite{MR954373} Beckner proposed a family of inequalities interpolating between the Poincar\'e and log-Sobolev inequalities and showed that they held true in the case of standard Gaussian measures.
Their form was subsequently generalized by Lata{\l}a and Oleszkiewicz~\cite{MR1796718} who used them to obtain intermediate concentration estimates between subexponential and subgaussian.
See~\cite{MR3185193} for further developments.

While initially studied mostly in the analytic setting, for diffusions on $\RR^n$ or on Riemannian manifolds, all the aforementioned inequalities have their counterparts for general Markov semigroups, including those of Markov chains on discrete spaces.
They are however not unique, since due to the lack of the chain rule, two forms of a single inequality, which are equivalent in the continuous framework, may differ significantly in the general case.

In particular one distinguishes between the log-Sobolev inequality and a weaker modified log-Sobolev inequality.
Also Beckner's inequalities have two formulations, one of them stronger than the other one.
Before stating our results let us briefly recall some of those inequalities.
We will introduce the remaining ones in Section~\ref{sec:state-of-art}.
For now we will be working in the setting of Dirichlet forms and keep the presentation slightly informal.
Our general setting will be described more precisely in Section~\ref{sec:setting}.
Below $(\mathcal{X},\mathcal{B},\mu)$ is a probability space and $\Dom(\calE) \subseteq L_2(\mathcal{X},\mu)$ is a linear subspace on which a Dirichlet form $\calE$ is defined.

Recall that one says that $\mu$ and $\calE$ satisfy the Poincar\'{e} inequality if there exists a constant $\lambda>0$ such that
\begin{equation}
\label{eq:Poinc-intro}
	\lambda\Var_\mu(f) \le \calE(f,f)	
\end{equation}
for any $f\in \Dom(\calE)$, where $\Var_\mu (f) = \mu(f^2) - \mu(f)^2$ is the variance of $f$ treated as a random variable on the probability space $(\calX,\calB,\mu)$ (we use the common notation $\mu(f)=\int_\calX f\,d\mu$).

Define the entropy of a nonnegative function $f\colon \calX\to \RR$ as
\begin{displaymath}
\Ent_\mu(f) = \mu(f\log(f)) - \mu(f)\log(\mu(f)).
\end{displaymath}
One says that the  modified log-Sobolev inequality is satisfied if there exists a constant $\rho_0>0$ such that
\begin{equation}
 \label{eq:mlS-log-intro}
 \rho_0 \Ent_\mu(f)\leq \calE(f,\log f)
\end{equation}
for any  nonnegative $f\in \Dom(\calE)$ such that $\log f\in \Dom(\calE)$.

Finally we say that the Beckner inequality
with parameter $p\in(1,2]$ holds if there exists a constant $\alpha_p>0$ such that
\begin{equation}
 \label{eq:Beckner-GJ-intro}
 \alpha_p( \mu(f^p) - \mu(f)^p) \leq \frac{p}{2} \calE(f,f^{p-1})
\end{equation}
for any nonnegative $f\in \Dom(\calE)$ such that $f^{p-1}\in \Dom(\calE)$.

While for each individual $p$ the Beckner inequality~\eqref{eq:Beckner-GJ-intro} is equivalent to the Poincar\'e inequality, from the point of view of the concentration of measure theory the full strength of~\eqref{eq:Beckner-GJ-intro} is captured in the behaviour of the constants $\alpha_p$ as $p\to 1^+$.

In particular it is a well known observation made by many authors that the Beckner inequality~\eqref{eq:Beckner-GJ-intro} with $\alpha_p$ separated from zero on the interval $(1,2]$ implies the modified log-Sobolev inequality~\eqref{eq:mlS-log-intro}.
Indeed, it is enough to divide both sides of~\eqref{eq:Beckner-GJ-intro} by $p-1$ and take $\liminf$ as $p \to 1^+$ (see, e.g., \cite{MR2283379,MR3078024,MR3693525}).
However, somewhat surprisingly, the reverse implication is not present in the literature and in fact there are quite a few works where Beckner's inequalities are proved or discussed separately from the modified log-Sobolev inequality for the same models (see, e.g., \cite{MR1989444,MR2123200} and~\cite{MR2283379,MR3078024,MR3693525,conforti2020probabilistic,weber2020entropy}).
This is in contrast with the continuous case in which the equivalence has been obtained in~\cite{MR1796718} (for the Reader's convenience we describe all the connections between various inequalities in Section~\ref{sec:SOA} below).

Our main result (the proof of which is presented in Section~\ref{sec:mLStoBec}) can be summarized as follows.
\begin{theorem}\label{thm:main-intro}
  The modified log-Sobolev inequality~\eqref{eq:mlS-log-intro} holds with some constant $\rho_0>0$ if and only if the Beckner inequality~\eqref{eq:Beckner-GJ-intro} holds for every $p\in(1,2]$ with some $\alpha_p$ bounded away from zero.
Moreover,  the optimal constants with which they  hold satisfy $\rho^\opt_0(\mu) =2 \lim_{p\to 1^+}\alpha_p^\opt(\mu)$.
\end{theorem}

Apart from being important in its own right from the point of view of the abstract theory of functional inequalities for Markov semigroups, the above result is motivated by applications to the theory of concentration of measure.
The classical Herbst's argument allows for deducing deviation estimates for Lipschitz (in an appropriate sense) functions from log-Sobolev inequalities.
As proven by Aida and Stroock~\cite{MR1258492} (see also~\cite{Bobkov-growth} for the discrete case), the usual log-Sobolev inequality due to Gross~\cite{MR420249} (see Definition \ref{defi:inequalities} below) implies certain moment estimates, which can be a starting point for obtaining concentration for more general functions, in particular polynomials or more generally functions with bounded derivatives of higher order.
Such concentration results were obtained in \cite{MR3383337,MR3670792,MR3743923} in the continuous setting and subsequently in \cite{gtze2018higher,MR3949267} in the discrete one (in particular for the Ising model). They have found numerous applications in signal processing, statistics and computer science, where they proved an important tool for deriving theoretical guarantees in particular for compressed sensing type algorithms or for learning Ising models (see, e.g., \cite{ISI:000435979500008,dagan2020estimating,MR3914933,MR4073556}). They are also useful in random graph theory, allowing to obtain concentration inequalities for the subgraph count beyond the large deviation regime or for models with dependencies \cite{MR3383337,gtze2018concentration}.

In the discrete case however, the usual log-Sobolev inequality is much more restrictive than its modified version~\eqref{eq:mlS-log-intro} -- on infinite spaces it is strictly stronger, while on finite spaces it often holds with much worse constants, which affects the concentration estimates.
At the same time Beckner inequalities for product distributions (treated as a special case of modified $\phi$-Sobolev inequalities) were used by Boucheron, Bousquet, Lugosi, and Massart~\cite{MR2123200} in order to obtain moment estimates for functions of independent random variables which generalize the classical Efron--Stein inequality for the variance.
It turns out that their argument can be adapted to the setting of general semigroups and beyond, providing moment estimates of the same nature as those by Aida--Stroock but under a weaker assumption of modified log-Sobolev inequality~\eqref{eq:mlS-log-intro}.
This allows to treat a variety of models and obtain Sobolev type inequalities with various types of gradients.

Since the precise formulation of the general moment inequalities requires an introduction of some additional notation, we postpone it to Section~\ref{sec:Beckner-Aida-Stroock} (see Propositions \ref{prop:V^2} and \ref{prop:moments-abstract}). Here let us just mention some of their applications, which we present in Sections \ref{sec:applications} and \ref{sec:higher-order}. In the continuous case we derive $L_r$-Poincar\'e inequalities with optimal growth of constants as $r\to \infty$  for measures satisfying the Beckner--Lata{\l}a--Oleszkiewicz inequality, as well as new inequalities for Cauchy-type measures (Section \ref{sec:continuous}). In the case of jump processes, we obtain moment bounds with discrete gradients. In particular we obtain estimates for stationary measures of Glauber dynamics, including the Ising model, exponential random graphs or hardcore model (Section \ref{sec:Glauber}). They can be used to derive higher order concentration inequalities, which when specialized to polynomials improve the results from \cite{MR3949267,gtze2018concentration} (Section \ref{sec:higher-order}). For the symmetric group we generalize moment estimates obtained by Chatterjee for Hoeffding statistics \cite{MR2288072} to general functions (Section \ref{sec:S_n}). We apply them to empirical processes of sampling without replacement, improving recent results due to  Tolstikhin, Zhivotovskiy, and Blanchard \cite{MR3480745}. Building on recent work of Hermon and Salez \cite{hermon2019entropy,hermon2019modified} we also obtain Beckner inequalities and moment estimates for measures satisfying the stochastic covering property and zero-range processes (Sections \ref{sec:SCP}, \ref{sec:zero}). In Section \ref{sec:Poisson} we obtain moment estimates for the Poisson path space. We remark that even though concentration of measure and functional inequalities for the Poisson space have been an object of intensive studies (to mention \cite{MR1757600,MR1800540,MR2081075,MR3473096,MR3485348,MR3849811,MR3151752,gozlan2020transport}), to the best of our knowledge these are the first moment estimates for the Poisson space beyond the classical Poincar\'e inequality in $L_2$.

\subsubsection*{Acknowledgements}
We would like to thank Franck Barthe, Sergey Bobkov, and Pawe{\l} Wolff for discussions concerning the equivalence between various functional inequalities and encouragement to pursue the topics presented in this article.

\subsection{General setting}
\label{sec:setting}

Let $(\calX,\calB,\mu)$ be a probability space and consider a symmetric non-negative definite bilinear form $\mathcal{E}\colon \Dom(\calE)\times \Dom(\calE) \to \RR$, where  $\Dom(\calE)$ is a linear subspace of the space of $\mathcal{B}$-measurable functions $L_0(\calX,\mu)$.

We will consider the following abstract assumption on $\mathcal{E}$.
\begin{assumption}
\label{a:I}
\leavevmode
If $\varphi \colon \RR \to \RR$ is a contraction and $f \in \Dom(\calE)$, then also $\varphi\circ f \in \Dom(\calE)$. Moreover, for any $f_1,f_2,g_1,g_2 \in \Dom(\calE)$ if we have a pointwise inequality
\begin{align}\label{eq:assumption-1-2}
(f_1(x) - f_1(y))(f_2(x) - f_2(y)) \le (g_1(x) - g_1(y))(g_2(x) - g_2(y))
\end{align}
for all $x,y \in \calX$, then
\begin{align}\label{eq:assumption-1-3}
  \calE(f_1,f_2) \le \calE(g_1,g_2).
\end{align}
\end{assumption}

\begin{remark} Let us provide some basic consequences of Assumption~\ref{a:I}, which we are going to use.
First, if $f\in \Dom(\calE)$ and $\varphi\colon \RR\to \RR$ is a contraction, then
\begin{align}\label{eq:assumption-1-1}
 \calE(\varphi(f),\varphi(f)) \le \calE(f,f).
\end{align}
Another consequence is the equality
\begin{displaymath}
  \calE(f,c) = 0
\end{displaymath}
for any constant $c \in \RR$. Finally, if $\varphi\colon \RR\to \RR$ is non-decreasing and $f,\varphi\circ f \in \Dom(\calE)$, then
\begin{displaymath}
  \calE(f,\varphi\circ f) \ge 0.
\end{displaymath}
\end{remark}

Assumption~\ref{a:I} is satisfied in particular if $\mathcal{E}$ is a Dirichlet form of a reversible Markov semigroup (for completeness of the exposition we recall basic properties of Dirichlet forms in the Appendix~\ref{app-Markov}). This is one of the main motivations for our investigations, however we prefer not to restrict to this specific setting, since in applications to concentration of measure and moment inequalities one often encounters quadratic forms which do not correspond to Markov semigroups. For instance, it may happen that the functional inequalities of interest are in fact valid for a larger class of functions than the domain of the Dirichlet form associated with some $\mu$-reversible Markov process or that the quadratic form appearing on the right-hand side does not correspond to a Dirichlet form, while it still satisfies  Assumption~\ref{a:I} and the available functional inequalities are meaningful from the concentration of measure point of view. In addition, Assumption~\ref{a:I} will allow us to avoid unnecessary discussion of domains and help us state our main results in a more concise way.

In our examples with $\mathcal{E}$ one will often associate a subspace $\mathcal{A} \subseteq L_0(\calX,\mu)$ and a symmetric bilinear function $\Gamma \colon \calA\times \calA \to L_0(\mathcal{X},\mu)$ such that $\Gamma(f,f) \ge 0$ and for $f,g \in \calA \cap \Dom(\calE)$,
\begin{align}\label{eq:champ-form}
  \calE(f,g) = \int_\mathcal{X}\Gamma(f,g)d\mu.
\end{align}
In what follows we will write $\Gamma(f)$ for $\Gamma(f,f)$.

In the Markovian setting $\Gamma$ will be the \emph{carr\'e du champ} operator defined as
\begin{align}\label{eq:carre-du-champ}
  \Gamma(f,g) = \frac{1}{2}\Big(L(fg) - gLf - fLg\Big),
\end{align}
where $L$ is an infinitesimal operator of a reversible Markov semigroup on $L_2(\calX,\mu)$ with domain $\Dom(L)$. In this case  $\Gamma$ is first defined on a suitable algebra of functions $\calA_0 \subseteq \Dom(L)$ and then extended to a larger algebra $\calA$. We refer to the monograph~\cite{MR3155209} for a very detailed description of the relations between the domain of the infinitesimal generator, the domain of the Dirichlet form, and the algebra $\mathcal{A}$.

\subsection{Examples}

We will now provide several concrete examples covered by the setting described above.
We remark that even though our setting is not the same as in~\cite{MR3078024} the exposition below parallels to some extent the one from this article.

As a first example let us take a diffusion $(X_t)_{t\ge 0}$ on $\mathcal{X} = \RR^n$, with the infinitesimal generator $L$ given by
\begin{displaymath}
  Lf (x) = \sum_{i,j=1}^n a_{ij}(x)\frac{\partial^2 f(x)}{\partial x_i\partial x_j} + \sum_{i=1}^n b_i(x)\frac{\partial f(x)}{\partial x_i},
\end{displaymath}
$a = \sigma\sigma^T$, where $\sigma$ is a smooth, locally bounded function from $\RR^n$ to the space of $n\times d$ matrices and $b\colon \RR^n\to \RR^n$ is a smooth function.
In this case $\calA = C^\infty(\RR^d)$ is the set of all smooth functions and
\begin{displaymath}
  \Gamma(f,g) = \sum_{i,j=1}^n a_{ij} \frac{\partial f}{\partial x_i}\frac{\partial g}{\partial x_j}.
\end{displaymath}
In order to make this class of processes fit into our setting, we need to assume that $(X_t)_{t\ge 0}$ has an invariant probability measure $\mu$, in which case one defines
\begin{displaymath}
  \calE(f,g) = \int_{\mathcal{X}} \sum_{i,j=1}^n a_{ij} \frac{\partial f}{\partial x_i}\frac{\partial g}{\partial x_j} \mu(dx)
\end{displaymath}
for $f,g \in \calA_0$ -- the space of smooth compactly supported functions, and then extends this to an appropriate domain, which is the completion of $\calA_0$ with respect to the norm $\|f\| = \sqrt{\mu(f^2) + \calE(f,f)}$.
The assumption concerning the existence of $\mu$ is satisfied, e.g., if $a$ is the identity matrix and $b = - \nabla V$ for some function $V \colon \RR \to \RR$ such that $e^{-V}$ is integrable. One can then show that  the normalized measure $\mu(dx) = \frac{1}{Z}e^{-V(x)}dx$ is an invariant measure of the process. One can also consider more general diffusions on Riemannian manifolds.
At this point we should stress that this class of examples satisfies the chain rule and as a consequence many functional inequalities become equivalent, even though in the general situation they are not. For this reason, this class will not be in our focus in the subsequent part of the article, even though we will state some Sobolev type estimates which to our best knowledge are new also in this setting (see Section~\ref{sec:continuous}).

\medskip

Another particular case of the operator $\Gamma$, which will become for us an important source of examples, is given by
\begin{align}\label{eq:Gamma-kernel}
\Gamma(f,g)(x) = \frac{1}{2}\int_\mathcal{X}(f(y) - f(x))(g(y) - g(x))Q_x(dy),
\end{align}
where $x\mapsto Q_x$ is a map from $\mathcal{X}$ to the set of positive measures on $\calX$ such that for all $A \in \calB$, $x\mapsto Q_x(A)$ is measurable and $Q_x$, $\mu$ satisfy the following detailed balance condition:
\begin{equation}\label{eq:detailed-balance}
  Q_x(dy)\mu(dx) = Q_y(dx)\mu(dy).
\end{equation}
The bilinear form is well-defined on $\calA\times \calA$, where
\begin{displaymath}
\calA = \{f\in L_0(\calX,\mu)\colon \int_\mathcal{X}(f(y) - f(x))^2Q_x(dy) < \infty \; \textrm{$\mu$-a.s.}\}.
\end{displaymath}
In this case
\begin{align}\label{eq:kernel-to-form}
  \calE(f,g) = \frac{1}{2}\int_\calX \int_\mathcal{X}(f(y) - f(x))(g(y) - g(x))Q_x(dy)\mu(dx)
\end{align}
with $\Dom(\calE) = \{f\in L_0(\mathcal{X},\mu)\colon \int_\calX\int_\calX (f(y)-f(x))^2Q_x(dy)\mu(dx) < \infty\}$.
It is straightforward to check that in this case Assumption~\ref{a:I}  is satisfied.
Moreover, by the detailed balance condition~\eqref{eq:detailed-balance}
we can further write
\begin{align}\label{eq:bilinear-form-reversibility}
	  \calE(f,g) = \int_\calX \int_\mathcal{X}(f(x) - f(y))_+(g(x) - g(y))Q_x(dy)\mu(dx)
\end{align}
and
\begin{align}\label{eq:form-reversibility}
  \calE(f,f) = \int_\calX \int_\mathcal{X}(f(x) - f(y))_+^2Q_x(dy)\mu(dx) = \int_\calX \Gamma_+(f) d\mu,
\end{align}
where
\begin{align}\label{eq:Gamma+}
  \Gamma_+(f)(x) = \int_\mathcal{X}(f(x) - f(y))_+^2Q_x(dy).
\end{align}
We remark that in many applications to concentration of measure, passing from $\Gamma$ to $\Gamma_+$ is essential, since the latter can be often effectively bounded, especially under certain convexity or monotonicity assumptions on the function $f$.

The case when $\mathcal{X}$ is countable and $Q_x(\calX) < \infty$ for all $x \in \calX$, corresponds to the Markov jump process with generator
\begin{displaymath}
  L f(x) = \int_\calX (f(y) - f(x))Q_x(dy).
\end{displaymath}
We will however see that examples of this nature appear also in spaces which are not necessarily discrete, e.g., on the Poisson space and for general product spaces endowed with Glauber type dynamics.

Let us also note that in some applications one considers $\Gamma$ which does not correspond to a Markov process for which $\mu$ is the invariant measure, but for instance to a Markov process reversible with respect to some other measure of reference. Examples of this kind can be found, e.g., in~\cite{MR3926125}, where the authors consider log-Sobolev inequalities for the Ising model as well as quenched log-Sobolev inequalities for the Sherrington--Kirkpatrick model with $\Gamma$ being the carr\'e du champ operator for the Glauber dynamics induced by the product measure on the cube. The Assumption~\ref{a:I} is also satisfied in this situation.

\subsection{Functional inequalities} \label{sec:state-of-art}\label{sec:SOA}

Let us now introduce more precisely the functional inequalities we will investigate. In addition to restating the definitions of Poincar\'e, modified log-Sobolev, and the Beckner inequalities in the abstract setting described in Section~\ref{sec:setting}, we will introduce the usual log-Sobolev inequality and Beckner inequality in its original version from~\cite{MR954373}.

\begin{definition}\label{defi:inequalities} Let $\calE$ be a symmetric, nonnegative definite bilinear form on $\Dom(\calE)\times \Dom(\calE)$, where $\Dom(\calE)$ is a linear subspace of $L_0(\calX,\mu)$.
We will say that:

 (i) the Poincar\'{e} inequality is satisfied if there exists a constant $\lambda>0$ such that
\begin{equation}
\label{eq:Poinc}
	\lambda\Var_\mu(f) \le \calE(f,f)
	\tag{\text{P}}
\end{equation}
for any $f\in \Dom(\calE)$;

 (ii)  the  modified log-Sobolev inequality is satisfied if there exists a constant $\rho_0>0$ such that
\begin{equation}
 \label{eq:mlS-log}
 \rho_0 \Ent_\mu(f)\leq \calE(f,\log f)
 \tag{\text{mLSI}}
\end{equation}
for any  nonnegative $f\in \Dom(\calE)$ such that $\log f\in \Dom(\calE)$;

 (iii)  Beckner's inequality~\eqref{eq:Beckner-GJ} with parameter $p\in(1,2]$ holds if there exists a constant $\alpha_p>0$ such that
\begin{equation}
 \label{eq:Beckner-GJ}
 \alpha_p( \mu(f^p) - \mu(f)^p) \leq \frac{p}{2} \calE(f,f^{p-1})
 \tag{\text{Bec-p}}
\end{equation}
for any nonnegative $f\in \Dom(\calE)$ such that $f^{p-1}\in \Dom(\calE)$; 

  (iv) the log-Sobolev inequality is satisfied if there exists a constant $\rho_1>0$ such that
\begin{equation}
 \label{eq:mlS-sqrts}
 \rho_1 \Ent_\mu(g^2)\leq \calE(g,g)
 \tag{\text{LSI}}
\end{equation}
for any $g \in \Dom(\calE)$;

(v)  Beckner's inequality~\eqref{eq:Beckner-LO} with parameter $q\in[1,2)$ holds if there exists a constant $\beta_q>0$ such that
\begin{equation}
 \label{eq:Beckner-LO}
 \beta_q (\mu(g^2) - \mu(g^q)^{2/q}) \leq (2-q) \calE(g,g)
 \tag{\text{Bec-q}}
\end{equation}
for any nonnegative $g\in \Dom(\calE)$.
\end{definition}

\begin{remark} Since we only assume that $\Dom(\calE) \subseteq L_0(\calX,\mu)$, the inequalities introduced above assert in particular that the left-hand sides are well-defined. In general estimates of the form $A \le B$ in this article should be understood as: \emph{if $B < \infty$, then $A$ is well-defined and the inequality holds}.
\end{remark}
If $\calE$ is a Dirichlet form corresponding to a diffusion, then by a substitution $f = g^2$ and by the chain rule one can easily see that the modified log-Sobolev inequality~\eqref{eq:mlS-log} and the log-Sobolev inequality~\eqref{eq:mlS-sqrts} are equivalent. Similarly Beckner's inequality~\eqref{eq:Beckner-GJ} for given $p$ is equivalent to Beckner's inequality~\eqref{eq:Beckner-LO} for $q=2/p$ (one substitutes $f^p=g^2$).

In general however there is no such equivalence. It remains true that the Poincar\'e inequality is implied by each of the other inequalities. Other known relations between them are presented in Figure~\ref{fig:diagram}.
Below we briefly comment on each of the implications.
Since usually they are proved in the literature in a particular context, not necessarily agreeing with our setting, in Appendix~\ref{app-implications} we also provide their proofs (being simple adjustments of the arguments known from the literature).
In the next section we will prove the remaining implication, between~\eqref{eq:mlS-log} and~\eqref{eq:Beckner-GJ}, in particular proving Theorem~\ref{thm:main-intro}.

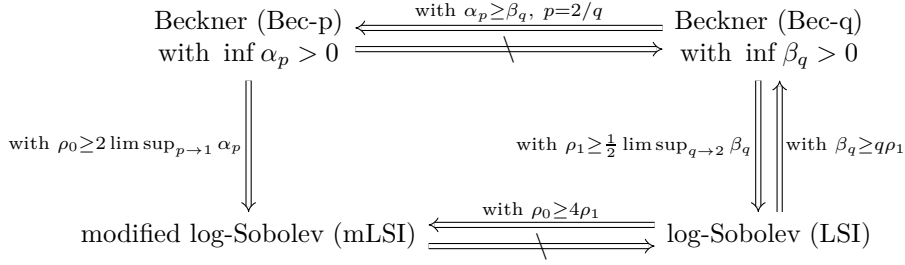
\begin{figure}[h] 
\begin{tikzcd}[arrows=Rightarrow, row sep=5em , column sep=8.5em]
\begin{matrix}
 \text{Beckner}\ \eqref{eq:Beckner-GJ}\\
 \text{with}\ \inf \alpha_p>0
\end{matrix}
     \arrow{d}[swap]{\text{with } \rho_0 \geq 2 \limsup_{p\to 1} \alpha_p}
     \arrow[negated, shift right=0.4em]{r}
&
\begin{matrix}
 \text{Beckner}\ \eqref{eq:Beckner-LO}\\
 \text{with}\ \inf \beta_q>0
\end{matrix}
     \arrow[shift right=0.4em]{l}[swap]{\text{with } \alpha_p\geq \beta_q,\ p=2/q}
     \arrow[shift right=0.4em]{d}[swap]{\text{with } \rho_1 \geq \frac{1}{2} \limsup_{q\to 2} \beta_q}
\\
\text{modified log-Sobolev}\ \eqref{eq:mlS-log}
     \arrow[negated, shift right=0.4em]{r}
&
\text{log-Sobolev}\ \eqref{eq:mlS-sqrts}
     \arrow[shift right=0.4em]{l}[swap]{\text{with } \rho_0 \geq 4\rho_1}
     \arrow[shift right=0.4em]{u}[swap]{\text{with } \beta_q \geq q \rho_1}
\end{tikzcd}
\caption{Arrows denote known implications.}
\label{fig:diagram} 
\end{figure}

The implication $\eqref{eq:mlS-sqrts} \implies~\eqref{eq:mlS-log}$ with $\rho_0\geq 4\rho_1$ was obtained by Bobkov and Tetali in~\cite{MR2283379}. The reverse implication is not true in general: if $\mu = \text{Poiss}(\lambda)$, and one considers $$\mathcal{E}(f,g) = \sum_{n \geq 0} (f(n+1)-f(n))^2\mu(\{n\})$$ corresponding to the birth and death Markov process with generator $Lf(n) = f(n+1)-f(n)+\lambda^{-1} n (f(n-1)-f(n))$, then the log-Sobolev inequality ~\eqref{eq:mlS-sqrts} does not hold (see~\cite{MR1636948}), while the modified log-Sobolev inequality~\eqref{eq:mlS-log} is satisfied (see~\cite{MR1944012}).

As observed in the original article~\cite{MR954373} by Beckner, if~\eqref{eq:Beckner-LO} holds for every $q\in[1,2)$ with $\beta_q$ bounded away from zero, then the log-Sobolev inequality~\eqref{eq:mlS-sqrts} holds as well with $\rho_1 \geq \frac{1}{2}\limsup_{q\to 2^-}\beta_q$. The reverse implication can be found in~\cite{MR1796718}.

The implication $\eqref{eq:Beckner-LO} \implies~\eqref{eq:Beckner-GJ}$ with  $\alpha_p\ge \beta_q$ (where $p = 2/q$) seems to be a part of folklore (we have not been able to find an explicit statement in the literature). It can be easily proved using arguments used to the best of our knowledge for the first time in~\cite{MR1410112} (see Appendix~\ref{app-implications} for details). The reverse implication also holds, but in this case one gets $\beta_q \ge q(2-q)\alpha_p$, so the dependence on constants degenerates when $q\to 2$. Such a degeneration indeed takes place, as the Poisson measure satisfies~\eqref{eq:Beckner-GJ} with $\alpha_p$ separated from zero (which can be easily proved by known results on the two point space~\cite{MR2283379} together with tensorization and Poisson limit theorem, similarly as it was done in~\cite{MR1944012} for the modified log-Sobolev inequality), whereas it cannot satisfy~\eqref{eq:Beckner-LO} with $\beta_q$ separated from zero, since this would imply~\eqref{eq:mlS-sqrts}, which as already mentioned fails for the Poisson measure.

The observation that if~\eqref{eq:Beckner-GJ} holds for every $p\in(1,2]$ with $\alpha_p$ bounded away from zero, then the modified log-Sobolev inequality ~\eqref{eq:mlS-log} holds with $\rho_0 \geq 2\limsup_{p\to 1^+}\alpha_p$ can be found, e.g., in~\cite{MR2283379} or~\cite{MR3078024}  (to see this divide both sides of~\eqref{eq:Beckner-GJ}  by $p-1$ and take $p\to 1^+$).

One can thus see that to complete the above diagram one should verify whether $\eqref{eq:mlS-log} \implies~\eqref{eq:Beckner-GJ}$ with $\inf_{p\in (1,2]} \alpha_p > 0$. We will establish this implication in the following section.

\section{From modified log-Sobolev to Beckner's inequalities}\label{sec:mLStoBec}

\subsection{Main result}
Throughout this section we assume that we are in the setting described in Section~\ref{sec:setting}, in particular that Assumption~\ref{a:I} holds.
The next theorem contains the precise statement of the result announced above, in Theorem~\ref{thm:main-intro}.

\begin{theorem}\label{thm:main:mLS_to_Bec}
 Let $\mu$ be a probability measure which satisfies the modified log-Sobolev inequality~\eqref{eq:mlS-log} with constant $\rho_0>0$.
 For  $p\in(1,2]$ and $\theta\in(0,1)$ denote
 \begin{align*}
  k(p,\theta) &\coloneqq \Bigl(1 - \frac{2\bigl((1+ \theta)^{p}-1 \bigr)}{p(p-1)(1-\theta)^2} \Bigr)
  \cdot		\frac{\theta^{p-1}}{e^{p-1}(1+\theta)^{p-1}},\\
  K_p &\coloneqq
  \max\Bigl\{
  	(1-1/p);\
  	\frac{p}{2}\cdot\!\sup_{\theta\in(0,1)} k(p,\theta) \Bigr\}.  		
 \end{align*}
  Then, for  any $p\in (1,2]$, $\mu$ satisfies the Beckner inequality~\eqref{eq:Beckner-GJ}  with constant $\alpha_p \ge K_p \rho_0$.

 Moreover, $\lim_{p\to 1^+} K_p = \lim_{p\to 2^-} K_p = 1/2$ and $\inf_{p\in(1,2]} K_p \geq 0.17$. In particular $\alpha_p \ge \rho_0/6$.
\end{theorem}

Note that this result is sharp in the most interesting regime, $p\to 1^+$, since if the Beckner inequality~\eqref{eq:Beckner-GJ} holds with some  constants $\alpha_p$, then the modified log-Sobolev inequality~\eqref{eq:mlS-log} holds with $\rho_0 \ge 2 \limsup_{p\to 1^+} \alpha_p$.
Combining this observation and the above theorem yields immediately the following corollary, which in particular implies Theorem \ref{thm:main-intro} from the Introduction.

\begin{corollary}
\label{cor:constants_in_main_thm}
The modified log-Sobolev inequality~\eqref{eq:mlS-log} holds with some constant $\rho_0>0$ if and only if the Beckner inequality~\eqref{eq:Beckner-GJ} holds for every $p\in(1,2]$ with some $\alpha_p$ bounded away from zero.
Moreover,  the optimal constants with which they  hold satisfy $\rho_0^\opt(\mu) = 2\lim_{p\to 1^+}\alpha_p^\opt(\mu)$.
\end{corollary}

\subsection{Auxiliary lemmas}
In this section we gather technical lemmas to be used in the proof of Theorem~\ref{thm:main:mLS_to_Bec}.

Since we work in the abstract setting described in Section~\ref{sec:setting}, we need the following lemma which asserts that it suffices to check the validity of the inequality~\eqref{eq:Beckner-GJ} for bounded functions only. Its proof, as well as proofs of some other auxiliary lemmas, is deferred to Appendix~\ref{app-auxiliary-lemmas}.
\begin{lemma}\label{L:Bec=>finite-mean}
	If for some $p\in(1,2]$ the the Beckner inequality~\eqref{eq:Beckner-GJ} is satisfied (with some constant $\alpha_p>0$) for all bounded nonnegative functions $f$ such that $f,f^{p-1}\in Dom(\calE)$, then it is satisfied with the same constant for all nonnegative functions $f$ such that $f,f^{p-1}\in Dom(\calE)$.
	In particular, for all such functions $\mu(f^p)<\infty$.
\end{lemma}

We will also need the following two well-known lemmas.
\begin{lemma}[{\cite[Lemma 14.4]{MR3185193}}]\label{L:BLM-entropic-inequality}
	For any nonnegative $f\in L_p(\calX,\mu)$ and $p\in (1,2]$
 \[
		\mu (f^p) - \mu(f)^p \le \Cov_\mu(f,f^{p-1}).
 \]
\end{lemma}

\begin{lemma}[{\cite[Lemma 2.6]{MR1410112}}]\label{L:ab-inequality}
 For $p\in(1,2]$, $a,b,>0$,
 \[
  (a-b)(a^{p-1}-b^{p-1}) \leq (a^{p/2}-b^{p/2})^2 \leq \frac{p^2}{4(p-1)} (a-b)(a^{p-1}-b^{p-1}).
 \]
\end{lemma}

Another point-wise inequality to be used in the proof of Theorem~\ref{thm:main:mLS_to_Bec} is given in the next lemma.

\begin{lemma}
\label{L:u'v'comp}
	If $a,b\ge e$, then for all $p\ge 1$,
	\[
		(a^p-b^p)(\log a - \log b) \le (a-b)(a^{p-1}\log a-b^{p-1}\log b)\,.
	\]
\end{lemma}
\begin{proof}
	The inequality is equivalent to $ab(a^{p-1}-b^{p-1})(\frac{\log a}{a}- \frac{\log b}{b})\le 0$, which follows since the function
	$\frac{\log x}{x}$ is decreasing for $x\ge e$.
\end{proof}

The last lemma we need  is a simple fact concerning differentiability of the bilinear form. To verify that it holds just under Assumption~\ref{a:I} we provide its complete proof in Appendix~\ref{app-auxiliary-lemmas}.

\begin{lemma}\label{L:v'formula}
	Let $s \ge 1$.
	Assume that $f\in Dom(\calE)$ is bounded and satisfies $0<\inf f$.
	Then $v(s) = \calE(f,f^{s-1})$ is well-defined, differentiable for $s\in(1,\infty)$, right-differentiable at $s=1$ and its derivative is given by the (also well-defined) formula $v'(s) = \calE(f,f^{s-1}\log f)$ for $s\in[1,\infty)$.
\end{lemma}

\subsection{Proof of Theorem~\ref{thm:main:mLS_to_Bec}}

Let us start with a simple proposition, which allows to deduce Beckner's inequality~\eqref{eq:Beckner-GJ} from the modified log-Sobolev inequality~\eqref{eq:mlS-log} with constant $\alpha_p$ degenerating as $p\to 1^+$.
\begin{proposition}
\label{prop:Poinc-to-Bec}
If $\mu$ satisfies satisfies the Poincar\'{e}~\eqref{eq:Poinc} inequality, then for all $p\in(1,2]$ it satisfies Beckner's inequality~\eqref{eq:Beckner-GJ} with constants satisfying the relation
 \[
 	\alpha_p \ge 2\frac{p-1}{p}\lambda.
 \]
 In particular, if the modified log-Sobolev inequality~\eqref{eq:mlS-log} holds, then $\alpha_p \ge \frac{p-1}{p}\rho_0$.
\end{proposition}

\begin{proof}[Proof of Proposition~\ref{prop:Poinc-to-Bec}]
	Fix $p\in(1,2]$ and take any nonnegative $f$ such that $f,f^{p-1}\in Dom(\calE)$.
  By Lemma~\ref{L:Bec=>finite-mean} we may and do assume that $f$ is bounded so that all the expressions below are well-defined.
	By Assumption~\ref{a:I}, the Lipschitz property of the mapping $x^{p-1}\mapsto x^{p/2}$ on the set $[0,\sup f]$ implies that $f^{p/2}\in \Dom(\calE)$.
	We have $\Cov_\mu(f,g) = \frac{1}{2} \int_\calX\int_\calX (f(x)- f(y))(g(x)-g(y))\mu(dx)\mu(dy)$.
	Using Lemma~\ref{L:BLM-entropic-inequality}, Lemma~\ref{L:ab-inequality} and the Poincar\'e inequality~\eqref{eq:Poinc} we see that
	\begin{align*}
		\lambda(\mu (f^p) - \mu(f)^p )
		\le
		\lambda\Cov_\mu(f,f^{p-1})
		\le
		\lambda\Cov_\mu (f^{p/2},f^{p/2})
		\le
		\calE(f^{p/2},f^{p/2}).
	\end{align*}	
	By Assumption~\ref{a:I} and another application of Lemma~\ref{L:ab-inequality} we conclude that
	\[
		\lambda(\mu (f^p) - \mu(f)^p )
		\le
		\frac{p^2}{4(p-1)}\calE(f,f^{p-1}),
	\]
	which ends the proof.
	The second part follows from the fact that the modified log-Sobolev inequality~\eqref{eq:mlS-log} implies the Poincar\'e inequality~\eqref{eq:Poinc}, see Proposition~\ref{prop:mLSI-to-Poinc} in the Appendix~\ref{app-implications}.
\end{proof}

To handle the case of $p \to 1^+$, we will need the following proposition.

\begin{proposition}
\label{P:intf/minf}
Suppose that the modified log-Sobolev inequality~\eqref{eq:mlS-log} holds with some constant $\rho_0>0$.
	Then for any $p\in(1,2]$ and any bounded $f\in \Dom(\calE)$ such that $\inf f>0$,
	\[
		\rho_0(\mu (f^p) - \mu(f)^p)
		\le
		\left(e\frac{\mu(f)}{\inf f}\right)^{p-1}
		\calE(f,f^{p-1}).
	\]	
\end{proposition}

\begin{proof}
Fix any $p\in(1,2]$ and any bounded $f\in \Dom(\calE)$ satisfying $\inf f>0$.
By homogeneity, we may and do assume that $\inf f =e$.

For $s\in[1,p]$, let $u(s)\coloneqq\mu(f^s) - \mu(f)^s$ and $v(s)\coloneqq\calE(f,f^{s-1})$.
For $s\in(1,p)$, Lemma~\ref{L:v'formula} implies that $v(s)$ and $v'(s)= \calE(f,f^{s-1}\log f)$ are well-defined. Recall the variational formula for the entropy
\begin{displaymath}
\Ent_\mu(g) = \sup_{h \in U} \mu(gh),
\end{displaymath}
where $U$ is the family of all measurable functions $h\colon \calX \to \overline{\RR}$, such that $\mu(e^h) = 1$ (see, e.g., \cite[Theorem 4.13]{MR3185193}).

Using this formula with $g = f^s$ and $h = \log \frac{f}{\mu(f)}$, we obtain
	\[
		u'(s) = \mu (f^s\log \tfrac{f}{\mu (f)}) + \log( \mu(f) )u(s)
		\le
		\Ent(f^s) + \log (\mu(f))u(s).
	\]
	Hence, by the modified log-Sobolev inequality~\eqref{eq:mlS-log} and Lemma~\ref{L:u'v'comp} combined with Assumption~\ref{a:I},
	\begin{align*}
		\rho_0 \bigl(u'(s) - \log(\mu(f)) u(s)\bigr)	
		\le \rho_0  \Ent(f^s)	
		&\le s \calE(f^s, \log f)\\
		&\le
		s\calE(f,f^{s-1}\log f) =s v'(s).
	\end{align*}
	Consequently, since $s\mu(f)^{1-s} \leq se^{1-s} \leq 1$, we arrive at
	\begin{equation*}
		\rho_0 \bigl(u(s)\mu(f)^{1-s}\bigr)'
		= \rho_0 \mu(f)^{1-s} \bigl(u'(s) - \log(\mu(f)) u(s)\bigr)	
		\leq 	s\mu(f)^{1-s} v'(s) \leq v'(s).
	\end{equation*}
	Integrating both sides over the interval $[1,p]$ yields the result (recall that $\inf f = e$).
\end{proof}

Having Propositions~\ref{prop:Poinc-to-Bec} and~\ref{P:intf/minf} we can turn to the proof of the main result.

\begin{proof}[Proof of Theorem~\ref{thm:main:mLS_to_Bec}]
Fix any $p\in(1,2]$ and take any bounded nonnegative function $f$ such that $f,f^{p-1}\in Dom(\calE)$.
For $\theta\in (0,1)$ denote $g = \max(f,\theta \mu(f))$ and $P_\theta = \mathbb{P}(f < \theta\mu(f))$.
	Then $\mu(f^p)\le \mu(g^p)$ and, since $\mu(g) \leq (1+\theta P_\theta)\mu(f)$ and $x\mapsto (1+\theta x)^p$ is convex for $x\in [0,1]$,

\begin{equation}\label{eq:change_f_to_g_on_LHS}
\begin{split}		\mu (f^p) -  \mu(f)^p
		&\le
		\mu (g^p) -  \mu(g)^p +
		\mu(f)^{p}\bigl((1+\theta P_{\theta})^p-1\bigr)\\
&\le
		\mu (g^p) -  \mu(g)^p +	
		\mu(f)^{p}P_{\theta}\bigl((1+\theta )^p-1\bigr).
		\end{split}
\end{equation}
	Clearly $\mu(g) / \inf g \le (1+\theta)/\theta$, therefore Proposition~\ref{P:intf/minf} implies that
        \begin{equation}
        \label{eq:mLS_to_Bec_with_defect}
         		\rho_0 (\mu (g^p) -  \mu(g)^p)
		\le
		\left(e\tfrac{1+\theta}{\theta}\right)^{p-1}\calE(g,g^{p-1})
		\le		\left(e\tfrac{1+\theta}{\theta}\right)^{p-1}\calE(f,f^{p-1}),
        \end{equation}
 where we also used the fact that $x\mapsto \max(x,a)$ is a contraction and Assumption~\ref{a:I}.
 Combining~\eqref{eq:change_f_to_g_on_LHS} and~\eqref{eq:mLS_to_Bec_with_defect} yields a defective Beckner inequality:
 \begin{equation}
 \label{eq:mLS_to_Bec_with_defect-2}
  \rho_0(\mu (f^p) -  \mu(f)^p)
		\le
	\rho_0\mu(f)^{p}P_{\theta}\bigl((1+\theta )^p-1\bigr)
+ \left(e\tfrac{1+\theta}{\theta}\right)^{p-1}\calE(f,f^{p-1}).
 \end{equation}
 It remains to deal with the first summand on the right-hand side

	By Taylor's expansion with the integral form of the remainder (and since $p\leq 2$)
	\begin{align*}
		\mu (f^p) - \mu (f)^p	
		&=
		p(p-1)\mu\Bigl( \int_{\mu(f)}^{f}u^{p-2}(f-u)\,du \Bigr)\\
		&\ge
		p(p-1)\mu\Bigl( \ind{\lbrace f<\mu(f) \rbrace} \int_{f}^{\mu (f)}u^{p-2}(u-f)\,du \Bigr)\\
		&\ge
		\frac{p(p-1)}{2} \mu (f)^{p-2} \mu \bigl((\mu(f) - f)_{+}^2\bigr)\,,
	\end{align*}
	while by Chebyshev's inequality
	\[
		P_\theta = \mathbb{P}\bigl((1-\theta)\mu (f) \le (\mu (f) - f)_+ \bigr)
		\le
		\frac{\mu \bigl( (\mu (f)- f)_+^2 \bigr)}{(1-\theta)^2(\mu (f))^2}\,,
	\]
	whence
	\[
		P_\theta \mu(f)^{p} \le \frac{2(\mu (f^p) - \mu(f)^p)}{p(p-1)(1-\theta)^2}\,.
	\]
	Plugging the above estimate into~\eqref{eq:mLS_to_Bec_with_defect-2} and optimizing over $\theta\in(0,1)$ yields
	Beckner's inequality~\eqref{eq:Beckner-GJ} with
	\[
	\alpha_p \geq \rho_0 \cdot \frac{p}{2}\cdot\!\sup_{\theta\in(0,1)} k(p,\theta),
  		\]
  		where we recall that
  		\[
  		 k(p,\theta)  =
  		 \Bigl(1 - \frac{2\bigl((1+ \theta)^{p}-1 \bigr)}{p(p-1)(1-\theta)^2} \Bigr) \cdot
  		\frac{\theta^{p-1}}{e^{p-1}(1+\theta)^{p-1}}.
  		\]
    The extension to not necessarily bounded functions follows by Lemma~\ref{L:Bec=>finite-mean}.
	
	Of course, for some values of $p\in(1,2]$ the bound $\alpha_p\geq (1-1/p)\rho_0$ provided by Proposition~\ref{prop:Poinc-to-Bec} may be better. We shall now compare both expressions to get some more explicit estimates on the multiplicative factor
	\[
	K_p = \max\{ 1-1/p;\ p/2 \cdot\! \sup_{\theta\in(0,1)} k(p,\theta)\}.	
	\]
	
	It is easy to see that $\lim_{p\to 2^-} K_p =1/2$. Since $\lim_{p\to 1^+} k(p,(p-1)^2) = 1$ and obviously $k(p,\theta)\leq 1$, we  conclude that $\lim_{p\to 1^+} K_p =1/2$.
	
	Moreover one can check that
	\begin{align}\label{eq:num-bound}
	K_p\geq \max\{1-1/p;\ p/2 \cdot k(p,0.25 \cdot (1-p)^2) \} \geq 0.17.
	\end{align}
	
	This ends the proof of the theorem.
	\end{proof}
\begin{remark}
Note that the numerical bound~\eqref{eq:num-bound} cannot be substantially improved if we want it to hold for \emph{all} $p\in(1,2]$. Indeed, we have
	\[
	 K_{6/5}
	 \leq 0.18.
	\]
Identification of the best constants $K^\opt_p$ such that $\alpha_p \ge K^\opt_p\rho_0$ seems to be an interesting open question.	
\end{remark}

\section{Moment estimates derived from Beckner's inequalities}

\label{sec:Beckner-Aida-Stroock}
In this section we revisit the arguments by Boucheron et al. \cite[Theorem 15.5]{MR3185193} and present them in the context of general Beckner inequalities. We derive moment inequalities, which are valid in particular under the assumption of modified log-Sobolev inequality. These moment estimates will lie at the core of the applications presented in subsequent sections. In this section we still  work with a probability space $(\calX,\calB,\mu)$. Accordingly all the moments of functions/random variables are calculated with respect to the measure $\mu$, i.e., for $g\colon \calX \to \RR$, we set
$\|g\|_r = (\mu(|g|^r))^{1/r}$.

\begin{proposition}
\label{prop:V^2}
 Assume that $\Gamma_+$ is defined, as in~\eqref{eq:Gamma+}, via some kernel $Q_x$ satisfying the detailed-balance condition~\eqref{eq:detailed-balance}.
 Let $\calE$ be given by~\eqref{eq:kernel-to-form} and assume that for all $p \in (1,2]$ the Beckner inequality~\eqref{eq:Beckner-GJ} is satisfied with constant $\alpha_p \ge a(p-1)^s$ for some $a> 0$, $s\ge0$.
 Then for every measurable $f \colon \calX \to \RR$ and $r\ge 2$,
	\begin{align}
		\norm{(f-\mu(f))_+}_r^2 &\le (1-2^{-(s+1)})\frac{r^{s+1}}{a}\kappa(s) \norm{\Gamma_+(f)}_{r/2},\label{eq:moments-1}\\
		\norm{(\mu(f)-f)_+}_r^2 &\le (1-2^{-(s+1)})\frac{r^{s+1}}{a}\kappa(s) \norm{\Gamma_+(-f)}_{r/2}\label{eq:moments-2},
	\end{align}
	where $\kappa(s)= (1-e^{-(s+1)/2})^{-1}$.
\end{proposition}

The case of $s=0$ corresponds via Theorem~\ref{thm:main:mLS_to_Bec} to the modified log-Sobolev inequality~\eqref{eq:mlS-log}, while the case $s=1$ via Proposition~\ref{prop:Poinc-to-Bec} to the Poincar\'e inequality. In fact, if the inequality~\eqref{eq:Beckner-GJ} holds for \emph{some} $p\in (1,2]$, then also the Poincar\'e inequality holds (see Proposition~\ref{prop:Bec-to-Poinc}) and so one can find $a>0$, such that~\eqref{eq:Beckner-GJ} holds for \emph{all} $p \in (1,2]$ with $\alpha_p \ge a(p-1)$. Thus the interesting range of the parameter $s$ in the above proposition is $[0,1]$.

In most applications that we have in mind, $\Gamma$ will be indeed defined by some kernel. However, similar estimates may be derived also in a more abstract setting, encompassing in particular general reversible Markov semigroups. In Section~\ref{sec:continuous} we will use such a statement to present certain weighted $L_p$ Poincar\'e inequalities. In line with our general approach of writing the inequalities in an abstract form, under structural assumptions, we will formulate the next result in terms of the following additional assumption.

\begin{assumption}\label{a:II}
\leavevmode
\begin{itemize}
\item For any bounded $f \in \calA$, any $c \in \RR$, and any $\gamma > 1$, $t \ge 1$,
\begin{align}\label{eq:assumption-2-1}
  \calE( |f+c|^\gamma, |f+c|) \le 2 \gamma \bigl\| |f+c|^{\gamma-1}\bigr\|_{\frac{t}{t-1}} \bigl\| \Gamma(f) \bigr\|_{t}.
\end{align}
\item For any $f \in \calA$ there exists a sequence $f_n$ of bounded elements of $\calA$, such that $f_n \to f$ and $\Gamma(f_n) \le \Gamma(f)$ $\mu$-a.s.
\end{itemize}
\end{assumption}

\begin{remark}
The first part of the above assumption is satisfied in particular if $\calA$ is any algebra contained in the domain of the infinitesimal operator $L$ of a Markov semigroup reversible with respect to $\mu$. The second part may depend on the choice of $\mathcal{A}$, however in most cases in the theory of Dirichlet forms one chooses $\calA$ which is stable under composition with smooth functions, which allows for appropriate truncations, implying the second part (see, e.g., Definition 3.3.1. of the extended algebra $\calA$ in the monograph~\cite{MR3155209}). We provide derivation of both parts of Assumption~\ref{a:II}  in this standard Markovian setting in Proposition~\ref{prop:assumption-2} in the Appendix~\ref{app-Markov}.
\end{remark}

\begin{proposition}\label{prop:moments-abstract}
Let $\calE\colon \Dom(\calE)\times \Dom(\calE) \to \RR$ be a nonnegative definite symmetric bilinear form and let $\Gamma\colon \calA\times \calA \to L_0(\calX,\mu)$, where $\calA \subseteq \Dom(\calE)$ is a linear subspace, be a bilinear form related to $\calE$ by~\eqref{eq:champ-form}.
If Assumptions~\ref{a:I}  and~\ref{a:II} are satisfied and for all $p \in (1,2]$ the Beckner inequality~\eqref{eq:Beckner-GJ} holds with $\alpha_p \ge a(p-1)^s$ for some $a>0$, $s\ge 0$, then for all $f\in \mathcal{A}$ and $r\ge 2$,	
	\begin{align}\label{eq:moments-4}
		\norm{f-\mu(f)}_r^2 &\le \frac{r^{s+1}\kappa(s)}{a} \norm{\Gamma(f)}_{r/2},
	\end{align}
where $\kappa(s)$ is as in Proposition~\ref{prop:V^2}.
\end{proposition}

\begin{remark}We remark that the inequalities of Propositions~\ref{prop:V^2} and~\ref{prop:moments-abstract} should be again understood in the following sense: \emph{if the right-hand side is finite, then the left-hand side is well-defined and the inequality holds}. Let us also mention that the inequalities of Proposition~\ref{prop:moments-abstract} can be extended beyond the space $\calA$, if one replaces the right-hand side via a family of norms extending the moments of $\sqrt{\Gamma(f)}$ and defined by appropriate duality. We will not pursue this direction and refer to the article~\cite{MR1258492} by Aida--Stroock where similar moment estimates were proved under the stronger assumption of the log-Sobolev inequality~\eqref{eq:mlS-sqrts}. The inequalities derived by Aida--Strock from~\eqref{eq:mlS-sqrts}, in our setting read as
\begin{align}\label{eq:Aida-Stroock}
||f - \mu(f)\|_r^2 \le \rho_1^{-1} (r-3/2) \|\Gamma(f)\|_{r/2}.
\end{align}

We remark that the derivation of moment estimates from the log-Sobolev inequality by Aida and Stroock is based on computing the derivative of $\|f\|_t^2$ with respect to $t\in [2,r]$, and identification of a part corresponding to $\Ent |f|^t$, which can be estimated via~\eqref{eq:mlS-sqrts}. Such an estimation allows for convenient cancelations and yields a uniform bound on the derivative on the interval $[2,r]$. This approach has been subsequently used, e.g., in \cite{MR2510011,MR3383337,MR3743923} in the context of weighted log-Sobolev inequalities or various modified log-Sobolev inequalities on $\RR^n$ (of different nature than~\eqref{eq:mlS-log}). It does not seem however that this approach can work with~\eqref{eq:mlS-log}. Theorem~\ref{thm:main:mLS_to_Bec} allows to pass  from~\eqref{eq:mlS-log} to~\eqref{eq:Beckner-GJ} and use the argument introduced by Boucheron, Bousquet, Lugosi, and Massart for product measures.

It is also known (see, e.g., Proposition 2.5. in~\cite{MR2507637} for the proof in the case $\Gamma(f) = |\nabla f|^2$), that the Poincar\'e inequality implies moment estimates of the form
\begin{align}\label{eq:Poincare-moments}
\|f - \mu(f)\|_{r}^2 \le \frac{C}{\lambda} r^2 \|\Gamma(f)\|_{r/2}
\end{align}
for $r \ge 2$. This corresponds to the case $s=1$ in Proposition~\ref{prop:moments-abstract}.
\end{remark}

\begin{remark} It is easy to check that Assumption~\ref{a:II}  is verified in the setting of Proposition~\ref{prop:V^2}. In fact the moment inequality of Proposition~\ref{prop:moments-abstract} provides better constants than one would obtain by combining the two estimates of Proposition~\ref{prop:V^2} and pointwise estimates $\Gamma_+(f),\Gamma_+(-f)\le 2 \Gamma(f)$.
\end{remark}

\begin{remark} \label{re:uwaga-o-momentach}
An inspection of the proofs of Propositions~\ref{prop:V^2} and~\ref{prop:moments-abstract} shows that if one assumes that the inequality~\eqref{eq:Beckner-GJ} holds just for $p \in [p_0,2]$ for some $p_0 > 1$, then the moment estimates will still hold, but for $2 \le r \le r_0 = \frac{p_0}{p_0-1}$. We will use this observation in Section~\ref{sec:continuous}.
\end{remark}

\begin{proof}[Proof of Proposition~\ref{prop:V^2}]

Let us start with the inequality~\eqref{eq:moments-1} and consider the case of bounded functions $f \in \mathcal{A}$.

We will show by induction a slightly stronger statement, namely that for all positive integers $k$ and $r\in (k,k+1]$
\begin{equation*}\label{eq:AS_pf_induction}
	\norm{(f-\mu (f))_+}_{r}^2 \le c_r \norm{\Gamma_+(f)}_{\max(r/2,1)},
\end{equation*}
where
\begin{align}\label{eq:kappa-and-c}
	c_r = \frac{1}{a}\max\big(\frac{\kappa_r(s)r^{s+1}}{\kappa_2(s)};1\big),
	\quad
	\kappa_r(s) = \Big( 1 - \left( \frac{r-1}{r} \right)^{(s+1)r/2}\Big)^{-1} \nearrow \kappa(s)
\end{align}
as $r\to\infty$.

In what follows the parameter $r$ will change while $s$ will remains fixed, so to simplify the notation we will suppress the dependence of $\kappa_r(s)$ on $s$ and write simply $\kappa_r$.

For $k=1$ and any $r\in (1,2]$,
\begin{align}\label{eq:application-of-Poincare}
	\norm{(f-\mu (f))_+}_r^2
	\le
	\norm{f-\mu (f)}_{r}^2
	\le
	\norm{f-\mu (f)}_{2}^2
	\le
	\frac{1}{a}\calE(f,f)
	\le
	c_r\norm{\Gamma_+(f)}_{1},
\end{align}
where in the second step we used Jensen's inequality, in the third one the Poincar\'{e} inequality (which holds if~\eqref{eq:Beckner-GJ} holds, see Proposition~\ref{prop:Bec-to-Poinc} in the Appendix), while the last one follows from~\eqref{eq:form-reversibility} and $c_r \ge c_1 = 1/a$. This yields the induction basis.

Assume that the induction hypothesis holds for all integers smaller than some $k>1$. Consider any $r\in(k,k+1]$ and a bounded function $f\colon \mathcal{X} \to \RR$.
Choose $p$ such that $r=\frac{p}{p-1}$ and denote $\gamma=\frac{1}{p-1}$, $g=(f-\mu(f))_+$.
Applying the Beckner inequality~\eqref{eq:Beckner-GJ} to the function $g^{\gamma}$ and using the form~\eqref{eq:bilinear-form-reversibility} of $\calE$ (which is a consequence of the detailed balance condition~\eqref{eq:detailed-balance}), together with the convexity of $x\mapsto x^\gamma$, we get
\begin{align}\label{eq:convexity}
	\alpha_p(\mu(g^{\gamma p})	 - \mu(g^\gamma)^p) &\le
	\frac{p}{2}\int (g^\gamma(x)-g^\gamma(y))_+(g(x)-g(y))_+\,Q_x(dy)\mu(dx) \nonumber \\
	&\le \frac{\gamma p}{2}\int g^{\gamma - 1}(x)\Gamma_+(g)\mu(dx).
\end{align}

 Since $(g(x)-g(y))_+\le (f(x)-f(y))_+$, we have $\Gamma_+(g)\le\Gamma_+(f)$, and so by H\"{o}lder's inequality with exponents $\frac{\gamma p}{\gamma p -2}$ and $\frac{\gamma p}{2}$ (recall that $\gamma + 1 = \gamma p=r$), we obtain
\begin{align}\label{eq:po-Holderze}
	\alpha_p (\mu (g^{\gamma p}) - \mu (g^{\gamma})^{p})
	&\le
	\frac{\gamma p}{2} ( \mu(g^{\gamma p}))^{\frac{\gamma p - 2}{\gamma p}}\norm{\Gamma_+(f)}_{\gamma p /2}.
\end{align}
 Denoting $l_r=\norm{g}_r$ and observing that $\alpha_p \ge a/r^s$, the above inequality divided by $\alpha_p$ and restated in terms of $r$ gives
\begin{align*}
	l_{r}^r &\le
	l_{r-1}^{r} +
	l_{r}^{r-2} \frac{r^{s+1}}{2a}\norm{\Gamma_+(f)}_{r/2}.
\end{align*}
 The induction hypothesis allows us to estimate $l_{r-1}$:
\begin{equation}\label{eq:lq-induction-step}
	l_{r}^r \le
	\left(c_{r-1}\norm{\Gamma_+(f)}_{\max((r-1)/2,1)}\right)^{r/2}
	+
	l_{r}^{r-2}\frac{r^{s+1}}{2a}\norm{\Gamma_+(f)}_{r/2}.
\end{equation}

Note that we can assume that $\|\Gamma_+(f)\|_{r/2} > 0$, since otherwise (as $r> 2$) we obtain $\|\Gamma_+(f)\|_{1} = 0$ and by the induction assumption $(f - \mu(f))_+ = 0$.
Recall that $c_r = \max\big(\kappa_rr^{s+1}/\kappa_2;1\big)/a$ and thus, by the monotonicity in $u$ of $\kappa_u$ and $\norm{\Gamma_+(f)}_{\max(u/2,1)}$, and since $r>2$,
\[
	\frac{c_{r-1}\norm{\Gamma_+(f)}_{\max((r-1)/2,1)}}{c_r\norm{\Gamma_+(f)}_{r/2}} \le
	\frac{c_{r-1}}{c_{r}} \le
	\left(\frac{r-1}{r}\right)^{s+1}.
\]
Consequently, dividing~\eqref{eq:lq-induction-step} by $(c_r\norm{\Gamma_+(f)}_{r/2})^{r/2}$, leads to
\begin{align}\label{eq:division}
	\left(\frac{l_r^2}{c_r\norm{\Gamma_+(f)}_{r/2}}\right)^{r/2}
	\le
	\left(\frac{r-1}{r}\right)^{(s+1)r/2}
	+
	\frac{\kappa_2}{2\kappa_r}\left(\frac{l_r^2}{c_r\norm{\Gamma_+(f)}_{r/2}}\right)^{(r-2)/2}.
\end{align}
The function
\[
	h(x) =
	\left(\frac{r-1}{r}\right)^{(1+s)r/2}
	+
	\frac{1}{\kappa_r} x^{1-2/r} - x
\]
is strictly concave on $[0,\infty)$, positive at $x=0$ and $h(1)=0$ (by the definition of $\kappa_r$).
As a consequence, $h(x)\geq 0$ implies $x\le 1$, whence (note that $\kappa_2/2\le 1$)
\[
	l_r^2
	\le
	c_r\norm{\Gamma_+(f)}_{r/2}
\]
which proves the induction step and demonstrates~\eqref{eq:moments-1} for bounded functions $f$.

Let us now remove the boundedness assumption. If $f \colon \mathcal{X}\to \RR$ is an arbitrary function with $\|\Gamma_+ (f)\|_r < \infty$ , then $\mathcal{E}(f,f) = \|\Gamma_+(f)\|_1 < \infty$ and as a consequence by the Poincar\'e inequality, we obtain $\mu(|f|) < \infty$ and
\begin{displaymath}
  \Var_\mu(f) \le \frac{1}{a} \mathcal{E}(f,f).
\end{displaymath}
In particular, defining $f_M = \max (\min(f,M),-M)$ we obtain $f_M \to f$ pointwise and $\mu( f_M) \to \mu (f)$ as $M\to \infty$. Applying~\eqref{eq:moments-1} to $f_M$ we obtain
\begin{align}\label{eq:bounded-case}
  \|(f_M - \mu(f_M))_+\|_r^2 \le (1-2^{-(s+1)})\frac{r^{s+1}}{a}\kappa(s)\|\Gamma_+(f_M)\|_{r/2}.
\end{align}
However
\begin{align*}
  \Gamma_+(f_M)(x) &= \int_\mathcal{X}(f_M(x) - f_M(y))_+^2Q_x(dy) \\
 &\le \int_\mathcal{X}(f(x) - f(y))_+^2Q_x(dy) = \Gamma_+(f)(x).
\end{align*}
Therefore Fatou's lemma implies that~\eqref{eq:moments-1} for $f$ follows from~\eqref{eq:bounded-case} by letting $M \to \infty$.

\medskip

The  inequality~\eqref{eq:moments-2} follows by~\eqref{eq:moments-1} applied to $-f$.
\end{proof}

\begin{proof}[Proof of Proposition~\ref{prop:moments-abstract}]
The general scheme of the proof is analogous as in the case of Proposition~\ref{prop:V^2}, one just needs to appropriately replace the pointwise estimates with the kernel $Q_x$ by the abstract assumptions. Therefore instead of writing the complete proof we will just explain how to modify the arguments leading to~\eqref{eq:moments-1}.

We again prove by induction that for all positive integers $k$, $r \in (k,k+1]$, and bounded $f \in \mathcal{A}$,
\begin{displaymath}
  \|f - \mu(f)\|_r^2 \le c_r \|\Gamma(f)\|_{\max(r/2,1)}
\end{displaymath}
with $c_r = \kappa_r(s)r^{s+1}/a\ge 1/a$. The quantity $\kappa_r(s)$ is defined as in~\eqref{eq:kappa-and-c}, note however the difference between the definition of $c_r$ in this proof and therein.

For $k = 1$, this follows analogously as in~\eqref{eq:application-of-Poincare}, by ignoring the first inequality and using $\calE(f,f) = \|\Gamma(f)\|_1$ in the last estimate (note that finiteness of $\|\Gamma(f)\|_1$ implies that $\mu(f)$ is well-defined).

As for the induction step, we consider $g = f - \mu(f)$ and $\gamma = \frac{1}{p-1}$ where $r = \frac{p}{p-1}$. By Assumption~\ref{a:I}, $g, |g| \in \Dom(\calE)$.

Assume that $\|g\|_\infty = M$ and observe that for $a,b \in [-M,M]$
\begin{align*}
\bigl| |a|^\gamma - |b|^\gamma\bigr| \le \gamma M^{\gamma-1}|a-b|.
\end{align*}
Therefore, again by Assumption~\ref{a:I}, $|g|^\gamma \in \Dom(\calE)$.

Applying thus~\eqref{eq:Beckner-GJ} with parameter $p$ to $|g|^\gamma$ we obtain
\begin{align}\label{eq:beckner-application}
	\alpha_p\bigl(\mu(|g|^{\gamma p})	 - \mu(|g|^\gamma)^p\bigr)\le \frac{p}{2}\mathcal{E}(|g|^\gamma,|g|).
\end{align}

Now, by the first part of Assumption~\ref{a:II} applied with $t = \gamma p/2 = r/2$ together with the equality $\frac{t}{t-1} = \gamma p/(\gamma p - 2) = \gamma p/(\gamma -1)$, we get
\begin{displaymath}
  \alpha_p\bigl(\mu(|g|^{\gamma p})	 - \mu(|g|^\gamma)^p\bigr)\le \gamma p ( \mu(|g|^{\gamma p}))^{\frac{\gamma p -2}{\gamma p}}\|\Gamma(f)\|_{\gamma p/2}.
\end{displaymath}
The last inequality is a direct analogon of~\eqref{eq:po-Holderze}, the difference being just the lack of the factor $1/2$ on the right-hand side.

The rest of the induction step is the same as in the proof of Proposition~\ref{prop:V^2}, leading to
\begin{displaymath}
  \|f - \mu(f)\|_r \le c'_r\|\Gamma(f)\|_{r/2}
\end{displaymath}
for bounded $f \in \mathcal{A}$, the only difference being the lack of the factor $\kappa_2/2$ in the counterpart of~\eqref{eq:division}, which in the proof of Proposition~\ref{prop:V^2} was estimated from above by one.

The extension to general $f \in \calA$ follows easily by approximation from the second part of Assumption~\ref{a:II}.
\end{proof}

\section{Applications}\label{sec:applications}

We will now present applications of our results to various stochastic models in which modified log-Sobolev inequalities or Beckner inequalities are proven. Our main goal is to obtain new moment inequalities and derive from them concentration etimates.

\subsection{The continuous setting}\label{sec:continuous}

As already mentioned in the Introduction, in the diffusive case, when the chain rule is satisfied, there is equivalence between the modified log-Sobolev inequality~\eqref{eq:mlS-log} and the usual log-Sobolev inequality~\eqref{eq:mlS-sqrts} as well as between the two forms~\eqref{eq:Beckner-GJ} and~\eqref{eq:Beckner-LO} of Beckner's inequality.
Therefore, as explained in Section~\ref{sec:state-of-art} the equivalence between the log-Sobolev inequalities and Beckner inequalities has been known in this case.
Nevertheless, the equivalence of~\eqref{eq:Beckner-GJ} and~\eqref{eq:Beckner-LO} as well as some known examples of measures satisfying~\eqref{eq:Beckner-LO} allow us to obtain moment estimates in $L_r$ with optimal rate of dependence on $r$ as $r\to \infty$ in several situations of interest.

We will start with a result proved by Wang \cite[Corollary 1.3]{MR2127729} in the setting of Riemannian manifolds.

\begin{proposition}
\label{prop:Wang}
Let $E$ be a $d$-dimensional non-compact connected complete
Riemannian manifold with Ricci curvature bounded from below. Let $\rho(x)$ be the Riemannian distance between $x$ and a fixed point $o$. Consider $\mu(dx) := Ze^{V} dx$,
where $V$ is a continuous function on  $E$ such that $V + \theta \rho^\gamma$
is bounded for some $\gamma \in (1, 2]$ and $\theta >0$, $dx$ stands for the Riemannian volume measure, and $Z$ is the normalization. Let
$\mathcal{E}(f,f) := \mu(|\nabla f|^2)$ with $D(\mathcal{E}) = H^{1,2}(\mu)$. Then there exists $\beta > 0$ such that~\eqref{eq:Beckner-LO} holds for all $q \in [1,2)$ with
$\beta_q \ge \beta(2-q)^{2/\gamma - 1}$.
\end{proposition}

As a consequence, by Proposition~\ref{prop:moments-abstract} applied to $\calA$ being the class of  smooth compactly supported functions, followed by standard approximation techniques, we obtain

\begin{corollary}
\label{cor:moments-Wang}
In the setting of Proposition~\ref{prop:Wang}, there exists a constant $C$, depending only on $\beta$, such that for any smooth function $f\colon M\to \RR$ and all $r \ge 2$,
\begin{align}\label{eq:moments-Wang}
  \|f - \mu (f)\|_r \le Cr^{1/\gamma}\|\nabla f\|_r,
\end{align}
where the norms are taken in $L_r(E,\mu)$.
\end{corollary}

We remark that the example of measures $\mu_\gamma$ ($\gamma\in [1,2]$) on $\RR^d$ with density $c_\gamma^d\exp(-\sum_{i=1}^d |x_i|^\gamma)$ (as investigated by Lata{\l}a and Oleszkiewicz in~\cite{MR1796718}, who proved that in this case $\beta>0 $ can be taken to be a universal constant) shows that the exponent $1/\gamma$ in the above corollary is optimal. We also note that in the case of $\RR^d$, moment estimates of the form~\eqref{eq:moments-Wang} for $\gamma \in (1,2)$ can be derived from a combination of recent result~\cite{barthe2019functional} and~\cite{MR3383337} (the case $\gamma=2$ corresponds to results by Aida--Stroock, the case $\gamma=1$ can be found in~\cite{MR2507637}).
The former article establishes an implication between Beckner inequalities~\eqref{eq:Beckner-LO} and certain log-Sobolev inequalities with modified energy form (introduced by Gentil et al. in~\cite{MR2198019}), which are shown in~\cite{MR3383337} to imply~\eqref{eq:moments-Wang}.
However in both of the said implications, additional dependence on $\gamma$ is introduced and the constants explode for $\gamma\to 1$.
To the best of our knowledge, the above corollary is new even in the case of measures $\mu_\gamma$, $\gamma\in (1,2)$.

Another example we would like to discuss concerns weighted inequalities for heavy tailed measures. We will focus on the Cauchy measure, defined on $\RR^n$ as
\begin{displaymath}
  \nu_{n,b}(dx) = \frac{1}{Z(1+|x|^2)^b}
\end{displaymath}
for $b > n/2$ (where $Z$ is a normalizing constant).

Being heavy-tailed, the measure $\nu_{n,b}$ cannot satisfy the usual functional inequalities of Definition~\ref{defi:inequalities} with $\calE(f,g) = \mu(\langle \nabla f,\nabla g\rangle)$. Nevertheless as shown in \cite{SchefferPhD,MR2320246,MR2726546,MR3132733}, for $b \ge n+1$, $\nu_{n,b}$ satisfies the following weighted Poincar\'e inequality for smooth functions:
\begin{displaymath}
\Var_{\nu_{n,b}}(f) \le \frac{1}{2(b-1)}\int_{\RR^n} |\nabla f(x)|^2(1+|x|^2)\nu_{n,b}(dx).
\end{displaymath}
Moreover the weight $1+|x|^2$ is of optimal growth at infinity and the constant $\frac{1}{2(b-1)}$ is optimal.
We remark that the weighted Poincar\'e inequality (without optimal constant) is known in a larger range of parameters (see, e.g.,~\cite{MR2510011}), in what follows we however restrict to $b \ge n+1$, as we are going to use recent Beckner inequalities obtained under this assumption.

In~\cite{MR3008255} the above inequality has been complemented by a weighted log-Sobolev inequality
\begin{displaymath}
  \Ent_{\nu_{n,b}} (f^2) \le C_{n,b}\int_{\RR^n} |\nabla f(x)|^2(1+|x|^2)\log(e+|x|^2)\nu_{n,b}(dx),
\end{displaymath}
where $C_{n,b}$ is a constant, depending only on $n,b$. Again, the growth of the weight is optimal at infinity (the result was earlier proved with a weight of faster growth in~\cite{MR2510011}).

By known approaches to moment estimates, related to~\eqref{eq:Poincare-moments} and~\eqref{eq:Aida-Stroock} the above results provide for $r \ge 2$ bounds of the form
\begin{align}\label{eq:moments-Cauchy-1}
\|f - \nu_{n,b}(f)\|_{L_r(\nu_{n,b})} \le \frac{C}{\sqrt{b-1}}r\| \sqrt{\omega_1}\nabla f\|_{L_r(\nu_{n,b})},
\end{align}
where $\omega_1(x) = 1+|x|^2$, and $C$ is a universal constant, and
\begin{align}\label{eq:moments-Cauchy-2}
  \|f - \nu_{n,b}(f)\|_{L_r(\nu_{n,b})} \le \sqrt{C_{n,b}(r-3/2)}\|\sqrt{\omega_2}\nabla f\|_{L_r(\nu_{n,b})},
\end{align}
where $\omega_2(x) = (1+|x|^2)\log(e+|x|^2)$. See \cite{MR2510011,MR3008255} where similar moment inequalities were considered for Lipschitz functions. It is easy to see that~\eqref{eq:moments-Cauchy-1} and~\eqref{eq:moments-Cauchy-2} are not comparable. The latter has better dependence on $r$, the former may perform better if the function is supported far from the origin.

Recently, Bakry, Gentil, and Scheffer~\cite{MR4048609} proved that for $q \in [1,2- \frac{2}{b-n+1}]$, the measure $\nu_{n,b}$ satisfies a weighted Beckner inequality
\begin{align}\label{eq:weighted-Beckner}
2(b-1)\Big(\nu_{n,b}(f^2) - \nu_{n,b}(f^q)^{2/q}\Big) \le (2-q)\int_{\RR^n} |\nabla f(x)|^2(1+|x|^2) \nu_{n,b}(dx).
\end{align}
Interpreting $\calE(f,g) = \int_{\RR^n} \langle \nabla f,\nabla g\rangle \omega_1d\nu_{n,b}$ as a Dirichlet form related to the diffusion with generator $L f = \omega_1 \Delta f +\langle \nabla \omega_1 - \omega_1\nabla V,\nabla f\rangle$, where $V = - \log(\frac{\nu_{n,b}(dx)}{dx})$, and using the relation between the inequalities~\eqref{eq:Beckner-LO} and~\eqref{eq:Beckner-GJ} discussed in Section~\ref{sec:state-of-art} we see that for all $p \in [1+\frac{1}{b-n},2]$,
\begin{displaymath}
  2(b-1)(\nu_{n,b}(f^p) - \nu_{n,b}(f)^p) \le \calE(f,f^{p-1}).
\end{displaymath}

Note that this inequality cannot be satisfied for all $p \in (1,2]$ with a uniform constant, since this would contradict the optimal growth of weight $w_2$ for the log-Sobolev inequality. Using thus Remark~\ref{re:uwaga-o-momentach} with $\Gamma(f,g) = \langle \nabla f,\nabla g\rangle \omega_1$, we obtain the following
\begin{corollary}
 If $b \ge n+1$ then for any smooth function $f\colon \RR^n\to \RR$, and $r \in [2,b-n+1]$,
\begin{align}\label{eq:moments-Cauchy-3}
\|f - \nu_{n,b}(f)\|_{L_r(\nu_{n,b})} \le \frac{C}{\sqrt{(b-1)}}\sqrt{r}\|\sqrt{\omega_1} \nabla f\|_{L_r(\nu_{n,b})},
\end{align}
where $C$ is a universal constant.
\end{corollary}

The above corollary complements the inequalities~\eqref{eq:moments-Cauchy-1} and~\eqref{eq:moments-Cauchy-2}, improving on some of their aspects in the situation when $b$ is substantially larger than $n$ and $r$ is large, as it provides  better dependence on $r$ than~\eqref{eq:moments-Cauchy-1} and at the same is based on the weight $\omega_1$ which is smaller than $\omega_2$ used in~\eqref{eq:moments-Cauchy-2}.
However, in the case of fixed $b$ the range of $r$ for which the estimate holds is restricted.  We remark that weighted Beckner inequalities for more general heavy tailed convex measures have been recently obtained in~\cite{MR4003575}. They have been also generalized to the manifold setting in~\cite{gentil2019family}. In all these cases one can derive similar weighted moment inequalities, we chose the case of the Cauchy measure to simplify the exposition.

\subsection{Product spaces} \label{sec:product-spaces}

The Herbst argument, which is now the main tool for deriving concentration estimates from log-Sobolev type inequalities, appeared for the first time in the 1970s in an unpublished letter from I. Herbst to L. Gross. In the mid 1990s in the seminal  paper~\cite{MR1399224} Ledoux demonstrated the strength of this argument in the context of concentration inequalities in product spaces, recovering many inequalities by Talagrand, obtained by a more difficult inductive approach based on appropriate notions of isoperimetry \cite{MR1419006,MR1361756}. Since then the method was further developed by many authors, most notably by Boucheron, Bousquet, Lugosi, and Massart.
 Massart~\cite{MR1782276} and  Boucheron, Lugosi, and Massart~\cite{MR1749290,MR1989444} developed many modified log-Sobolev inequalities for product spaces, which were applied to a variety of problems, ranging from information theory through combinatorics to statistics and probability in Banach spaces. In a subsequent paper with Bousquet~\cite{MR2123200} they also established moment estimates, which are a direct inspiration for our Proposition~\ref{prop:V^2}. For this purpose they developed  Beckner inequalities of the form~\eqref{eq:Beckner-GJ} in product spaces, by deriving first tensorization properties for $\phi$-entropies (present also in the work~\cite{MR1796718} by Lata{\l}a and Oleszkiewicz) and then establishing one dimensional cases of~\eqref{eq:Beckner-GJ} (thus proceeding in a manner parallel to the derivation of the modified log-Sobolev inequality in product spaces, based on tensorization properties of the usual entropy functional).

Our contribution in the context of product measures is an observation that thanks to the equivalence of~\eqref{eq:Beckner-GJ} and~\eqref{eq:mlS-log} with a mild change in constants, the Beckner's inequalities obtained in~\cite{MR2123200} can be derived directly from the most basic modified log-Sobolev inequality for product distributions. We would like to stress again that the subsequent derivation of moment inequalities that we present in Proposition~\ref{prop:V^2} relies very heavily on the approach from~\cite{MR2123200}.

For reader's information and for comparison with the more general case of Glauber dynamics, discussed in the next section, we will now state some of the  moment inequalities presented in~\cite{MR2123200} (we remark that this paper provides also other moment inequalities obtained under additional assumptions on the random variables in question).

\begin{theorem}[{\cite[Theorem 2]{MR2123200}}]
\label{thm:BBLM}
Let $(X_1,\ldots,X_n)$ be independent random variables with values in a measurable space $E$ and let $Z = f(X_1,\ldots,X_n)$ for some measurable function $f$. Let moreover $X_1',\ldots,X_n'$ be independent copies of $X_1,\ldots,X_n$
and set $Z_i = f(X_1,\ldots,X_{i-1},X_i',X_{i+1},\ldots,X_n)$.

Then for $r \ge 2$,
\begin{align}\label{eq:BBLM}
  \|(Z - \EE Z)_+\|_r \le \sqrt{\kappa r}\|\sqrt{V_+}\|_r,
\end{align}
where
\begin{displaymath}
  V_+ = \EE\Big(\sum_{i=1}^n (Z - Z_i)_+^2\Big|X_1,\ldots,X_n\Big)
\end{displaymath}
and $\kappa = \frac{\sqrt{e}}{\sqrt{e} - 1}$.
\end{theorem}

We remark that for $p = 2$ this result recovers (up to constants) the Efron--Stein inequality for the variance.

Let us now relate the above theorem to our Theorem~\ref{thm:main:mLS_to_Bec} and Proposition~\ref{prop:V^2} and explain how they imply a version of the estimate~\eqref{eq:BBLM}. Denote by $\mu$  the distribution of the sequence $(X_1,\ldots,X_n)$ and observe that the quantity $V_+$ coincides with our $\Gamma_+(f)$ (recall~\eqref{eq:Gamma+}) for $\Gamma$ given by
\begin{displaymath}
  \Gamma(f,g) = \frac{1}{2}\int_{E^n} (f(y) - f(x))(g(y) - g(x))Q_x(dy),
\end{displaymath}
for the kernel
\begin{displaymath}
Q_x(A) = \int_{E^n}\sum_{i=1}^n \ind{A}(x_1,\ldots,x_{i-1},y_i,x_{i+1},\ldots,x_n) \mu(dy),
\end{displaymath}
reversible with respect to $\mu$ (as already mentioned this can be seen as a special case of Glauber dynamics).

The modified log-Sobolev inequality~\eqref{eq:mlS-log} holds in this case simply due to  Jensen's inequality and tensorization (the idea present already in the paper~\cite{MR1399224} by Ledoux).
Indeed, for any random variable $X$, denoting by $X'$ its independent copy, we have
\begin{align*}
  \Ent(f(X)) & \le \EE f(X)\log f(X) - \EE f(X)\EE\log f(X) \\
  & = \EE f(X)(\log f(X) - \log f(X')) \\
  &= \frac{1}{2} \EE (f(X) - f(X'))(\log f(X) - \log f(X')),
\end{align*}
which, when combined with the well known tensorization property of entropy (see, e.g. \cite[Theorem 4.10]{MR3185193})
\begin{displaymath}
  \Ent_{\mu}(f) \le \int  \sum_{i=1}^n \Ent_{\mu_i}(f) d\mu,
\end{displaymath}

for $\mu = \mu_1\otimes\cdots\otimes \mu_n$ gives~\eqref{eq:mlS-log} with $\rho_0 = 1$ (here $\Ent_{\mu_i}$ denotes the entropy computed on a product space just with respect to the $i$-th coordinate and the measure $\mu_i$, with the other coordinates fixed). By Theorem~\ref{thm:main:mLS_to_Bec} this gives Beckner's inequality~\eqref{eq:Beckner-GJ} with $\alpha_p \ge \frac{1}{6}$.
Now, Proposition~\ref{prop:V^2} applied with $a = 1/6$ and $s = 0$ gives~\eqref{eq:BBLM} with $\kappa = 3\frac{\sqrt{e}}{\sqrt{e}-1}$, which is worse than Theorem~\ref{thm:BBLM} just by a factor~$\sqrt{3}$.

\subsection{Glauber dynamics} \label{sec:Glauber}

Let us now consider $\mathcal{X} = E^I$, where $I$ is a finite set and $E$ is a Polish space endowed with the Borel $\sigma$-field. Let $\mu$ be  a probability measure on $\mathcal{X}$. For $x \in \mathcal{X}$ and $J \subseteq I$, let $x_J = (x_i)_{i\in J}$. Let also $X=(X_i)_{i\in I}$ be an $\mathcal{X}$-valued random variable distributed according to $\mu$. Finally for $i\in I$ let $X_i'$ be an $E$-valued random variable such that its (regular) conditional distribution given $X$ satisfies
\begin{displaymath}
  \mu_i(\cdot|x) := \PP(X_i' \in \cdot \ |X = x) = \PP(X_i\in \cdot \ |X_{\{i\}^c} = x_{\{i\}^c}).
\end{displaymath}
In other words, $X_i'$ and $X_i$ are conditionally i.i.d. given $X_1,\ldots,X_{i-1},X_{i+1},\ldots,X_n$.

Denote $X^i = (Y_j)_{j\in I}$ where $Y_j = X_j$ for $j\neq i$ and $Y_i= X_i'$ (i.e., $X^i$ is obtained from $X$ by replacing $X_i$ with $X_i'$). The Glauber dynamics (known also as the Gibbs sampler or heath bath) is given by a generator of the form
\begin{displaymath}
  L f(x) = \sum_{i \in I} \int_{E} (f(x_1,\ldots,x_{i-1},y,x_{i+1},\ldots,x_n) - f(x)) \mu_i(dy|x)
\end{displaymath}
and corresponds to a c\`adl\`ag Markov process $(X(t))_{t\ge 0}$ in which at rate $|I|$ a coordinate $i\in I$ is chosen uniformly and  $X_i(t-)$ is replaced with a value drawn from the distribution  $\mu_i(\cdot|X(t-))$, while the remaining coordinates are kept intact.

Let us note that if $\mu$ is a product measure, then $\mu_i(\cdot|x)$ equals to the $i$-th marginal of $\mu$ (in particular is independent of $x$) and the situation reduces to the case described in the previous section with $I=[n]:=\{1,\ldots,n\}$. In the general case the generator and the carr\'e du champ operator are given by the kernel
\begin{align}\label{eq:Glauber-Q}
Q_x(A) = \int_{\mathcal{X}}\sum_{i\in I} \ind{A}(x_1,\ldots,x_{i-1},y_i,x_{i+1},\ldots,x_n)\mu_i(dy|x).
\end{align}

Plugging this kernel into formulas~\eqref{eq:Gamma-kernel} and~\eqref{eq:Gamma+}, and using the properties of the variables $X^i$, we obtain
\begin{align*}
  \Gamma(f) &= \frac{1}{2}\sum_{i\in I} \EE((f(X) - f(X^i))^2|X),\\
\Gamma_+(f) &= \sum_{i\in I} \EE((f(X) - f(X^i))^2_+|X).
\end{align*}
Therefore, a combination of Theorem~\ref{thm:main:mLS_to_Bec}, Proposition~\ref{prop:V^2} and Proposition~\ref{prop:moments-abstract} gives in this setting the following corollary.

\begin{corollary}\label{cor:Glauber}
In the setting described above, if the Glauber dynamics satisfies the modified log-Sobolev inequality~\eqref{eq:mlS-log} with constant $\rho_0$, then for $p \in (1,2]$ it satisfies~\eqref{eq:Beckner-GJ} with $\alpha_p \ge \rho_0/6$. Moreover for every function $f\colon \mathcal{X}\to \RR$ and $r\ge 2$,
\begin{align*}
  \|(f(X) - \EE f(X))_+\|_r  &\le K\sqrt{r}\Big \| \Big(\sum_{i\in I} \EE((f(X) - f(X^i))^2_+|X)\Big)^{1/2}\Big\|_r\\
 &\le K\sqrt{r}\Big \| \Big( \sum_{i\in I} (f(X) - f(X^i))^2_+\Big)^{1/2}\Big\|_r
\end{align*}
and
\begin{align*}
\|f(X) - \EE f(X)\|_r
&\le K\sqrt{r}\Big \| \Big(\sum_{i\in I} \EE((f(X) - f(X^i))^2|X)\Big)^{1/2}\Big\|_r\\
&\le K\sqrt{r}\Big \| \Big(\sum_{i\in I} (f(X) - f(X^i))^2\Big)^{1/2}\Big\|_r,
\end{align*}
where $K = \sqrt{\frac{3\sqrt{e}}{\rho_0(\sqrt{e}-1)}}$.
\end{corollary}

Inequalities of this type have been recently derived for measures on finite product spaces~\cite{gtze2018higher} using the Aida--Stroock approach, based on the usual log-Sobolev inequality~\eqref{eq:mlS-sqrts}. This results in the constant $K$ being a multiple of $\rho_1^{-1/2}$. However in many cases (even if $\mu$ is a product measure on a finite set) the constant $\rho_1$ is much smaller than $\rho_0$. Moreover,  as shown in~\cite{gtze2018concentration} in the general case $\rho_1 > 0$ for the Glauber dynamics only if $\mu$ is finitely supported, which is in contrast to $\rho_0$ which, as stated in the previous section, is positive, e.g., for all product measures. Using a Holley--Stroock type perturbation argument (cf.~\cite{MR893137,MR1845806}) one can also easily produce examples of non-product measures with infinite support and $\rho_0 > 0$.

Several examples satisfying the modified log-Sobolev inequality~\eqref{eq:mlS-log} and the log-Sobolev inequality~\eqref{eq:mlS-sqrts} have been recently presented by Sambale and Sinulis in~\cite{Sambale_2020}. They are based on a general theorem concerning approximate tensorization of entropy under a Dobrushin type condition due to Marton~\cite{MR4015662} (see also~\cite{gtze2018higher}).
Let us now state their general result.

Assume that $E$ is finite and define the Dobrushin matrix  $A= (A_{ij})_{i,j\in I}$ as
\begin{displaymath}
  A_{ij} = \sup_{x,y\in \mathcal{X}\colon x_{\{j\}^c} = y_{\{j\}^c}} \|\mathcal{L}(X_i|X_{\{i\}^c} = x_{\{i\}^c}) - \mathcal{L}(X_i|X_{\{i\}^c} = y_{\{i\}^c})\|_{TV}
\end{displaymath}
for $i\neq j$ and $A_{ii} = 0$  (where $\|\cdot\|_{TV}$ denotes the total variation norm). Let  $\alpha = 1 - \|A\|_{\ell_2\to \ell_2}$, where $\|A\|_{\ell_2\to \ell_2}$ is the operator norm of the matrix $A$.
Define also for $J \subsetneq I$ and $i \notin J$
\begin{displaymath}
  \beta_{i,J} = \inf_{\stackrel{x_J \in E^J, y_{J^c} \in E^{J^c}}{(x_{J},y_{J^c}) \in \supp(\mu)}} \PP(X_i = (y_{J^c})_i|X_J = x_J)
\end{displaymath}
(for $J = \emptyset$ we understand the above simply as $\inf_{y \in \supp(\mu)} \mu(\{y\})$).

Finally set
\begin{displaymath}
  \beta = \inf_{J\subsetneq I}\inf_{i\notin J}\beta_{i,J}.
\end{displaymath}

\begin{theorem}[{\cite[Theorem 4.1]{Sambale_2020}}] \label{thm:Sambale}
If $\alpha,\beta > 0$, then
\begin{displaymath}
  \rho_0 \ge \alpha^2\beta,\quad \rho_1 \ge  \frac{\log(2)\alpha^2\beta}{2\log(\beta^{-1})}.
\end{displaymath}
\end{theorem}

(Recall that here $\rho_0$ is the constant on the left-hand side in the modified log-Sobolev~\eqref{eq:mlS-log} and note that due to a different normalization of the Dirichlet form and a different convention concerning constants in~\cite{Sambale_2020} our parameter $\rho_0$ corresponds to $2|I|/\rho_0$ therein).
Using Theorem~\ref{thm:main:mLS_to_Bec} we immediately obtain

\begin{corollary} If $\alpha, \beta > 0$, then for $p \in (1,2]$ the inequality~\eqref{eq:Beckner-GJ} holds with $\alpha_p \ge \frac{\alpha^2\beta}{6}$.
\end{corollary}

Sambale and Sinulis apply Theorem~\ref{thm:Sambale} to several stochastic models, including exponential random graphs, random graph colorings, hardcore model.
In an earlier paper~\cite{gtze2018concentration} with G\"otze they also treat the Ising model.
They are primarily interested in situations when for a family of models on sets $I_n$ with $|I_n|\to \infty$ the constants $\rho_0,\rho_1$ are uniformly separated from zero.
From this point of view the sufficient conditions that can be obtained from Theorem~\ref{thm:Sambale} are the same for both constants.
If one is however interested in a more quantitative analysis, and looks at the dependence of the constants on the parameters of the model, then typically $\rho_0$ is of smaller order than $\rho_1$ (as $\beta$ becomes small).
In particular, a combination of Corollary~\ref{cor:Glauber} with estimates on $\rho_0$ given in Theorem~\ref{thm:Sambale} indeed gives better dependence of constants in moment inequalities than those derived from the Aida--Stroock approach based on $\rho_1$.

Below we discuss this in more detail for selected classical stochastic models.

\medskip

\subsubsection{Exponential random graphs} Let $I_n = \{(i,j)\in [n]^2\colon i < j\}$ and identify elements of $\mathcal{G}_n = \{0,1\}^{I_n}$ with simple graphs on $n$-vertices in a natural way. For $\gamma = (\gamma_1,\ldots,\gamma_s) \in \RR^s$ and simple connected graphs $G_i = (V_i,E_i)$, $i=1,\ldots,s$, let $\mu_\gamma$ be a probability measure on $\mathcal{G}_n$ given by the weight of the form
    \begin{displaymath}
      \exp(-H_\gamma(x)), \quad x\in \mathcal{G}_n,
    \end{displaymath}
    with
    \begin{displaymath}
    H_\gamma(x) = n^2 \sum_{i=1}^s \gamma_i \frac{N_{G_i}(x)}{n^{|V_i|}},
    \end{displaymath}
    where for simple graphs $H = (V_H,E_H)$ and $G= (V_G,E_G)$, $N_{H}(G)$ is the number of homomorphisms from $H$ to $G$, i.e., injective maps $i\colon V_H\to V_G$, which preserve edges. We assume (by convention) that $G_1$ is the complete graph on two vertices and that $|E_i| > 1$ for $i > 1$. Then, for $s=1$, the measure $\mu_\gamma$ corresponds to the distribution of the Erd\H{o}s--R\'enyi random graph $G(n,p)$ with $p = e^{\gamma_1}(1+e^{\gamma_1})^{-1}$. The general model for arbitrary $s$ and graphs $G_i$ is known as the exponential random graph model and it has been an object of intensive studies in recent years, both from the theoretical point of view and in connection to applications, e.g., to modeling of social networks. We refer to~\cite{MR3544262} and~\cite{MR3127871} for a detailed presentation. Sambale and Sinulis provide a sufficient condition for the constants $\rho_0$ and $\rho_1$ to be separated from zero independently of the size $n$ of the model. An inspection of their proof
 reveals that if

    \begin{displaymath}
    \delta := \frac{1}{2}\sum_{i=2}^s |\gamma_i||E_i|(|E_i|-1) < 1,
    \end{displaymath}
    then one can apply Theorem~\ref{thm:Sambale} with some $\alpha \ge 1 - \delta$ and $\beta \ge ce^{-2|\gamma_1|}$ for a universal constant $c>0$. Thus in this case the bounds on $\rho_0$ and $\rho_1$ differ by a factor of order $|\gamma_1|$  when $|\gamma_1| \to \infty$. It is an interesting question to verify if the constant $\rho_1$ indeed depends on $\gamma_1$, which corresponds to the Erd\H{o}s--R\'enyi product-type behaviour of the graph.

\subsubsection{Ising model on finite sets} Let $I = \{1,\ldots,n\}$ and consider the measure on $\mathcal{X} = \{+1,-1\}^n$ given by
\begin{displaymath}
  \mu(\{\varepsilon\}) = \frac{1}{Z}\exp\Big(\frac{1}{2}\sum_{i,j=1}^n J_{ij}\varepsilon_i\varepsilon_j - \sum_{i=1}^n h_i\varepsilon_i\Big),
\end{displaymath}
where $J = (J_{ij})_{i,j=1}^n$ is a symmetric matrix with vanishing diagonal and $h \in \RR^n$. From the statistical physics point of view the matrix $J$ correspond to interaction between spins, while $h$ describes the external field. Concentration inequalities for the Ising model have been considered by many authors starting from the 1990s~\cite{MR2002993,MR3711609}, as it is arguably the most basic discrete model with dependencies. The interest in them has been recently  revived in relation to algorithmic applications \cite{MR3775918,2017arXiv171004170D,MR3873783}. Estimates on $\rho_0$ given by Theorem~\ref{thm:Sambale} have been a starting point for inequalities obtained in~\cite{gtze2018higher,MR3949267}, with the Aida--Stroock approach playing a crucial role. Since each function of the discrete cube can be regarded as a polynomial, by considering its Fourier--Walsh expansion, it is natural to investigate concentration of measure in terms of characteristics related to the polynomial representation. In this case  the dependence of the estimates from the said papers on the constant $\rho_1$ increases with the degree of the polynomial. Therefore an application of Corollary~\ref{cor:Glauber} again allows to improve the behaviour of inequalities in the asymptotic case. As for the parameters $\alpha$, $\beta$ of Theorem~\ref{thm:Sambale}, an inspection of the calculations from~\cite{gtze2018higher} (cf.\ Lemma 3.1. therein) reveals that in this case they can be taken as
\begin{align}\label{eq:Ising-constants}
\alpha \ge  1 - \max_{i\le n}\sum_{j\le n}|J_{ij}|,\quad \beta \ge ce^{-\|h\|_\infty}.
\end{align}
Since the constants in the modified log-Sobolev inequalities do not depend on $h$ in the product case $J= 0$, the same question as in the case of exponential random graphs seems natural also in this setting.

We note that in both cases it is not clear to us whether the above estimates on $\rho_0$ and $\rho_1$ can be improved in a general situation and what the true gap between the two constants is.  Let us also point out that the gap in the estimates of Theorem~\ref{thm:Sambale} appears in the regime $\beta \to 0$ and is only logarithmic in $1/\beta$, while the dependence of the bounds for $\rho_0$ and $\rho_1$ on $\beta$ is polynomial.

\medskip

\subsubsection{Hardcore model}
We will conclude this section with another example of a classical stochastic  model. In this case the model does demonstrate a gap between $\rho_0$ and $\rho_1$ and not just their known lower bounds. More specifically we will show a family of hardcore models on a growing sequence of graphs, for which $\rho_0$ remains separated from zero, while $\rho_1 \to 0$.

Let $G = (V,E)$ be a finite simple graph with maximum degree $\Delta$ and let $\eta > 0$ be a parameter. A binary function $\varepsilon  = (\varepsilon_i)_{i\in V}\in  \mathcal{X} \coloneqq \{0,1\}^V$ will be called admissible if
$\varepsilon_i\varepsilon_j = 0$ whenever $\{i,j\} \in E$. Thus admissible functions describe allocations of particles on $V$ in which one can have at most one particle per vertex and no two adjacent vertices can be occupied simultaneously.  Let $\mu$ be a probability measure on $\mathcal{X}$ given by
\begin{displaymath}
  \mu(\{\varepsilon\}) = \frac{1}{Z}\prod_{i\in V} \eta^{\varepsilon_i}\ind{\{\textrm{$\varepsilon$ is admissible}\}},
\end{displaymath}
where $Z$ is the normalization constant.
Recently Conforti~\cite{conforti2020probabilistic} obtained modified log-Sobolev inequalities and Beckner inequalities for this model. In particular, improving earlier estimates from~\cite{MR3646066}, under the assumption $\eta \Delta < 1$ he proved that
\begin{align}\label{eq:conforti}
  \rho_0 := \rho_0(G,\eta) \ge \frac{1 - \eta(\Delta-1) + 2\min(\eta,1-\eta\Delta)}{1+\eta}.
\end{align}
(we remark that the Glauber dynamics considered by us is slowed down by a factor $1+\eta$ with respect to the one used in~\cite{conforti2020probabilistic}).
  He obtained also general $\Phi$-Sobolev inequalities, in particular~\eqref{eq:Beckner-GJ}. Estimates for $\rho_0$ and $\rho_1$ independent on $|V|$ have been also obtained in~\cite{Sambale_2020} by means of Theorem~\ref{thm:Sambale}, under the assumption $\eta(\Delta - 1) < 1$, however due to the dependence on the parameter $\beta$, they are of worse order.

Below we will provide a sequence of graphs $G_n = (V_n,E_n)$ with $|V_n| = n+1$,  maximum degree $\Delta_n = n$, and such that $\rho_0 = \rho_0(G_n,1/(2\Delta_n))$ is bounded away from zero, while $\rho_1 = \rho_1(G_n,1/(2\Delta_n)) = O(\frac{1}{\log n})$.
Let $G_n$ be a star with center $0$ and $n$ rays, i.e., $V_n = \{0\}\cup[n]$, $E_n = \{\{0,i\}\colon i \in [n]\}$.

Note that the set of  admissible $\varepsilon$'s is composed of $2^n + 1$ elements:  $\varepsilon^\ast$ placing a single particle at zero and $2^n$ configurations with $\varepsilon_0 = 0$. Among them let us distinguish $\varepsilon^\circ$ such that $\varepsilon^\circ(0) = 0$, $\varepsilon^\circ(i) = 1$ for $i \in [n]$.

In particular it follows from the above discussion that $Z = \eta + (1+\eta)^n$, $\mu(\{\varepsilon^\ast\}) = \frac{\eta}{Z}$, $\mu(\{\varepsilon^\circ\}) = \frac{\eta^n}{Z}$.

Let us test the inequality~\eqref{eq:mlS-sqrts} with  $f= \ind{\{\varepsilon^\circ\}}$. Denoting $p = \mu(\{\varepsilon^\circ\})$, we obtain
\begin{displaymath}
  \Ent f = p\log(p^{-1}).
\end{displaymath}
On the other hand
\begin{displaymath}
  \calE(\sqrt{f},\sqrt{f}) = \calE(f,f) = \EE \sum_{i=0}^n (f(X) - f(X^i))_+^2,
\end{displaymath}
where $X,X^i$ are defined at the beginning of this section.

Now, $f$ is nonnegative and equal $0$ on $\{\varepsilon^\circ\}^c$, therefore if $X \neq \varepsilon^\circ$, then
\begin{displaymath}
\sum_{i=0}^n (f(X) - f(X^i))_+^2=0.
\end{displaymath}
On the other hand, if $X = \varepsilon^\circ$ then
$X^0 = X$ and for $i\neq 0$, $X^i \neq X$ with conditional probability $\frac{1}{1+\eta}$. Thus
\begin{displaymath}
\calE(\sqrt{f},\sqrt{f}) = \frac{pn}{1+\eta}.
\end{displaymath}
This shows that $\rho_1 \le \frac{n}{(1+\eta)\log(p^{-1})}$. Since $p^{-1} = \frac{\eta+(1+\eta)^n}{\eta^n}  \ge \frac{1}{\eta^{n}}$, we obtain
\begin{displaymath}
  \rho_1 \le \frac{1}{1+\eta} \frac{1}{\log(\eta^{-1})}.
\end{displaymath}
In particular for $\eta = \frac{1}{2n}$, we get $\rho_1 = O\Big(\frac{1}{\log n}\Big)$, whereas by~\eqref{eq:conforti} $\rho_0 \ge c$ for some  $c>0$, independent of $n$.

\subsection{The symmetric group}\label{sec:S_n}
\subsubsection{General moment estimates}
Consider the symmetric group $S_n$ of permutations of the set $[n]$ equipped with the uniform probability measure $\pi_n$. We will view this measure as the stationary distribution for the \emph{interchange process}. Recall that this process describes the dynamics of $n$ particles, labeled by the set $[n]$ which occupy $n$ distinct sites (also labelled by $[n]$). At rate one a randomly chosen pair of particles exchange their positions. Let $L$ be the infinitesimal operator for this process, i.e.,
\begin{displaymath}
  L f(\sigma) = \frac{1}{n(n-1)}\sum_{i,j = 1}^n \Big(f(\sigma \circ \tau_{ij}) - f(\sigma)\Big) = \frac{2}{n(n-1)}\sum_{1\le i < j \le n}\Big(f(\sigma \circ \tau_{ij}) - f(\sigma)\Big),
\end{displaymath}
where $\tau_{ij}$ stands for the transposition of elements $i$ and $j$.
The corresponding Dirichlet form is
\begin{align*}
\calE (f,g) &= \frac{1}{2n(n-1)n!}\sum_{\sigma \in S_n} \sum_{i,j=1}^n \Big(f(\sigma\circ \tau_{ij}) - f(\sigma)\Big)\Big( g(\sigma\circ \tau_{ij}) - g(\sigma)\Big)\\
& = \frac{1}{n(n-1)n!}\sum_{\sigma \in S_n} \sum_{1\le i < j\le n}\Big(f(\sigma\circ \tau_{ij}) - f(\sigma)\Big)\Big( g(\sigma\circ \tau_{ij}) - g(\sigma)\Big).
\end{align*}

The modified log-Sobolev inequality for this process with $\rho_0 \ge \frac{1}{n-1}$
was obtained independently by Gao--Quastel~\cite{MR2023890} and Bobkov--Tetali~\cite{MR2283379}, who also obtained the Beckner inequality~\eqref{eq:Beckner-GJ} with $\alpha_p = \frac{p(n+2)}{2n(n-1)}$ (we note that the normalization of the generator $L$ differs across various references, we provide here scaled constants matching our setting). The Poincar\'e constant was computed earlier by Diaconis and Shahshahani~\cite{MR626813}. These results can be considered another example demonstrating that the behavior of constants in Poincar\'e, modified log-Sobolev or Beckner inequalities can be much better than of the constant in the classical log-Sobolev inequality, which was proved by Lee and Yau~\cite{MR1675008} to be of order $\frac{1}{n\log n}$. In a recent work G\"otze--Sambale--Sinulis~\cite{gtze2018higher} used the result from~\cite{MR1675008} in combination with the Aida--Stroock approach to obtain certain tail estimates on the symmetric group. However the constants in their estimates explode as $n \to \infty$.

As a consequence of Beckner inequalities we obtain the following moment estimate for functions on the symmetric group. As explained below it generalizes to arbitrary functions the moment bound obtained by Chatterjee \cite{MR2288072} for a special class of variables known as Hoeffding statistics.

\begin{proposition}\label{prop:symmetric-group-moments}
Let $\sigma$ be a uniform random permutation of the set $[n]$. For an arbitrary function $f \colon S_n \to \RR$ and any $r\ge 2$,
\begin{equation}\label{eq:symmetric-moment}
  \|f(\sigma) - \EE f(\sigma)\|_r \le  \ds \sqrt{r} \Big\|\Bigl(\frac{1}{n+2}\sum_{i,j=1}^n (f(\sigma) - f(\sigma\circ \tau_{ij}))^2\Bigr)^{1/2}\Big\|_r
\end{equation}
and
\begin{equation}\label{eq:symetric-moment-+}
  \|(f(\sigma) - \EE f(\sigma))_+\|_r \le  \ds \sqrt{r} \Big\| \Bigl(\frac{1}{n+2}\sum_{i,j=1}^n (f(\sigma) - f(\sigma\circ \tau_{ij}))_+^2\Bigr)^{1/2}\Big\|_r,
\end{equation}
where $\ds = \sqrt{\frac{\sqrt{e}}{\sqrt{e}-1}}$.
\end{proposition}

\begin{proof}
 We have
 \[
  \Gamma(f)(\sigma) = \frac{1}{2} L f^2(\sigma) - f(\sigma)Lf(\sigma) = \frac{1}{2n(n-1)} \sum_{i,j=1}^n \bigl(f(\sigma) - f(\sigma\circ \tau_{ij})\bigr)^2.
 \]
 The assertion follows by the aforementioned result of Bobkov--Tetali and Propositions~\ref{prop:V^2} and~\ref{prop:moments-abstract} (with $s=0$ and $a=
 \min_{p\in(1,2]} \frac{p(n+2)}{2n(n-1)} =\frac{(n+2)}{2n(n-1)}$).
\end{proof}

\subsubsection{Hoeffding statistics}

In the special case of Hoeffding statistics, i.e., functions of the form
\begin{equation}\label{eq:Hoeffding-statistic}
f(\sigma) = \sum_{k=1}^n a_{k\sigma(k)},
\end{equation}
where $(a_{ij})_{i,j=1}^n$ is a real matrix, the inequality~\eqref{eq:symmetric-moment} was proved for integer $r$ by Chatterjee~\cite{MR2288072} (with slightly different constants), who also obtained a Bernstein type inequality for matrices with positive bounded entries. Since then, concentration of measure for Hoeffding statistics has been investigated, e.g., by Albert~\cite{mlis2018concentration} and Bercu--Delyon--Rio~\cite{MR3363542}. They obtained Bernstein type estimates for general bounded entries. The methods used in these references are quite diverse: while Chatterjee uses Stein's method, Albert relies on Talagrand's convex distance inequality on the symmetric group~\cite{MR1361756} and Bercu--Delyon--Rio on martingale methods (used for the first time in the context of random permutations by Maurey~\cite{MR533901}).

Let us mention that Hoeffding statistics have been widely studied in the literature, starting from the article~\cite{MR44058} of Hoeffding himself who obtained their asymptotic normality under certain assumptions (a result known as combinatorial CLT, which in fact motivated Stein to introduce his method for proving weak convergence). They are important since they include many functions of interest in combinatorics or non-parametric statistics. In particular it is easy to see that one can encode in the form~\eqref{eq:Hoeffding-statistic} sums of functions of samples without replacement from a finite populations.

Below we will use the second inequality of Proposition~\ref{prop:symmetric-group-moments} together with an approach of Boucheron--Bousquet--Lugosi--Massart~\cite{MR2123200} to obtain an inequality for suprema of Hoeffding statistics. In the special case of sampling without replacement this inequality will improve certain aspects of an estimate obtained by Tolstikhin--Blanchard--Kloft~\cite{MR3480745}. Their main motivation were applications to transductive learning, we believe that bounds of this type may be also useful in the context of bootstrap for empirical processes.

\begin{proposition}
\label{prop:Hoeffding}
Let $\mathcal{A}$ be a collection of $n\times n$ matrices and let $\sigma$ be a uniform random permutation of the set $[n]$. Define the random variable
\begin{displaymath}
  Z = \sup_{a \in \mathcal{A}} \sum_{k=1}^n a_{k\sigma(k)}.
\end{displaymath}
Then for any $r\ge 2$,
\begin{equation}\label{eq:Hoeffding-prop-1}
\|(Z - \EE Z)_+\|_r \le 4\ds\sqrt{r}A + 10\ds ^2rB_r,
\end{equation}
where $A = \EE\sup_{a \in \mathcal{A}}\sqrt{\sum_{k=1}^n a_{k\sigma(k)}^2}$, $B_r = \big\|\max_{k\le n}\sup_{a\in\mathcal{A}}|a_{k\sigma(k)}|\big\|_r$, and $\ds$ is the constant from Proposition~\ref{prop:symmetric-group-moments}.
As a consequence, for any $r \ge 0$,
\begin{equation}\label{eq:Hoeffding-prop-2}
  \PP\Big(Z \ge \EE Z + 4e\ds\sqrt{r}A + 10e\ds ^2rB_r\Big) \le e^{2-r}.
\end{equation}
\end{proposition}

Before we prove the above proposition, we will provide two examples of applications, comparing it with the results mentioned above.

\begin{example}
If $\mathcal{A}$ consists of a single element and one does not pay attention to universal constants, then inequality~\eqref{eq:Hoeffding-prop-2} is a strengthening of the results by Bercu--Delyon--Rio~\cite{MR3363542} and Albert~\cite{mlis2018concentration}. Their results give
\begin{align}\label{eq:DRA}
  \PP\Big(Z \ge \EE Z + K \Big(\sqrt{r}\Big(\frac{1}{n} \sum_{ij=1}^n a_{ij}^2\Big)^{1/2} + r\max_{i,j\le n}|a_{ij}|\Big)\Big) \le 2e^{-r}
\end{align}
for a certain universal constant $K$. The parameters $A$ and $B_r$ of Proposition~\ref{prop:Hoeffding} read in this case as
\begin{displaymath}
  A = \EE\sqrt{\sum_{k=1}^n a_{k\sigma(k)}^2}, \quad  B_r = \big\|\max_{k\le n}|a_{k\sigma(k)}|\big\|_r.
\end{displaymath}

Clearly $\max_{ij}|a_{ij}| \ge B_r$, moreover in certain situation $\max_{ij} |a_{ij}|$ may be significantly greater than the $r$-th moment $B_r$ (this happens when there are few large elements in the matrix $(a_{ij})$ and $r$ is not too large).

By Jensen's inequality we also have  $(n^{-1}\sum_{ij=1}^n a_{ij}^2)^{1/2} \ge A$, but in fact the difference of these two quantities is at most of the order $\|\max_{k}|a_{k\sigma(k)}|\|_2$ (see the proof of Proposition~\ref{prop:Hoeffding} below), so it can be absorbed in the coefficient in front of $r$. Hence,~\eqref{eq:Hoeffding-prop-2} does not improve on the subgaussian coefficient of~\eqref{eq:DRA}. This is not surprising, since (as observed in~\cite{MR3363542}) if one defines
\[
d_{ij} = a_{ij} - n^{-1}\sum_{l=1}^n a_{il} - n^{-1}\sum_{l=1}^n a_{lj} + n^{-2}\sum_{l,m=1}^n a_{lm},
\]
then $\sum_{k=1}^n d_{k\sigma(k)} = Z - \EE Z$ and $n^{-1} \sum_{ij=1}^n d_{ij}^2 = \Var(Z)$.

To summarize, the main advantage of Proposition~\ref{prop:Hoeffding} over~\eqref{eq:DRA} is the fact that $\max_{i,j\le n} |a_{ij}|$ can be replaced by a smaller parameter $B_r$.
\end{example}

\begin{example}
Let us now specialize to the setting of sampling without replacement and compare our result with the one of Tolstikhin--Blanchard--Kloft~\cite{MR3480745}. To this end we will need to rephrase both results in the same notation. Let us consider a set of vectors $\mathcal{X} \subseteq \{x\in \RR^n\colon x_1+\ldots+x_n = 0\}$. For $m \le n$ let $I_1,\ldots,I_m$ be a uniform sample without replacement and $J_1,\ldots,J_m$ a sample with replacements from the set $[n]$. Define
\begin{displaymath}
  Z = \sup_{x\in \mathcal{X}} \sum_{k=1}^m x_{I_k},\qquad Z' = \sup_{x\in \mathcal{X}} \sum_{k=1}^m x_{J_k}.
\end{displaymath}
Thus $Z'$ can be considered a supremum of the empirical process in independent random variables $J_k$. The tails of such suprema have been thoroughly studied, beginning with the seminal work by Talagrand~\cite{MR1419006}, who obtained Bernstein and Bennett type inequalities. The authors of~\cite{MR3480745} combined optimal forms of such inequalities proved by Bousquet~\cite{MR1890640} with a stochastic domination between  $Z$ and $Z'$ (due to Hoeffding) to derive a bound of the form
\begin{align}\label{eq:T-B-K}
  \PP\Big(Z \ge \EE Z' + \sqrt{2v r} + \frac{r}{3}\sup_{x\in \mathcal{X}} \|x\|_\infty\Big) \le e^{-r}
\end{align}
for $r \ge 0$, where $v = m\sup_{x\in \mathcal{X}}\Var(x_{J_1}) + 2\sup_{x\in \mathcal{X}}\|x\|_\infty \EE Z'$.

One can easily see that the variable $Z$ corresponds to the supremum of Hoeffding statistics over matrices given by $a^x_{ij} = x_j$ for $i \le m$, and $a^x_{ij} = 0$ for $i > m$. Therefore Proposition~\ref{prop:Hoeffding} yields
\begin{align}\label{eq:T-B-K-improved}
  \PP\Big(Z \ge \EE Z  + 4e\ds\sqrt{r}A + 10e\ds^2 rB_r \Big) \le e^{2-r}
\end{align}
with
\begin{displaymath}
  A = \EE \sup_{x \in \mathcal{X}} \Big(\sum_{k=1}^m x_{I_k}^2\Big)^{1/2}, \quad B_r = \Big\|\sup_{x\in\mathcal{X}} \max_{k\le m}|x_{I_k}|\Big\|_r.
\end{displaymath}

Again, in certain situations, especially for  relatively small values of $r$, the quantity $B_r$  may be of smaller order than $\sup_{x\in \mathcal{X}} \|x\|_\infty$ used in~\eqref{eq:T-B-K}. However, the main difference between the two estimates is the fact that~\eqref{eq:T-B-K-improved} provides deviation above $\EE Z$, while~\eqref{eq:T-B-K} considers deviations above $\EE Z'$, which always exceeds $\EE Z$ (see the inequality~\eqref{eq:stochasti-domination} below) and in certain situation can be significantly larger. The Authors of~\cite{MR3480745} provide a bound
\begin{displaymath}
\EE Z' - \EE Z \le 2\frac{m^3}{n}\sup_{x\in \mathcal{X}}\|x\|_\infty.
\end{displaymath}
Specializing to the case $\mathcal{X} \subseteq [-1,1]^n$, it follows from the above estimate that if one is interested in a bound on $Z - \EE Z$ which is of the order $\sqrt{m}$ (corresponding to the CLT type rates one would like to obtain in statistical applications), the inequality~\eqref{eq:T-B-K} is applicable for $m = O(n^{2/5})$. Note that $\EE Z' \le m$, so the quantity $\sqrt{v}$ is of the right order $\sqrt{m}$. On the other hand $\EE \sup_{x \in \mathcal{X}} \Big(\sum_{k=1}^m x_{I_k}^2\Big)^{1/2}$ also equals at most $\sqrt{m}$, so~\eqref{eq:T-B-K-improved} provides a bound on $Z - \EE Z$ of the order $\sqrt{m}$ without any restrictions on $m$ (we remark that the interesting case is $m \le n/2$ since thanks to the mean zero assumption one can always pass from $m$ to $n-m$).

Let us now discuss in more detail the subgaussian coefficients of the two inequalities.
As pointed out in~\cite{gross2010note} it follows from an argument due to Hoeffding~\cite{MR144363} that if $E$ is a normed space and $f\colon [n] \to E$, then for any convex function $\Psi \colon E \to \RR$,
\begin{equation}\label{eq:stochasti-domination}
  \EE \Psi\Big(\sum_{k=1}^m f(x_{I_k})\Big) \le \EE \Psi\Big(\sum_{k=1}^m f(x_{J_k})\Big).
\end{equation}
In particular this implies that
\begin{displaymath}
  A^2 \le \EE \sup_{x \in \mathcal{X}} \sum_{k=1}^m x_{I_k}^2 \le \EE \sup_{x \in \mathcal{X}} \sum_{k=1}^m x_{J_k}^2 \le m\sup_{x\in \mathcal{X}}\Var(x_{J_1}) + 8\sup_{x\in \mathcal{X}}\|x\|_\infty \EE Z'',
\end{displaymath}
with $Z'' = \sup_{x \in \mathcal{X}} \sum_{k=1}^m \varepsilon_k x_{J_k}$, where $\varepsilon_1,\ldots,\varepsilon_m$ are i.i.d.\ Rademacher variables independent of $J_1,\ldots,J_m$.  The last inequality is a classical result of the theory of empirical processes (see, e.g., \cite{MR1258865,MR1782276,MR2123200}) based on symmetrization and Talagrand's contraction principle for Rademacher averages~\cite{MR1102015}. If the set $\mathcal{X}$ is symmetric with respect to the origin, one can further write $\EE Z'' \le 2\EE Z'$. Thus in this case the subgaussian coefficient of~\eqref{eq:T-B-K-improved} is up to absolute constants dominated by $\sqrt{v}$ used in~\eqref{eq:T-B-K}. Let us note that using results from~\cite{MR2123200} one can also provide a similar bound on the subgaussian coefficient of~\eqref{eq:T-B-K-improved} with $\|\max_{i\le m} \sup_{x\in \mathcal{X}} |x_{J_i}|\|_2$ in place of $\sup_{x\in\mathcal{X}} \|x\|_\infty$. Since our goal is rather to illustrate Proposition~\ref{prop:symmetric-group-moments} than to provide the most general estimate, we skip the details.

The above discussion shows that $\eqref{eq:T-B-K-improved}$ may give better estimates than~\eqref{eq:T-B-K}. On the other hand~\eqref{eq:T-B-K} has better constants, in particular provides the optimal constant $\sqrt{2}$ in the subgaussian part. Let us remark that~\cite{MR3480745} contains also a more refined Bennett type inequality for the deviation of $Z$ above $\EE Z'$, which does not follow from the moment type bounds we consider here, however a similar improvement, giving concentration around $\EE Z$ can be up to constants recovered from the modified log-Sobolev inequality on the symmetric group. We do not discuss it in detail, since it is necessarily expressed in terms of $v$ and $\sup_{x \in \mathcal{X}}\|x\|_\infty$ and we are interested primarily in improvements one can obtain by looking at the $p$-th moments rather than the $\ell_\infty$-norm.
\end{example}

Let us now pass to the proof of Proposition~\ref{prop:Hoeffding}.

\begin{proof}[Proof of Proposition~\ref{prop:Hoeffding}] Without loss of generality we can assume that $\mathcal{A}$ is finite, the general case follows then by approximation.
For $i,j \in [n]$ define $Z_{ij} = \sup_{a\in \mathcal{A}} \sum_{k=1}^n a_{k\sigma(\tau_{ij}(k))}$. Note that by the definition of $Z$ and the triangle inequality in $\ell_2$,
\begin{align*}
  \Big(\sum_{ i,j=1}^n (Z - Z_{ij})_+^2\Big)^{1/2} &\le \sup_{a\in \mathcal{A}}\Big(\sum_{i,j=1}^n (a_{i\sigma(i)} + a_{j\sigma(j)} - a_{i\sigma(j)} - a_{j\sigma(i)})^2\Big)^{1/2}\\
  &\le 2\sqrt{n}\sup_{a\in \mathcal{A}}\Big(\sum_{i=1}^n a_{i\sigma(i)}^2\Big)^{1/2} + 2\sup_{a\in \mathcal{A}} \Big(\sum_{i,j=1}^n a_{ij}^2\Big)^{1/2}.
\end{align*}

Therefore, by Proposition~\ref{prop:symmetric-group-moments}, we obtain
\begin{align}\label{eq:Hoeffding-proof-added-1}
  \|(Z-\EE Z)_+\|_r & \le
  2\ds\sqrt{r} \Big\|\sup_{a\in \mathcal{A}}\Big(\sum_{i=1}^n a_{i\sigma(i)}^2\Big)^{1/2}\Big\|_r
  + 2\ds\sqrt{r} \Big( \sup_{a\in \mathcal{A}} \frac{1}{n}\sum_{i,j=1}^n a_{ij}^2\Big)^{1/2} .
\end{align}

We start with estimating the first summand.
Denote
\[S = \sup_{a\in \mathcal{A}}\Big(\sum_{k=1}^n a_{k\sigma(k)}^2\Big)^{1/2},
\quad
S_{ij} = \sup_{a\in \mathcal{A}}\Big(\sum_{k=1}^n a_{k\sigma(\tau_{ij}(k))}^2\Big)^{1/2}.
\]
 By another application  of Proposition~\ref{prop:symmetric-group-moments}, we get
\begin{align}\label{eq:Hoeffding-bound-on-S}
  \|(S-\EE S)_+\|_r \le \ds\sqrt{r} \Big\|\Big(\frac{1}{n}\sum_{i,j=1}^n (S-S_{ij})_+^2\Big)^{1/2}\Big\|_r.
\end{align}
For a fixed value of $\sigma$ let $a \in \mathcal{A}$ be such that
\begin{displaymath}
  S = \Big(\sum_{i=1}^n a_{i\sigma(i)}^2\Big)^{1/2}.
\end{displaymath}
Fix $i,j \in [n]$ and denote
\begin{displaymath}
s = \sqrt{\sum_{k\neq i,j} a_{k\sigma(k)}^2},
\qquad
x = \sqrt{a_{i\sigma(i)}^2+a_{j\sigma(j)}^2},
\qquad
y = \sqrt{a_{i\sigma(j)}^2+a_{j\sigma(i)}^2}.
\end{displaymath}
 Denote also by $\varphi$ the function $t \mapsto \sqrt{s^2 + t^2}$. Then $\varphi$ is convex and increasing on $\RR_+$. Moreover, if $(S - S_{ij})_+$ is nonzero, then $x^2 > y^2$, in particular $x > 0$ and so $\varphi$ is differentiable at $x$. As a consequence, by convexity and monotonicity of the function $t \mapsto t_+^2$, we obtain
\begin{align*}
  (S-S_{ij})_+^2 & \le (\varphi'(x)(x-y))_+^2 \le \varphi'(x)^2 x^2 = \frac{(a_{i\sigma(i)}^2 + a_{j\sigma(j)}^2)^2}{ S^2} \\
  & \le 2\frac{a_{i\sigma(i)}^2 + a_{j\sigma(j)}^2}{ S^2} \max_{k\le n} a_{k\sigma(k)}^2.
\end{align*}
Summing over all $i,j \in [n]$ we obtain
\begin{displaymath}
  \sum_{i,j=1}^n (S - S_{ij})_+^2 \le 4n\max_{k\le n} \sup_{a \in \mathcal{A}} a_{k\sigma(k)}^2,
\end{displaymath}
which in combination with~\eqref{eq:Hoeffding-bound-on-S} gives
\begin{align}\label{eq:Hoeffding-bound-on-S-1}
  \|(S- \EE S)_+\|_r \le 2\ds\sqrt{r}\Big\|\max_{k\le n}\sup_{a\in\mathcal{A}}|a_{k\sigma(k)}|\Big\|_r = 2\ds\sqrt{r}B_r.
\end{align}
Hence,
\begin{equation}
\label{eq:Hoeff-proof-added-15}
\|S\|_r \leq \|\EE S\|_r  + \|(S- \EE S)_+\|_r \leq A+ 2\ds\sqrt{r}B_r.
\end{equation}
Let us note that
\begin{displaymath}
A = \EE S \ge \sup_{a\in \mathcal{A}}\EE \Big(\sum_{i=1}^n a_{i\sigma(i)}^2\Big)^{1/2}.
\end{displaymath}
Applying~\eqref{eq:Hoeff-proof-added-15} with $r=2$ to the one element sets $\{a\}$ instead of $\mathcal{A}$, we obtain
\begin{align}
   \Big( \sup_{a\in \mathcal{A}} \frac{1}{n}\sum_{i,j=1}^n a_{ij}^2\Big)^{1/2} = \Big(\EE \sum_{i=1}^n a_{i\sigma(i)}^2\Big)^{1/2}
  &\le A + 2\sqrt{2}\ds \Big\|\max_{k\le n}|a_{k\sigma(k)}|\Big\|_2 \nonumber \\
  & \le A + 2\sqrt{2}\ds B_r.\nonumber
\end{align}
Combining the above inequality with~\eqref{eq:Hoeffding-proof-added-1} and~\eqref{eq:Hoeff-proof-added-15} we obtain
\begin{align*}
  \|(Z-\EE Z)_+\|_r  
  & \le 4\ds\sqrt{r}A  + 10\ds^2rB_r,
\end{align*}
which ends the proof of~\eqref{eq:Hoeffding-prop-1}. The inequality~\eqref{eq:Hoeffding-prop-2} is now an easy consequence of Chebyshev's inequality in $L_r$ (note that for $r \le 2$ the right-hand side exceeds one, so the inequality is trivial).
\end{proof}

\subsubsection{Multislices}\label{sec:multiscales}
Let us conclude this section with a remark concerning multislices. For  a positive integers $n \ge l$ and a sequence $\kappa = (\kappa_1,\ldots,\kappa_l) \in \NN_+^l$ such that $\kappa_1+\ldots+\kappa_l = n$ consider
\[
U_\kappa = \{x=(x_1,\ldots,x_n) \in [l]^n \colon \#\{j : x_j = i\} = \kappa_i \text{ for }  i = 1,\ldots,l\}
\]
 -- the \emph{multislice} of $[l]^n$ consisting of all the sequences which for $i \le l$ take the value $i$ exactly $\kappa_i$ times. If $l=2$ then $U_\kappa$ can be identified with a slice of the discrete cube $\{0,1\}^n$ by a hyperplane perpendicular to the vector $(1,\ldots,1)$. The dynamics corresponding to switching a randomly chosen pair of coordinates of an element of $U_\kappa$ is related to the Bernoulli--Laplace model of statistical mechanics (which can also be interpreted as an urn scheme). In~\cite{MR2283379} Bobkov and Tetali proved Beckner inequalities for this dynamics in the case of $l = 2$. From this result they inferred modified log-Sobolev inequalities, which were proven independently by Gao--Quastel~\cite{MR2023890}. Again the constant $\rho_1$ in the log-Sobolev inequality~\eqref{eq:mlS-sqrts} degenerates as $n \to \infty$.
It was first computed in~\cite{MR1675008} for $l = 2$.
Recently estimates on this constant for general $n,l,\kappa$ were obtained in~\cite{filmus2018logsobolev}.
These estimates are optimal for $l$ fixed and also deteriorate as $n$ tends to $\infty$.
We would like to point out that in the case of Beckner and modified log-Sobolev inequalities the results on the symmetric group cited in the previous section can be projected onto $U_\kappa$ yielding inequalities with constants of a better order than $\rho_1$, which can be then used to conclude moment estimates.
We skip the rather standard details.

\subsection{Stochastic covering property}\label{sec:SCP}
New examples of measures satisfying the modified log-Sobolev inequality have been recently obtained in the work by Hermon and Salez~\cite{hermon2019entropy,hermon2019modified}. A thorough discussion is beyond the scope of this article, therefore in this and the following section we will provide just an outline of their results and briefly comment on what can be obtained by combining them with ours.

The results in~\cite{hermon2019modified} concern measures on $\mathcal{X} \subseteq \{0,1\}^n$. In order to present them let us recall the definition of \emph{stochastic covering property} (abbrev. SCP) introduced by Pemantle and Peres~\cite{MR3197973}. For $x,y \in \calX$ we will say that $x$ covers $y$ ($x \rhd y$) if
\begin{displaymath}
  x = y \quad \text{ or } \quad \exists_{i \le n} \enspace x = y + e_i,
\end{displaymath}
where $e_i$'s are the standard basis vectors, i.e., if $x\neq y$ then $x$ can be obtained from $y$ by increasing a single coordinate. For probability measures $\nu_1, \nu_2$ on $\mathcal{X}$ we say that $\nu_1$ covers $\nu_2$ if there is  coupling of $\nu_1,\nu_2$ supported on the set $\{(x,y) \in \calX^2\colon x\rhd y\}$.

Let $\mu$ be a probability measure on $\calX$ and $X$ a random vector with law $\mu$. For a set $I \subseteq [n]$ we will write $X_I = (x_i)_{i\in I}$. We say that $\mu$ satisisfies the SCP if for every $I \subseteq [n]$ and $x,y \in \{0,1\}^I$, such that $\PP(X_I = x), \PP(X_I = y) > 0$ and $x \rhd y$,
one has
\begin{displaymath}
  \PP(X_{I^c}\in \cdot | X_I = y) \rhd \PP(X_{I^c}\in \cdot | X_I = x).
\end{displaymath}
Let us moreover introduce a relation $\sim$ on $\mathcal{X}$:
$x\sim y$ if and only if $x$ and $y$ differ at a single coordinate or by a transposition of two coordinates.

Examples of measures satisfying the SCP are given, e.g., by laws of weighted random bases of balanced matroids~\cite{FederMihail}, in particular the uniform measure on the set of all spanning trees of a given graph (we identify here the spanning tree with an element of $\{0,1\}^E$, where $E$ is the set of edges). We refer to~\cite{MR3197973} for further examples.

The authors of~\cite{hermon2019modified} obtain modified log-Sobolev inequalities for measures with SCP.

\begin{theorem} \label{prop:SCP}
Let $X$ be a random vector with values in $\calX \subseteq \{0,1\}^n$ and law $\mu$, satisfying the SCP.

(i)
 Let $Q$ be any kernel, reversible with respect to $\mu$. Then the measure $\mu$ satisfies~\eqref{eq:mlS-log} with constant $\rho_0 \ge \min_{x,y\in\calX, x\sim y}\max(Q_x(y),Q_y(x))$.

(ii)
There exists a kernel $Q$ such that for all $x\in \mathcal{X}$, $Q_x(\cdot)$ is supported on $\{y\in \calX\colon y\sim x\}$, $\sum_{y\in \calX\setminus\{x\}} Q_x(y) \le 1$ and~\eqref{eq:mlS-log} is satisfied with $\rho_0 \ge 1/n$.

(iii)
If $\mu$ is supported on the set $\{x\in\calX\colon \sum_{i=1}^n x_i = k\}$, then one can find a kernel supported on $\{y\in\calX\colon \textrm{$x$ and $y$ differ by a transposition of two coordinates}\}$, such that $\sum_{y\in \calX\setminus\{x\}} Q_x(y) \le 1$ and ~\eqref{eq:mlS-log} is satisfied with $\rho_0 \ge 1/(2k)$.
\end{theorem}

Combining the above theorem with Theorem~\ref{thm:main:mLS_to_Bec} and Proposition~\ref{prop:V^2} we immediately obtain the following corollary.
\begin{corollary}\label{cor:SCP}
In the setting of Theorem~\ref{prop:SCP}, for $p \in (1,2]$, the measure $\mu$ satisfies~\eqref{eq:Beckner-GJ} with $\alpha_p \ge \rho_0/6$. As a consequence, for any function $f\colon \calX \to \RR$ and $r \ge 2$,
\begin{align}
  \|(f(X) - \EE f(X))_+\|_r \le K\sqrt{p}\Big\|\Big(\sum_{y\in \calX} (f(y)-f(x))_-^2 Q_x(y)\Big)^{1/2}\Big\|_r,
\end{align}
where $K=\sqrt{\frac{3\sqrt{e}}{\rho_0(\sqrt{e}-1)}}$.
\end{corollary}

To the best of our knowledge this is the first result in the literature providing Beckner inequalities for measures satisfying the SCP.

The case (i) of Theorem~\ref{prop:SCP} applies in particular to $Q$ given by the Metropolis-Hastings kernel $Q_x(y) = \frac{1}{2kn}\min\Big(\frac{\mu(x)}{\mu(y)},1\Big)$ if $y \sim x$ and $Q_x(y) = 0$ otherwise. Note however that, e.g., in the case of the uniform distribution on slices of the cube application of this part of Theorem~\ref{prop:SCP} gives a suboptimal bound (cf.\ the discussion in Section~\ref{sec:multiscales} and~\cite{MR2283379,MR2023890}). Part (ii) gives the right order of $\rho_0$, however the existence of $Q$ is obtained via an inductive procedure with respect to the dimension $n$ and so in general it is not explicit.

\subsection{Zero-range processes}\label{sec:zero}

Another class of examples coming from the recent work of Hermon and Salez is decribed in~\cite{hermon2019entropy} and concerns zero-range processes, i.e., stochastic systems in which a fixed number $m$ of particles occupy $n$ sites. The particles leave the present site, with rates $\lambda_i$ depending on the number of particles they share it with, and choose the new site according to a prescribed probability measure $p$ common for all the particles. More precisely, let $m, n$ be positive integers and let $\mathcal{X} = \{x\in \NN^n\colon \sum_{i=1}^n x_i = m\}$. Consider functions $\lambda_i \colon \{0\}\cup[m] \to [0,\infty)$, $i=1,\ldots,n$, such that $\lambda_i(0) = 0$  and let $p = (p_1,\ldots,p_n)$ be a probability vector. The zero-range dynamics is given by a Markov generator of the form
\begin{align}\label{eq:L-zero-range}
  Lf(x) = \sum_{i,j=1}^n (f(x+ e_j-e_i) - f(x)) \lambda_i(x_i)p_j,
\end{align}
where $e_1,\ldots,e_n$ is the standard basis in $\RR^n$.
This dynamics is reversible with respect to the probability measure $\mu$ on $\calX$, defined by
\begin{equation}\label{eq:mu-0-range}
  \mu(\{x\}) = \frac{1}{Z}\prod_{i=1}^n\frac{p_i^{x_i}}{\lambda_i(1)\cdots \lambda_i(x_i)}.
\end{equation}

Hermon and Salez obtained a modified log-Sobolev inequality for the case when the rates of escape are sandwiched between two linear functions, with constant $\rho_0$ depending only on the directional coefficients of the functions. In particular, this provides a solution to a conjecture posed by Caputo, Dai Pra, and Posta~\cite{MR2322692,MR2548501}. Below we state their theorem and a corollary one can immediately obtain from it with our results.

We remark that Beckner's inequalities for zero-range processes were previously considered in~\cite{MR3693525} and very recently in~\cite{conforti2020probabilistic} in the case of $p$ being the uniform distribution and under a restriction on $\Delta,\delta$ (for instance~\cite{conforti2020probabilistic} assumes that $\Delta \le 2\delta$). See Remark~5.3 in~\cite{conforti2020probabilistic} for a detailed discussion of the applicability of the Bakry-\'Emery approach used in these references. The equivalence with the modified log-Sobolev inequality allows to go beyond this restriction and conclude Beckner inequalities directly from the result by Hermon and Salez.

\begin{theorem}
Assume that for $l \in \{0\}\cup [m-1]$,
\begin{align}\label{eq:zero-range-assumption}
  \delta \leq \lambda_i(l+1) - \lambda_i(l) \leq \Delta,
\end{align}
where $\delta,\Delta$ are positive constants.
Then the zero range dynamics corresponding to the generator~\eqref{eq:L-zero-range} satisfies the modified log-Sobolev inequality with $\rho_0 \ge \frac{\delta^2}{2\Delta}$.
\end{theorem}

Theorem~\ref{thm:main:mLS_to_Bec} and Proposition~\ref{prop:V^2} immediately yield the following

\begin{corollary}\label{proc:zero range}
If the assumption~\eqref{eq:zero-range-assumption} is satisfied, then the zero-range dynamics satisfies for any $p \in (1,2]$ the Beckner inequality~\eqref{eq:Beckner-GJ} with constant $\alpha_p \ge \frac{\delta^2}{12\Delta}$.

As a consequence, if $X=(X_1,\ldots,X_n)$ is a random vector with law $\mu$ given by~\eqref{eq:mu-0-range}, then for every function $f\colon \calX \to \RR$ and $r \ge 2$,
\begin{displaymath}
  \|(f(X) - \EE f(X))_+\|_r \le K\frac{\sqrt{\Delta}}{\delta}\sqrt{r}\Big\|\Big(\sum_{i=1}^n\sum_{j=1}^n (f(X + e_j - e_i) - f(X))_-^2\lambda_i(X_i) p_j\Big)^{1/2}\Big\|_r,
\end{displaymath}
where $K = \sqrt{\frac{6\sqrt{e}}{\sqrt{e}-1}}$.
\end{corollary}

\subsection{The Poisson space} \label{sec:Poisson}
We will now present applications of our results to concentration of measure on the Poisson space. In literature there are quite a few results, providing functional inequalities and concentration estimates on path spaces of Poisson point processes. In particular An\'e and Ledoux~\cite{MR1757600} obtained certain modified log-Sobolev inequalities (of a different form than~\eqref{eq:mlS-log}), Wu~\cite{MR1800540} proved a modified log-Sobolev inequality implying in particular the one from~\cite{MR1757600} as well as~\eqref{eq:mlS-log}, Chafa\"i~\cite{MR2081075} considered general $\Phi$-Sobolev inequalities (including ones of Beckner type), Reynaud-Bouret~\cite{MR1981635} obtained concentration estimates for suprema of compensated stochastic  integrals (see also~\cite{MR2073426,MR2294982} for multiple stochastic integrals). More recently Reitzner introduced a version of the convex distance inequality~\cite{MR3151752}, Bachmann and Peccati~\cite{MR3485348} used modified log-Sobolev inequalities due to Wu to obtain concentration results with focus on geometric functionals, an approach which was subsequently extended by Bachmann~\cite{MR3473096}, Bachmann and Reitzner~\cite{MR3849811}. Nourdin, Peccati, and Yang~\cite{nourdin2019restricted} proved restricted hypercontractive for certain classes of functions, whereas Gozlan, Herry, Peccati~\cite{gozlan2020transport} obtained transportation type inequalities.

Our goal is to complement these results with moment estimates and derive from them concentration inequalities. To the best of our knowledge the moment inequalities we present in Proposition \ref{prop:moments-Poisson} are the first general moment estimates on the Poisson space beyond those that can be obtained from the Poincar\'e inequality. Moreover, as one can easily see using infinite divisibility and the Central Limit Theorem, the growth of constants in our estimate as $r \to \infty$ is optimal.

We will start by a necessarily brief introduction of the setting. We refer to~\cite{MR3791470,MR3585396} for a detailed presentation of Poisson point processes and stochastic calculus on the Poisson space.

Let $(\mathcal{X},\mathcal{B})$ be a measurable space endowed with a $\sigma$-finite measure $\lambda$. Let $\mathcal{N}$ be the space of all $\NN\cup\{\infty\}$-valued measures on $(\mathcal{X},\mathcal{B})$ which can be expressed as countable sums of $\NN$-valued measures. The measurable structure on $\mathcal{N}$ that we consider is given by the smallest $\sigma$-field $\mathcal{G}$ such that for all $B\in \mathcal{B}$ the map $\mu \mapsto \mu(B)$ is $\mathcal{G}$-measurable. Recall that an $\mathcal{N}$-valued random variable $\eta$ is a Poisson process with intensity $\lambda$  if
\begin{itemize}
\item[(i)] for every $B\in \mathcal{B}$, the random variable $\eta(B)$ has Poisson distribution with parameter $\lambda(B)$ (which we interpret as the Dirac mass at $\lambda(B)$ if $\lambda(B)\in \{0,\infty\}$),
\item[(ii)] for every positive integer $m$ and all pairwise disjoint sets $B_1,\ldots,B_m\in \mathcal{B}$, the random variables $\eta(B_1),\ldots,\eta(B_m)$ are jointly independent.
\end{itemize}

A Poisson process $\eta$ is called proper, if there exists a random variable $\kappa \le \infty$ and a sequence of $\mathcal{X}$-valued random variables $X_i$ such that
\begin{align}\label{eq:Poisson-proper}
  \eta = \sum_{i=1}^\kappa \delta_{X_i},
\end{align}
where $\delta_x$ stands for Dirac's mass at $x$. Corollary~3.7 in~\cite{MR3791470} asserts that for every Poisson process there exists a proper Poisson process with the same distribution. We will use this fact together with $\sigma$-finiteness of $\lambda$ to avoid certain measurability issues in the definition of quantities that we are about to consider.

More precisely, let $\mathcal{X}_n \subseteq \mathcal{X}$ be a sequence of sets with $\bigcup_n \mathcal{X}_n = \mathcal{X}$, such that $\lambda(\mathcal{X}_n) < \infty$ for all $n$.
We can and do assume that $\eta$ is proper and consider $\eta$ as a random variable with values in the space $\mathcal{M}\subset\mathcal{N}$ of measures of the form $\mu = \sum_{i=1}^K \delta_{x_i}$ where $K \le \infty$ and $x_i \in \mathcal{X}$, such that for every $n$, $\mu(\mathcal{X}_n) < \infty$.
We will again endow this space with the smallest $\sigma$-field $\mathcal{S}$ such that the maps $\mu\mapsto \mu(B)$ are measurable for all $B\in\mathcal{B}$.

\begin{remark}\label{rk:poisson-domain-discussion}
It is not difficult to see that $\mathcal{S} = \{A\cap \mathcal{M}\colon A \in \mathcal{G}\}$. In particular $\mathcal{S}$-measurable functions on $\mathcal{M}$ are just restrictions of $\mathcal{G}$-measurable functions on $\mathcal{N}$.
We stress that in what follows we will consider inequalities for functions defined on the path space $(\mathcal{N}, \mathcal{G})$, however one should remember that the underlying Poisson process takes values in $(\mathcal{M},\mathcal{S})$, which makes the quantities we will deal with well-defined. In what follows, so as not to obscure the main ideas, we will not discuss in detail the standard but somewhat tedious measurability issues.
\end{remark}

For $F\colon \mathcal{N} \to \RR$ and $x \in \mathcal{X}$ define
\begin{displaymath}
 D_x^+ F(\eta) = F(\eta + \delta_x) - F(\eta)
\end{displaymath}
and
\begin{displaymath}
  D_x^- F(\eta) = F(\eta) - F(\eta - \delta_x)
\end{displaymath}
if $\eta \ge \delta_x$ and $D_x^- F = 0$ otherwise.

In~\cite{MR1800540} Wu proved that for arbitrary positive integrable $F\colon \mathcal{N}\to [0, \infty)$,
\begin{equation}\label{eq:Wu}
  \Ent F \le \EE \int_\mathcal{X} (D_x^+ \Phi(F) - \Phi'(F) D_x^+ F)\lambda(dx),
\end{equation}
where $\Phi(t) = t\log t$ and the expectation is taken with respect to $\eta$.
From this inequality it easily follows that $\eta$ satisfies the modified log-Sobolev inequality
\begin{align}\label{eq:mLSI-Poisson}
\Ent F \le \calE (F, \log F)
\end{align}
for nonnegative $F\colon \mathcal{N}\to [0, \infty)$, and
	\begin{displaymath}
  \calE(F,G) = \int_\mathcal{X}\EE (D^+_x F)(D^+_x G)\lambda(dx)
\end{displaymath}
is a bilinear form with
\begin{displaymath}
\Dom(\calE) = \Big\{F \in L_0(\mathcal{N},\mathcal{L}(\eta))\colon \int_\mathcal{X}\EE (D^+_x F)^2 \lambda(dx) < \infty\Big\},
\end{displaymath}
where $\mathcal{L}(\eta)$ denotes the law of $\eta$.

Recall now the Mecke formula (see, e.g. \cite[Theorem 4.1]{MR3791470}), which asserts that for every measurable function  $H\colon \mathcal{N}\times\mathcal{X} \to [0,\infty)$,
\begin{equation*}
  \EE \int_\mathcal{X} H(\eta,x)\eta(dx) = \int_\mathcal{X} \EE H(\eta+\delta_x,x)\lambda(dx).
\end{equation*}
We will actually need a slightly different version of this formula, valid for proper Poisson processes, given in \cite[Theorem 4.5]{MR3791470}. Namely for $\mu \in \mathcal{N}$ and $x\in\mathcal{X}$ define
$\mu\backslash \delta_x$ as $\mu - \delta_x$ if $\mu \ge \delta_x$, and as $\mu$ otherwise. If $\eta$ is a proper Poisson process, then for every $H$ as above
\begin{equation}\label{eq:Mecke}
  \EE \int_\mathcal{X} H(\eta\backslash \delta_x,x)\eta(dx) = \int_\mathcal{X} \EE H(\eta,x)\lambda(dx).
\end{equation}
We remark that the assumption that $\eta$ is proper allows to interpret the left-hand side as
\begin{displaymath}
  \EE \sum_{i=1}^\kappa H\Big(\sum_{j=1}^\kappa \ind{\{j\neq i\}} \delta_{X_j},X_i\Big).
\end{displaymath}
Let us also note that clearly Mecke's formula holds also for measurable functions $H\colon \mathcal{N}\times \calX \to \RR$, provided that its left- or right-hand side with $H$ replaced by $|H|$ is finite.

For $\rho = \sum_{k=1}^K \delta_{x_k} \in \mathcal{M}$ (with $K \le \infty$) we define a measure $Q_\rho$ on $\mathcal{M}$ as
\begin{align*}
    Q_\rho(A) & = \int_\calX \ind{A}(\rho + \delta_x)\lambda(dx) + \int_\calX \ind{A}(\rho - \delta_{x})\rho(dx)\\
    & = \lambda(\{x\colon \rho + \delta_x \in A\}) + \sum_{k=1}^\kappa \ind{A}(\rho - \delta_{x_k}).
\end{align*}

Using Dynkin's $\pi$-$\lambda$ theorem one can prove that the definition of $Q$ does not depend on the representation of $\rho$ as a sum of Dirac's deltas, moreover $Q$ is a kernel on $\mathcal{M}$ (we skip the details and just remark that this is the main reason for which we introduce the space $\mathcal{M}$).

By Mecke's formula~\eqref{eq:Mecke} for every measurable $G\colon \mathcal{M}\times \mathcal{M} \to [0,\infty)$,
\begin{align*}
&\EE \int_\mathcal{M} G(\sigma,\eta)Q_\eta(d\sigma) \\
&= \EE \int_\mathcal{X} G(\eta + \delta_x, \eta)\lambda(dx)  + \EE \int_\mathcal{X} G(\eta - \delta_x,\eta)\eta(dx)\\
& = \EE \int_\mathcal{X} G(\eta, \eta-\delta_x)\eta(dx) +  \EE \int_\mathcal{X} G(\eta,\eta+\delta_x)\lambda(dx) \\
&= \EE \int_\mathcal{M} G(\eta,\sigma)Q_\eta(d\sigma).
\end{align*}
Thus the kernel $Q_\sigma$ is reversible with respect to the law of $\eta$.

Note that by another application of Mecke's formula~\eqref{eq:Mecke}, for any $F,G\in\Dom(\calE)$,
\begin{align}\label{eq:Poisson-form-kernel}
  \calE(F,G)  &= \frac{1}{2}\Big(\EE \int_\mathcal{X} (F(\eta) - F(\eta - \delta_x))(G(\eta) - G(\eta - \delta_x))\eta(dx) \\
  &\qquad + \EE \int_\mathcal{X} (F(\eta+\delta_x) - F(\eta))(G(\eta+\delta_x) - G(\eta))\lambda(dx)\Big)\nonumber\\
   &= \frac{1}{2}\EE \int_\mathcal{M}(F(\sigma) - F(\eta))(G(\sigma) - G(\eta))Q_\eta(d\sigma).\nonumber
\end{align}

We remark that the functions $F,G$ above are defined on $\mathcal{N}$, but their restrictions to $\mathcal{M}$ are $\mathcal{S}$-measurable (cf.\ Remark~\ref{rk:poisson-domain-discussion}), so the last expression in the above formula is well defined. In particular one can see that the value of $\calE(F,G)$ depends only on the behaviour of $F$ and $G$ on $\mathcal{M}$.

Consider now the space
\begin{displaymath}
\calA = \Big\{F \in L_0(\mathcal{M},\mathcal{L}(\eta))\colon  \int_\mathcal{M} (F(\eta) - F(\sigma))^2Q_\eta(d\sigma) < \infty \textrm{ a.s.}\Big \}.
\end{displaymath}
Observe also that the restriction from $\mathcal{N}$ to $\mathcal{M}$ gives a natural identification of $L_0(\mathcal{M},\mathcal{L}(\eta))$ and $L_0(\mathcal{N},\mathcal{L}(\eta))$ (cf.\ again Remark~\eqref{rk:poisson-domain-discussion}), therefore we can also consider $\calA$ as a subspace of the latter space.

On $\mathcal{A}\times\mathcal{A}$ define
\begin{align*}
  \Gamma(F,G) &=
  \frac{1}{2}\int_\mathcal{M} (F(\eta) - F(\sigma))(G(\eta)-G(\sigma))Q_\eta(d\sigma)\\
  &= \frac{1}{2}\Big(\int_\mathcal{X} (F(\eta) - F(\eta - \delta_x))(G(\eta) - G(\eta - \delta_x))\eta(dx)\\
  &\qquad + \int_\mathcal{X} (F(\eta+\delta_x) - F(\eta))(G(\eta+\delta_x) - G(\eta))\lambda(dx)\Big)\\
  &= \frac{1}{2}\int_\mathcal{X} (D_x^- F(\eta))(D_x^- G(\eta))\eta(dx) + \frac{1}{2}\int_\mathcal{X} (D_x^+ F(\eta))(D_x^+ G(\eta))\lambda(dx)
\end{align*}
and
\begin{align}\label{eq:Poisson-gamma-plus}
  \Gamma_+(F) &=  \int_\mathcal{M} (F(\eta) - F(\sigma))_+^2Q_\eta(d\sigma)\\
  &=\int_\mathcal{X} (F(\eta) - F(\eta - \delta_x))_+^2\eta(dx) + \int_\mathcal{X} (F(\eta+\delta_x) - F(\eta))_-^2\lambda(dx)\nonumber\\
  &= \int_\mathcal{X} (D_x^- F(\eta))_+^2\eta(dx) + \int_\mathcal{X} (D_x^+ F(\eta))_-^2\lambda(dx).\nonumber
\end{align}
By \eqref{eq:Poisson-form-kernel} we then have
\begin{displaymath}
  \calE(F,G) = \EE \Gamma(F,G),
\end{displaymath}
which shows that one can interpret Wu's inequality \eqref{eq:mLSI-Poisson} in the setting of our main results (cf.~\eqref{eq:Gamma-kernel},~\eqref{eq:kernel-to-form},~\eqref{eq:Gamma+}) and~\eqref{eq:mLSI-Poisson} becomes just~\eqref{eq:mlS-log} with $\rho_0 = 1$. In particular we also obtain Beckner's inequality~\eqref{eq:Beckner-GJ} with $\alpha_p \ge 1/6$.

\begin{remark}
We remark that $\Gamma$ is closely related to the carr\'e du champ operator for the Ornstein--Uhlenbeck process on the Poisson space (see~\cite{MR3585396}). In \cite[Proposition 2.6]{MR3896829}
it is shown that under suitable assumptions $\Gamma(F,G)$ actually coincides with the carr\'e du champ operator. Similarly as in~\cite{MR3485348} we find it however simpler to introduce $\Gamma$ and $\Gamma_+$ via the Mecke formula~\eqref{eq:Mecke}, which gives greater generality and does not require a detailed discussion of domains.
\end{remark}

Beckner's inequality~\eqref{eq:Beckner-GJ} and Propositions~\ref{prop:V^2} and~\ref{prop:moments-abstract}
 imply the following proposition providing Sobolev type inequalities on the Poisson space.
\begin{proposition}\label{prop:moments-Poisson}
For any $F\colon \mathcal{N}\to \RR$ and any $r \ge 2$,
\begin{align*}
  \|F - \EE F\|_r &\le \fs \sqrt{r}\|\sqrt{\Gamma(F)}\|_r\\
  &= \fs \sqrt{r} \Big\|\Big(\int_{\mathcal{X}} (D_x^+ F)^2\lambda(dx) + \int_{\mathcal{X}}(D_x^- F)^2\eta(dx)\Big)^{1/2}\Big\|_r
\end{align*}
and
\begin{align*}
  \|(F - \EE F)_+\|_r &\le \fs \sqrt{r}\|\sqrt{\Gamma_+(F)}\|_r \\
  &= \fs \sqrt{r} \Big\|\Big(\int_{\mathcal{X}} (D_x^+ F)_-^2\lambda(dx) + \int_{\mathcal{X}}(D_x^- F)_+^2\eta(dx)\Big)^{1/2}\Big\|_r,
\end{align*}
where $\fs = \sqrt{3\frac{\sqrt{e}}{\sqrt{e}-1}}$.
\end{proposition}

Recently Bachmann and Peccati~\cite{MR3485348} used Wu's inequality~\eqref{eq:Wu} to derive concentration inequalities for Poisson functionals by various variants of the Herbst argument. They put special emphasis on increasing functionals, i.e., functionals $F$ such that $D_x^+ F \ge 0$ arguing that for them the second integral on the right-hand side of~\eqref{eq:Poisson-gamma-plus} vanishes, while the first integral can be often relatively easily estimated by appealing just to geometric properties of the functional, without taking into account the dependence on the intensity $\lambda$. Further applications of inequalities from~\cite{MR3485348} were presented in~\cite{MR3473096,MR3849811}.
The approach used in these papers relies on Herbst's argument, which usually requires either that $\Gamma_+ (F)$ or $\Gamma(F)$ is uniformly bounded or that the function has some self-bounding properties (e.g., $\Gamma_+ (F) \le \varphi(F)$ for some function $\varphi$). One aspect in which moment estimates of Proposition~\ref{prop:moments-Poisson} complement this approach is that they can be easily used also if $\Gamma(F)$ or $\Gamma_+(F)$ have heavier tails, e.g., if they are not exponentially integrable.

Proposition~\ref{prop:moments-Poisson} may also be an efficient tool in the self-bounded setting leading to inequalities which are (up to constants) comparable to those presented in said articles. We will illustrate it with use of the following proposition, which may be considered a counterpart of \cite[Corrolary 3.5]{MR3485348}, which instead of moments concerns the Laplace transform.

\begin{proposition}\label{prop:Poisson-self-bounded}
Assume that $F\colon \mathcal{N}\to [0,\infty)$ is a measurable function which satisfies
\begin{align}\label{eq:self-bounded-assumption}
  \Gamma_+(F) \le F^\alpha G
\end{align}
for some $\alpha \in [0,2)$ and a function $G\colon \mathcal{N} \to [0,\infty)$. Then for $r \ge 2$ and $\varepsilon \in (0,1)$,
\begin{equation}
\label{eq:Poisson-self-bounded}
\begin{split}
  \|(F-\EE F)_+\|_r
  &\le 2\fs\sqrt{r}(\EE F)^{\alpha/2}\|G^{1/(2-\alpha)}\|_r^{1-\alpha/2}  \\
  &\qquad + (2\fs)^{2/(2-\alpha)} r^{1/(2-\alpha)}\|G^{1/(2-\alpha)}\|_r,
\end{split}
\end{equation}
where $\fs$ is the constant from Proposition~\ref{prop:moments-Poisson}.
\end{proposition}

\begin{proof}
Denote $A:= \|(F-\EE F)_+\|_r$. We will first show that if $\EE G^{r/(2-\alpha)} < \infty$, then $A < \infty$.
Note that the inequality $a^{\alpha/2}(a^{1-\alpha/2} - b^{1-\alpha/2})_+ \le (a - b)_+$ implies that
\begin{displaymath}
  F^\alpha \Gamma_+(F^{1-\alpha/2}) \le \Gamma_+(F) \le F^{\alpha} G.
\end{displaymath}
As a consequence $\Gamma_+(F^{1-\alpha/2}) \le G$ and thus if $\EE G^{q/2} < \infty$, then by Proposition~\ref{prop:moments-Poisson} $\EE F^{q(1-\alpha/2)} < \infty$. Choosing $q = 2r/(2-\alpha)$ we obtain that $A < \infty$.

Passing to the main part of the proof, we can assume that the right-hand side of~\eqref{eq:Poisson-self-bounded} is finite and as a consequence $A < \infty$.
By Proposition~\ref{prop:moments-Poisson} and the assumption~\eqref{eq:self-bounded-assumption},
\begin{displaymath}
  A^r \le \fs^r r^{r/2} \EE F^{\alpha r/2} G^{r/2} \le \fs^r r^{r/2} (\EE F^r)^{\alpha/2} (\EE G^{r/(2-\alpha)})^{1-\alpha/2},
\end{displaymath}
where in the second estimate we used H\"older's inequality. Using the triangle inequality in $L_r$, together with subadditivity of the function $t \mapsto t^{\alpha/2}$ we get
\begin{align*}
  A &\le \fs \sqrt{r} \|F\|_r^{\alpha/2}\|G^{1/(2-\alpha)}\|_r^{1-\alpha/2} \\
  &\le \fs\sqrt{r} A^{\alpha/2} \|G^{1/(2-\alpha)}\|_r^{1-\alpha/2} + \fs\sqrt{r}(\EE F)^{\alpha/2}\|G^{1/(2-\alpha)}\|_r^{1-\alpha/2},
\end{align*}
which easily implies that either
\begin{displaymath}
A \le 2\fs\sqrt{r}(\EE F)^{\alpha/2}\|G^{1/(2-\alpha)}\|_r^{1-\alpha/2}
\end{displaymath}
 or
\begin{displaymath}
A \le (2\fs)^{2/(2-\alpha)} r^{1/(2-\alpha)}\|G^{1/(2-\alpha)}\|_r,
\end{displaymath}
proving the proposition.
\end{proof}

Let us illustrate Proposition~\ref{prop:Poisson-self-bounded} with two applications.

\subsubsection{Suprema of Poisson stochastic integrals}
Let $\mathcal{F}$ be a countable family of real valued functions on $\mathcal{X}$. Consider random variables of the form
\begin{align}\label{eq:empirical-Z}
  Z = \sup_{f\in \mathcal{F}} \int_\mathcal{X} f(x)\eta(dx),
\end{align}
where all functions $f\in \mathcal{F}$ are nonnegative and $\mathcal{F} \subseteq L_1(\mathcal{X},\lambda)$
and
\begin{align}\label{eq:empirical-S}
  S = \sup_{f\in \mathcal{F}} \int_\mathcal{X} f(x)(\eta - \lambda)(dx),
\end{align}
where $\mathcal{F} \subseteq L_2(\mathcal{X},\lambda)$.
Here the compensated integral is defined in the usual way, first directly on $L_1(\mathcal{X},\lambda)\cap L_2(\mathcal{X},\lambda)$, then extended by density -- we refer to \cite[Chapter 12]{MR3791470} for background on the Wiener-Ito integral in the Poisson case.
In the case when the functions in $\mathcal{F}$ are uniformly bounded and $\lambda$ is finite, concentration inequalities for $Z$ and $S$ were obtained by Reynaud-Bouret in~\cite{MR1981635}. Here we will complement them with moment inequalities valid for not necessarily bounded classes or finite intensity measures.

The inequalities we obtain can be considered counterparts of results due to Gin\'e--Lata{\l}a--Zinn~\cite{MR1857312} for empirical processes in independent random variables. Originally they were derived from Talagrand's concentration inequality for empirical processes and the Hoffman--J{\o}rgensen inequality; an alternate proof based on moment estimates of Theorem~\ref{thm:BBLM} was provided by Boucheron--Bousquet--Lugosi--Massart~\cite{MR2123200}. We remark that it should be possible to use this inequality together with infinite divisibility of Poisson processes in a similar way as in~\cite{MR1981635} to recover the estimates we present below (passing through finite intensity measures first), it seems however that this approach would require dealing with more technicalities in comparison with a direct application of general Poissonian moment estimates.

Let us start with the estimate on $Z$. Assume first that $\mathcal{F}$ is a finite class of functions. Note that if $f\ge 0$ for all $f\in \mathcal{F}$, then $Z(\eta) \le Z(\eta+\delta_x)$ and as a consequence
\begin{displaymath}
  \Gamma_+(Z) = \int_\mathcal{X} (D_y^- Z)_+^2\eta(dy).
\end{displaymath}
If $g_\eta \in \mathcal{F}$ is such that $Z(\eta) = \int_\mathcal{X} g_\eta(x)\eta(dx)$, then for all $y \in \supp(\eta):=\{X_1,X_2,\ldots\}$, where $X_i$ are the random variables from the representation \eqref{eq:Poisson-proper}, we have
\begin{displaymath}
  D_y^- Z \le  \int_\mathcal{X} g_\eta(x)\eta(dx) - \int_\mathcal{X} g_\eta(x)(\eta-\delta_y)(dx) = g_\eta(y).
\end{displaymath}
Thus
\begin{displaymath}
  \Gamma_+(Z) \le \int_\mathcal{X}  g_\eta(y)^2 \eta(dy) \le ZG
\end{displaymath}
where
\begin{align}\label{eq:def-G}
G = \sup_{y \in \supp(\eta)} \sup_{f\in\mathcal{F}} f(y)
\end{align}
and as a consequence, an application of Proposition~\ref{prop:Poisson-self-bounded} with $\alpha = 1$ (followed by the monotone convergence theorem if $\mathcal{F}$ is infinite) gives

\begin{corollary}
\label{cor:empirical-Z}
 If $Z$ is given by~\eqref{eq:empirical-Z}, and $\EE Z < \infty$, then for a universal constant $C$ and  all $r \ge 2$ and any $\varepsilon > 0$,
\begin{align*}
  \|(Z - \EE Z)_+\|_r &\le C(\sqrt{r}\sqrt{\EE Z}\sqrt{\|G\|_r} + r\|G\|_r)\\
  & \le C(\varepsilon\EE Z + (1+ \varepsilon^{-1})r\|G\|_r).
\end{align*}
\end{corollary}

Let us now pass to the variable $S$ given by~\eqref{eq:empirical-S}. By a limiting argument, we can again assume without loss of generality that $\mathcal{F}$ is finite. Further we can assume that $\mathcal{F} \subseteq L_1(\mathcal{X},\lambda)\cap L_2(\mathcal{X},\lambda)$, so that one can consider separately integration with respect to $\eta$ and $\lambda$.

Let $g_\eta \in \mathcal{F}$ be such that $S = \int_\mathcal{X} g_\eta(x)\eta(dx) - \int_\mathcal{X} g_\eta \lambda(dx)$.
Arguing similarly as for the variable $Z$, we have
\begin{align*}
\Gamma_+(S) & \le \int_\mathcal{X} g_\eta(y)_-^2 \lambda(dy) + \int_\mathcal{X} g_\eta(y)_+^2 \eta(dy)\\
& \le \sup_{f\in \mathcal{F}} \int_\mathcal{X} f(x)^2 \lambda(dx) + \sup_{f\in \mathcal{F}} \int_\mathcal{X} f(x)^2 \eta(dx).
\end{align*}
Thus, by Proposition~\ref{prop:moments-Poisson}, the subadditivity  of the function $x\mapsto x^{1/2}$ and the triangle inequality, we obtain
\begin{align*}
\|(S-\EE S)_+\|_r \le C\sqrt{r} \Big(\Big(\sup_{f\in \mathcal{F}} \int_\mathcal{X} f(x)^2 \lambda(dx)\Big)^{1/2} + \Big\|\sup_{f\in \mathcal{F}} \int_\mathcal{X} f(x)^2 \eta(dx)\Big\|_{r/2}^{1/2}\Big).
\end{align*}
The second term can be bounded from above by Corollary~\ref{cor:empirical-Z} applied to $\mathcal{F}' = \{f^2\colon f\in\mathcal{F}\}$, which results in

\begin{corollary}
If $S$ is given by~\eqref{eq:empirical-S}, then for all $r \ge 4$,
\begin{displaymath}
\|(S - \EE S)_+\|_r \le C\Big(\sqrt{r}\Sigma + r \Big\|\sup_{x\in{\rm supp \eta}}\sup_{f\in \mathcal{F}} |f(x)|\Big\|_r\Big),
\end{displaymath}
where
\begin{displaymath}
  \Sigma^2 = \sup_{f\in \mathcal{F}} \int_\mathcal{X} f(x)^2 \lambda(dx) + \EE \sup_{f\in\mathcal{F}} \int_\mathcal{X} f^2(x)\eta(dx)
\end{displaymath}
and $C$ is a universal constant.
\end{corollary}

We remark that if the class $\mathcal{F}$ is uniformly bounded, then by Chebyshev's inequality, the above corollary allows to recover (up to universal constants) the exponential upper tail estimates for $S$ obtained in~\cite{MR1981635}.

\subsubsection{Non-negative $U$-statistics}
Another application of Proposition~\ref{prop:Poisson-self-bounded} is related to geometric functionals of the Poisson process, specifically certain non-negative $U$-statistics, investigated recently by several authors \cite{MR3473096,MR3485348,MR3849811,gozlan2020transport}. For a measurable kernel $h\colon \mathcal{X}^m\to [0,\infty)$, symmetric under permutation of arguments, let us define
 \begin{displaymath}
   U(\eta) = \sum_{i_1,\ldots,i_m}^{\neq} h(X_{i_1},\ldots,X_{i_m}),
 \end{displaymath}
where the $X_i$'s are given by the representation~\eqref{eq:Poisson-proper} and the superscript $\neq$ indicates that the summation is taken over pairwise disjoint indices.

 Let us note that for nonnegative $h$ we have $D_X^+ h \ge 0$, therefore
 \begin{displaymath}
   \Gamma_+(U) = \int_\mathcal{X} (D_x^- U)_+^2\eta(dx) = m^2 \sum_{i} \Big(\sum_{i_1,\ldots,i_{m-1}\colon i_j\neq i}^{\neq} h(X_{i_1},\ldots,X_{i_{m-1}},X_i)\Big)^2.
 \end{displaymath}
  Therefore, using Proposition~\ref{prop:Poisson-self-bounded} and Chebyshev's inequality, we get the following corollary.

  \begin{corollary}\label{cor:U-stat}
  If $U$ is an almost surely finite Poisson $U$-statistic based on a non-negative symmetric kernel $h$, and there exists $a\ge 0$ and $\alpha \in [0,2)$ such that
  \begin{align}\label{eq:U-stat-assumption}
    \sum_{i} \Bigl(\sum_{i_1,\ldots,i_{m-1}\colon i_j\neq i}^{\neq} h(X_{i_1},\ldots,X_{i_{m-1}},X_i) \Bigr)^2 \le a U^{\alpha}
  \end{align}
  then  for any $r \ge 2$,
  \begin{displaymath}
    \|(U - \EE U)_+\|_r \le C\sqrt{r}m\sqrt{a} (\EE U)^{\alpha/2} + (C r m^2a)^{1/(2-\alpha)},
  \end{displaymath}
  where $C$ is some universal constant. As a consequence, for $t \ge 0$,
  \begin{displaymath}
    \PP(U \ge \EE U + t) \le 2\exp\Big(-\min\Big(\frac{t^2}{C' m^2 a (\EE U)^\alpha}, \frac{t^{2-\alpha}}{C' m^2 a}\Big)\Big),
  \end{displaymath}
  where $C'$ is some universal constant.
  \end{corollary}

  Let us remark that the references \cite{MR3485348,MR3849811,gozlan2020transport} provide also bounds on the left tail of $U$. It does not seem that such a bound can be easily recovered from the moment approach, since it relies heavily on an another property of Poisson $U$-statistics with non-negative kernels, namely an appropriate notion of convexity, which allows for an application of certain correlation inequalities~\cite{MR3485348} or the Poisson convex distance inequality~\cite{MR3849811,gozlan2020transport}. It is an interesting question what moment estimates can be obtained under an additional convexity assumption. We remark that for the usual notion of convexity on $\RR^n$, certain self-normalized moment estimates have been derived for all measures satisfying the convex concentration property~\cite{MR3892322}.

  The upper bounds on the upper tail of $U$, presented in the above references are
  \begin{align*}
    \PP(U \ge \EE U + t) &\le \exp\Big(-\frac{((\EE U + t)^{1-\alpha/2} - (\EE U)^{1-\alpha/2})^2}{2m^2 a}\Big),\\
    \PP(U \ge \Med U + t) &\le  2\exp\Big(-\frac{t^2}{4m^2a (t+ \Med U)^\alpha}\Big)\Big).
  \end{align*}
  where $\Med U$ is the median of $U$. One can show that their behaviour (disregarding the exact values of constants and using the fact that median and the mean of concentrated random variables are not far apart) is the same as of the upper bound of Corollary~\ref{cor:U-stat}.

  In~\cite{MR3473096} Bachmann and Reitzner verified the assumption~\eqref{eq:U-stat-assumption} for a general class of $U$-statistics of Poisson processes on $\RR^d$, with $\alpha = 2 - \frac{1}{m}$ and $a$ depending on the dimension $d$ and  certain parameters of the kernel. In particular they showed that this assumption is satisfied in the case when $U$ is the number of copies of a given connected graph $H$ on $m$ vertices in the Gilbert graph based on $\eta$. They also proved that the above bounds are of the right order as $t\to \infty$ and agree with known limit theorems if one increases the intensity of the process.

  This shows that the moment bounds of Proposition~\ref{prop:moments-Poisson} may be an alternative for proving exponential inequalities for the upper tail of geometric functionals. It is an interesting question, however beyond the scope of this article, to verify whether Proposition~\ref{prop:moments-Poisson} can give meaningful bounds in cases when functionals in question are known to have polynomial tails.

\section{Higher order concentration}\label{sec:higher-order}

We will now describe applications of our results to higher order concentration inequalities, which provide estimates on probabilities of deviations from the mean for not necessarily Lipschitz functions, expressed in terms of higher order derivatives. Such estimates were obtained, e.g., in \cite{MR3383337,MR3743923,MR3670792,gtze2018higher,MR3949267}, both in the discrete and continuous settings. Since the latter case follows in a straightforward way from results in~\cite{MR3383337} we will focus here mainly on the discrete case. We will however start with an abstract statement, and only later specialize it to concrete examples.

\subsection{Abstract inequality}
Let $\mathcal{A}$ be a linear space of functions on $\mathcal{X}$ and $D_i \colon \mathcal{A} \to \RR^\mathcal{X}$, $i = 1,\ldots,n$ be linear maps (we will think of them as directional derivatives). For positive integers $k, i_1,\ldots,i_k$ denote $D_{i_1\ldots i_k} f = D_{i_1}\cdots D_{i_k} f$, $\mathbf{D}^k f = (D_{i_1\ldots i_k} f)_{i_1,\ldots,i_k=1}^n$. Thus $\mathbf{D} = \mathbf{D}^1$ corresponds to the gradient and $\mathbf{D}^k$ for $k > 1$ to tensors of higher order derivatives, in particular $\mathbf{D}^k f(x) \in (\RR^n)^{\otimes k}$. For $x^1,\ldots,x^k \in \RR^n$, let $x^1\otimes\cdots \otimes x^k = (x^1_{i_1}\cdots x^k_{i_k})_{i_1,\ldots,i_k=1}^n \in (\RR^n)^{\otimes k}$.

Let us also define the inner product on $(\RR^n)^{\otimes k}$ with the formula
\begin{displaymath}
  \langle x, y\rangle = \sum_{i_1,\ldots,i_k = 1}^n x_{i_1\ldots i_k}y_{i_1\ldots i_k}.
\end{displaymath}

The following fact was proved in~\cite{MR3383337} in the case of the usual derivatives (see Proposition 3.2. therein). Since the easy proof is completely analogous to the one presented in~\cite{MR3383337} (it uses only linearity of $D_i$'s) we will skip it.
Below $\abs{\cdot}$ denotes the euclidean norm on $\RR^n$.

\begin{proposition}\label{prop:higher-order-moments} Let $X$ be an $\mathcal{X}$-valued random variable.
Assume that $r \ge 2$ and that there exists a constants $K$ such that for all $f \in \mathcal{A}$,
\begin{displaymath}
  \|f(X) - \EE f(X)\|_r \le K\Big\||\mathbf{D} f(X)|\Big\|_r.
\end{displaymath}
Then for any integer $d$ and any function $f\colon \mathcal{X} \to \RR$ in the domain of $D_{i_1\ldots i_d}$, $i_1,\ldots,i_d \le n$ such that $D^df(X)\in L_r$,
\begin{align*}
  \| f(X) - \EE f(X)\|_r \le&  \frac{C^d K^d}{r^{d/2}}\| \langle \mathbf{D}^d f(X),G_1\otimes\cdots\otimes G_d\rangle\|_r \\
  &+ \sum_{k=1}^{d-1} \frac{C^k K^k}{r^{k/2}}\| \langle \EE_X \mathbf{D}^k f(X),G_1\otimes\cdots\otimes G_k\rangle\|_r,
\end{align*}
where $G_1,\ldots,G_d$ are i.i.d. standard Gaussian vectors in $\RR^n$, independent of $X$,
 and $C$ is a universal constant.
\end{proposition}

If $\mathbf{D}^d$ is uniformly bounded and $X$ satisfies Beckner's inequality, then one can combine the above proposition with moment estimates of Theorem~\ref{prop:V^2} and inequalities for multilinear forms in i.i.d. Gaussian vectors, obtained by Lata{\l}a~\cite{MR2294983}, which we will now recall.

Let us start by introducing the (rather involved) notation. For $\ii = (i_1,\ldots,i_d) \in [n]^d$ and $I\subseteq[d]$ we will write $\ii_{I}=(i_{k})_{k\in I}$. We will also denote $|\ii| = \max_{j\le d} {i_j}$ and $|\ii_I| = \max_{j \in I} i_j$.  Let moreover $P_d$ be the set of partitions of $[d]$ into nonempty, pairwise disjoint sets. For a partition $\mathcal{I} =\{I_1,\ldots,I_k\} \in P_d$ and a $d$-indexed matrix $A = (a_\ii)_{\ii \in [n]^d}$, define

\begin{align}\label{eq:Gaussian_norm_def}
\|A\|_{\mathcal{I}}=\sup\Big\{\sum_{\ii\in [n]^d} a_{\ii}\prod_{l=1}^k x^{l}_{\ii_{I_l}}\colon
\|(x^{l}_{\ii_{I_l}})\|_2\leq 1, 1\leq l\leq k \Big\},
\end{align}

where $\|(x_{\ii_{I_l}})\|_2 = \sqrt{\sum_{|\ii_{I_l}|\le n} x_{\ii_{I_l}}^2}$. Thus, e.g.,
\begin{align*}
&\|(a_{ij})_{i,j\le n}\|_{\{1,2\}}= \sup\Big\{ \sum_{i,j\le n} a_{ij}x_{ij}\colon \sum_{i,j\le n} x_{ij}^2 \le 1\Big\} = \sqrt{\sum_{i,j\le n}a_{ij}^2} = \|(a_{ij})_{i,j\le n}\|_{\HS},\\
& \|(a_{ij})_{i,j\le n}\|_{\{1\}\{2\}}= \sup\Big\{ \sum_{i,j\le n} a_{ij}x_iy_j\colon \sum_{i\le n} x_{i}^2\le 1,\sum_{j\le n}y_j^2 \le 1\Big\} \\
\MoveEqLeft[-26.5]= \|(a_{ij})_{i,j\le n}\|_{\ell_2^n\to \ell_2^n},\\
& \|(a_{ijk})_{i,j,k\le n}\|_{\{1,2\} \{3\}} = \sup\Big\{ \sum_{i,j,k\le n} a_{ijk}x_{ij}y_k\colon \sum_{i,j\le n} x_{ij}^2\le 1,\sum_{k\le n}y_k^2 \le 1\Big\},
\end{align*}
where for simplicity in the notation we skip the outer brackets and commas in the subscript and write, e.g., $\|\cdot\|_{\{1\}\{2\}}$ instead of $\|\cdot\|_{\{\{1\},\{2\}\}}$.

In the special case of $d=2$, $\|\cdot\|_{\{1,2\}}$ and $\|\cdot\|_{\{1\}\{2\}}$ are just the Hilbert--Schmidt and operator norms of a matrix. We remark that for every $d$ and $\mathcal{I} \in P_d$ we have $\|A\|_{\mathcal{I}} \le \|A\|_{\{[d]\}} = \sqrt{\sum_{|\ii|\le n} a_\ii^2}$. The norm $\|A\|_{\{[d]\}}$ can be considered a counterpart of the Hilbert--Schmidt norm for higher order tensors.

The result by Latała provides bounds on moments of multilinear forms in i.i.d. Gaussian variables in terms of the $\mathcal{I}$-norms of the corresponding matrix of coefficients.

\begin{theorem}[\cite{MR2294983}]\label{thm:Latala} Let $G_1,\ldots,G_k$ be independent standard Gaussian vectors in $\RR^n$ and let $A \in (\RR^n)^{\otimes k}$. There exist constants $C_k$, depending only on $k$, such that  for any $r\ge 2$,
\begin{displaymath}
  \frac{1}{C_k} \sum_{\mathcal{I}\in P_k} r^{|\mathcal{I}|/2}\|A\|_\mathcal{I} \le \|\langle A,G_1\otimes\cdots\otimes G_k\rangle\|_r\le C_k\sum_{\mathcal{I}\in P_k} r^{|\mathcal{I}|/2}\|A\|_\mathcal{I},
\end{displaymath}
\end{theorem}

Combining this result with Proposition~\ref{prop:higher-order-moments}, we obtain the following corollary.

\begin{corollary}\label{cor:higher-order-tails}
   Assume that there exist constants $M,\gamma > 0$ such  that for all functions $f \in \mathcal{A}$, and all $r \ge 2$,
   \begin{align}\label{eq:momen-assumption}
     \| f(X) - \EE f(X)\|_r \le M r^{\gamma} \Big\||\mathbf{D} f(X)|\Big\|_r.
   \end{align}
   Then for every integer $d \ge 1$, any $r \ge 2$ and for every $f$ in the domain of $\mathbf{D}^d$ such that $D^df(X)\in L_r$,

\begin{align*}
\|f(X) - \EE f(X)\|_r \le& C_d\Big(M^d\sum_{\mathcal{J} \in P_d} r^{(\gamma - \frac{1}{2})d + \frac{|\mathcal{J}|}{2}} \Big\|\|\mathbf{D}^d f(X)\|_\mathcal{J}\Big\|_r\\
&\qquad+ \sum_{k=1}^{d-1} M^k \sum_{\mathcal{J}\in P_k} r^{(\gamma -
\frac{1}{2})k  + \frac{|\mathcal{J}|}{2}} \|\EE \mathbf{D}^k f(X)\|_\mathcal{J}\Big),
\end{align*}
where $C_d$ depends only on $d$. Moreover if $\mathbf{D}^df (x)$ is uniformly bounded on $\mathcal{X}$, then
for $t > 0$,
\begin{displaymath}
\PP(|f(X) - \EE f(X)| \ge t ) \le 2\exp\Big(-\frac{1}{C_d'} \eta_f(t)\Big),
\end{displaymath}
where $C_d'$ is another constant depending only on $d$ and
\begin{displaymath}
  \eta_f(t) = \min(A,B)
\end{displaymath}
with
\begin{align*}
A & =  \min_{\mathcal{J}\in P_d}\Big(\frac{t}{M^{d}\sup_{x\in \mathcal{X}}\|\mathbf{D}^d f(x)\|_\mathcal{J}}\Big)^{\frac{2}{(2\gamma-1)d+ |\mathcal{J}|}},\\
B & = \min_{1\le k\le d-1}\min_{\mathcal{J}\in P_k} \Big(\frac{t}{M^{k}\|\EE \mathbf{D}^k f(X)\|_\mathcal{J}}\Big)^{\frac{2}{(2\gamma -1)k + |\mathcal{J}|}}.
\end{align*}
\end{corollary}

\begin{proof}
To obtain the moment estimate we combine Proposition~\ref{prop:higher-order-moments} with $K = Mr^\gamma$ and Theorem~\ref{thm:Latala}. The second part follows from the first one by an application of Chebyshev's inequality for the $r$-th moment and optimization in $r$.
\end{proof}
A typical application of the above corollary is the situation in which $\gamma = 1/2$ (such a subgaussian bound holds by Proposition~\ref{prop:moments-abstract}, e.g., under the assumption of modified log-Sobolev inequalities) and $f$ has bounded derivatives of second order. The tail bound one obtains is then
\begin{align*}
&\PP(|f(X) - \EE f(X)| \ge t) \\
&\le 2\exp\Big(-c\min\Big(\frac{t^2}{\sup_{x \in \mathcal{X}} \|\mathbf{D}^2 f(x)\|_{\HS}^2  + |\EE \mathbf{D} f(X)|^2},\frac{t}{\sup_{x\in \mathcal{X}} \|\mathbf{D}^2 f(x)\|_{\ell_2\to\ell_2}}\Big)\Big).
\end{align*}
Estimates of this type are counterparts of the well known Hanson--Wright inequality for quadratic forms in independent subgaussian random variables (see~\cite{MR0279864,MR1686370}).

Let us also mention that if $\mathcal{X} = \RR^n$, $D_i$'s correspond to the usual partial derivatives, $X$ is a standard Gaussian vector and $f$ is a polynomial of degree $d$ then the inequalities of Corollary~\ref{cor:higher-order-tails} can be reversed (up to constants depending on $d$)~\cite{MR3383337}. The fact that in this case the assumptions of the corollary are satisfied was proved for the first time by Maurey and Pisier (see, e.g.,~\cite{MR1036275}). Other continuous type examples for which this assumption is satisfied are described in~\cite{MR3383337, MR3743923} by means of various types of modified log-Sobolev inequalities corresponding to tail behaviour between exponential and Gaussian. As announced, we will not describe in details such examples. Instead we will now focus on the discrete case and discuss a general situation, related to applications considered in previous sections, in which one can find gradients $D_i$ satisfying~\eqref{eq:momen-assumption}.

\subsection{Discussion on the choice of gradients}
We will consider the following setting. Let $G$ be a group with a set of generators $g_1,g_2,\ldots,g_m$, acting on a countable set $\calX$ (we will denote the result of the action of $g \in G$ on $x\in \calX$ simply by $gx$). Assume that $g_1,\ldots,g_m$ are pairwise distinct, distinct from the neutral element of $G$ (denoted by $e$) and that no two distinct elements  among the $g_i$'s are reciprocal to each other.

Let $\mu$ be a probability measure on $\calX$. Set $A = \{g_i,g_i^{-1}\colon i \le m\}$ and let $ \lambda \colon \calX \times A\to [0,\infty)$ be a function satisfying the detailed balance condition
\begin{align}\label{eq:detailed-balance-group}
 \lambda(x,g)\mu(x) = \lambda(gx, g^{-1})\mu(gx)
\end{align}
for $x \in \calX$, $g \in A$.

Finally consider the Markov process with the generator
\begin{displaymath}
  L f (x) = \sum_{g \in A} (f(gx) - f(x))\lambda(x,g)
\end{displaymath}
and the corresponding Dirichlet form
\begin{displaymath}
  \calE(f,h)  = \frac{1}{2} \sum_{x\in \calX} \sum_{g \in A} (f(gx) - f(x))(h(gx) - h(x))\lambda(x,g) \mu(x) = \EE \Gamma(f,h)
\end{displaymath}
with
\begin{align*}
  \Gamma(f,h)(x) &= \frac{1}{2} \sum_{g \in A} (f(gx) - f(x))(h(gx) - h(x))\lambda(x,g) \\
  &= \frac{1}{2} \int_\calX (f(y) - f(x))(h(y) - h(x))Q_x(dy),
\end{align*}
where
\begin{displaymath}
Q_x(y) = \sum_{g \in A \colon gx = y}{\lambda}(x,g)
\end{displaymath}
Moreover the pair $(Q, \mu)$ satisfies the detailed balance condition.
Therefore if the form $\calE$ satisfies Beckner's inequality with $\alpha_p \ge a(p-1)^s$, for $p \in (1,2]$, then by Proposition~\ref{prop:moments-abstract} for all functions $f \colon \calX \to \RR$, and all $r \ge 2$
\begin{align}\label{eq:higher-order-moment-assumption}
  \|f - \mu (f)\|_r \le  Kr^{\frac{1+s}{2}} \Big\| \Big(\sum_{g \in A} (f(gx) - f(x))^2 \lambda(x,g)\Big)^{1/2}\Big\|_r,
\end{align}
for $K = \sqrt{\kappa(s)/2a}$.

The above inequality allows for a direct use of Corollary~\ref{cor:higher-order-tails} with gradients $D_g f(x)= (f(gx) - f(x))\sqrt{\lambda(x,g)}$. This  choice however may have some disadvantages from the point of view of higher order concentration, especially when one deals with finite groups. To illustrate this let us focus on the situation when for some $g \in A$, $g^2 = e$ (in the sequel we will discuss natural examples when this is true for all elements of $A$). One then gets
\begin{align*}
D_{g}D_{g} f(x) &= D_{g} (f(gx) - f(x))\sqrt{\lambda(x,g)}\\
&= \bigl((f(x) - f(gx))\sqrt{\lambda(gx,g)}  - (f(gx) - f(x))\sqrt{\lambda(x,g)}\bigr)\sqrt{\lambda(x,g)} \\
&= (f(x) - f(gx))(\sqrt{\lambda(x,g)}+\sqrt{\lambda(gx,g)})\sqrt{\lambda(x,g)}.
\end{align*}
In particular, if $M$ is the constant from Corollary~\ref{cor:higher-order-tails}, it may happen that
\begin{displaymath}
  M^2\|\mathbf{D}^2 f(x)\|_{\{1,2\}} \ge M \|\mathbf{D} f(x)\|_{\{1\}} = M|\mathbf{D} f(x)|
\end{displaymath}
and so Corollary~\ref{cor:higher-order-tails} applied with $d=2$ is strictly weaker than its assumption corresponding to $d=1$ (while the goal of introducing second order concentration is handling functions for which first order bounds are too conservative). Also, in certain situations, especially when dealing with a class of processes or when one does not have full knowledge about the transition rates, one may want to have a notion of gradient, which depends only on the function $f$ and not on the rates $\lambda(x,g_i)$.  For these reason one may want to replace the natural choice of the gradient with another one. We will now briefly discuss some possibilities.

Let us view the set $\mathcal{X}$ with the action $G$ as a graph, i.e., define the set of edges $E = \{\{x,y\}\colon x,y \in \mathcal{X}, \exists_{i\le m} y = g_ix \textrm{ or } x = g_iy\}$. Impose also an arbitrary orientation on the edges, by choosing functions $s,t \colon E \to \mathcal{X}$ such that for all $\{x,y\} \in E$, $\{s(x,y),t(x,y)\} = \{x,y\}$. Then one can define for $g \in A$,
\begin{displaymath}
D_g f(x) = \bigl(f(t(x,gx))- f(s(x,gx))\bigr)\sqrt{\max(\lambda(x,g),\lambda(gx,g^{-1}))}.
\end{displaymath}
Clearly, by ~\eqref{eq:higher-order-moment-assumption}, we then have
\begin{displaymath}
  \|f(X) - \EE f(X)\|_r \le  Kr^{\frac{1+s}{2}} \Big\||\mathbf{D} f(X)|\Big\|_r.
\end{displaymath}
Moreover for $g^2 = e$, $D_gf(x) = D_g f(gx)$ and so $D_gD_g f(x) = 0$. If $\lambda^\ast := \sup_{x\in\mathcal{X}}\max_{g\in A} \lambda(x,g) < \infty$ one can also take
\begin{displaymath}
\widetilde{D}_g f(x) = f(t(x,gx))- f(s(x,gx)),
\end{displaymath}
obtaining a gradient independent of the transition kernel, at the cost of changing the constant $K$ by a factor $\sqrt{\lambda^\ast}$. Such situation may happen especially in the finite case, when the Markov semigroup is obtained by embedding in continuous time a discrete time Markov chain as, e.g., in the case of Glauber dynamics (cf.\ Section~\ref{sec:Glauber}).

For instance if $\mathcal{X} = \{-1,1\}^m$ and $g_i$'s act on $\mathcal{X}$ by flipping the $i$-th coordinate then all $g_i$'s satisfy $g_i^2 = e$. In this case it is natural to choose $t(x,y) = \max(x,y)$, $s(x,y) = \min(x,y)$ where $\max,\min$ are taken with respect to the lexicographic order. It is then easy to see that in this case $\widetilde{D}_{g_i} f$ coincides, up to a factor of 2, with the usual partial derivative of the polynomial corresponding to the Fourier-Walsh representation of the function $f$. The article~\cite{MR3949267} uses the strategy described above to obtain counterparts of Lata{\l}a's inequalities for Gaussian polynomials for polynomials in Ising models satisfying the Dobrushin condition discussed in Section~\ref{sec:Glauber}. We will generalize the inequalities obtained therein in Corollary~\ref{cor:polynomials} below.

Another situation in which all the $g_i$'s are of order two is related to transpositions, and the action of the symmetric group, corresponding to moment inequalities described in Propositions~\ref{prop:symmetric-group-moments} and~\ref{prop:SCP} (iii). In this case it is also natural to use the lexicographic order to define the functions $t$ and $s$. Note that in this case we have $\binom{n}{2}$ directional derivatives along transpositions. For instance in the setting of Proposition~\ref{prop:symmetric-group-moments}, one has

\begin{displaymath}
  D_{\tau_{ij}} f(x) = \frac{1}{n-1}\Big(f(\max(\sigma\circ \tau_{ij},\sigma)) - f(\min(\sigma\circ \tau_{ij},\sigma))\Big).
\end{displaymath}
In the more general situation of Proposition~\ref{prop:SCP} (ii) one can consider the action of the group $\mathbb{Z}_2^n\times S_n$.
One may also consider actions of infinite groups, for instance $\mathbb{Z}$ or more generally $\mathbb{Z}^n$, as Beckner and log-Sobolev inequalities
are well known for various measures on $\mathbb{Z}$, e.g., for stationary measures of various birth-and-death processes (see \cite{MR3129037,MR1944012,MR2548501,MR1710983,MR2052235})

\subsection{Applications to tetrahedral polynomials}

Let us now provide an application of Corollary~\ref{cor:higher-order-tails} to tetrahedral polynomials of random vectors with values in a cube, say $[-1,1]^n$, for which the corresponding Glauber dynamics satisfies~\eqref{eq:mlS-log} (as discussed in Section~\ref{sec:Glauber}). This will generalize results in~\cite{gtze2018higher,MR3949267} concerning the Ising model and results from~\cite{gtze2018concentration} concerning exponential random graphs (and by the discussion in Section~\ref{sec:Glauber} will allow also for a slight strengthening of the dependence of constants on the parameters in these models). Recall that a polynomial $f \colon \RR^n \to \RR$ is tetrahedral, if it is affine in every variable, i.e. it is of the form
\begin{displaymath}
  f(x_1,\ldots,x_n) = \sum_{k=0}^d \sum_{I \subseteq [n], |I| = k} a_I \prod_{j\in I} x_i.
\end{displaymath}

Assume that $X = (X_1,\ldots,X_n)$ is distributed according to a measure $\mu$ on $[-1,1]^n$, which satisfies~\eqref{eq:mlS-log} with constant $\rho_0 > 0$ for the Glauber dynamics. Recall Corollary~\ref{cor:Glauber}. Let $\mathcal{A}$ be the linear space of tetrahedral polynomials and note that for any $f \in \mathcal{A}$, the inequality $|X_i-X_i'|\le 2$ implies that
\begin{displaymath}
|f(X) - f(X^i)| \le 2 \Big|\frac{\partial f}{\partial x_i} (X)\Big|.
\end{displaymath}
As a consequence, by Corollary~\ref{cor:Glauber}, the assumptions of Corollary~\ref{cor:higher-order-tails} are satisfied with $\mathbf{D} = \nabla$, $\gamma= 1/2$ and $K = C\rho_0^{-1}$ for some universal constant $C$.
Noting that partial derivatives of tetrahedral polynomials are tetrahedral, and for a polynomial $f$ of degree $d$, $\nabla^d f$ is constant, we obtain

\begin{corollary}\label{cor:polynomials} Assume that $\mu$ is a probability measure on $[-1,1]^n$, satisfying the inequality~\eqref{eq:mlS-log} with $\rho_0 > 0$ for the Glauber dynamics. Let $X$ be a random vector with law $\mu$ and let $f \colon \RR^n \to \RR$ be a tetrahedral polynomial of degree $d$. Then for any $t > 0$,
\begin{align}\label{eq:polynomials}
\PP\Big(\Big|f(X) - \EE f(X)\Big| \ge t\Big)\le 2\exp\Big(-\frac{1}{C_d}\min_{1\le k \le d} \min_{\mathcal{J}\in P_k}\Big(\frac{\rho_0^{k/2} t}{\|\EE \nabla^k f(X)\|_{\mathcal{J}}}\Big)^{2/|\mathcal{J}|}\Big).
\end{align}
\end{corollary}

In particular the above corollary applies to the Ising model, exponential random graphs and hardcore models under the assumptions presented in Section~\ref{sec:Glauber}. Note that in these cases (or more generally for measures supported on $\{-1,0,1\}^n$) every polynomial can be reduced to a tetrahedral one.

Let us illustrate the above corollary with an application to triangle count in random graphs. Consider a simple random graph $G = (V,E)$, where $|V| = n$. For distinct vertices $v,w \in V$, let $X_{v,w} = \ind{\{v,w\}\in E}$. Then, the number of triangles in $G$ can be written as $T = \frac{1}{6}\sum_{u,v,w \in V}^{\neq} X_{u,v}X_{v,w}X_{w,u}$. The problem of tail behaviour for subgraph counts in Erd\H{o}s--R\'enyi random graphs has a long history, and a lot of progress has been made recently in the large deviation regime (see, e.g.,~\cite{MR3519474,MR3632418}). Here, we would like to illustrate how Corollary~\ref{cor:polynomials} may be used to obtain bounds for the non-product case.

Assume that $G$ is exchangeable in the following sense: for any permutation $\sigma\colon V\to V$, the distribution of the random set $E_\sigma = \{\{\sigma(v),\sigma(w)\} \colon \{v,w\}\in E\}$ is the same as that of $E$ (in other words, the adjacency matrix of $G$ has distribution invariant under a simultaneous permutation of rows and columns).  Assuming that $n\ge 3$ and $V = [n]$, let us define $A = \EE X_{1,2}$, $B = \EE X_{1,2}X_{2,3}$ (i.e., $A$ is the probability of finding an edge and $B$ is the probability of finding a \emph{cherry} at a fixed place in $G$). If the random vector $(X_{v,w})_{1\le v<  w \le n}$ satisfies~\eqref{eq:mlS-log}, then one can estimate the $\mathcal{I}$-norms appearing in Corollary~\ref{cor:polynomials} to find a tail estimate for $T$. Note that the expected derivatives of $T$ are expressed in terms of $A$ and $B$.

In the Erd\H{o}s--R\'enyi case such estimates were obtained in~\cite{MR3383337} and combined with a slight refinement of~\eqref{eq:polynomials}, specific to the product situation which allowed to replace $\rho_0$ with an appropriate subgaussian norm.
Surprisingly, the inequality obtained from such a general approach turned out to be optimal in the large deviation regime for edge probability $p \ge(n^{1/4}\sqrt{\log(n)})^{-1}$.
In~\cite{gtze2018concentration} the calculations from~\cite{MR3383337} were combined with~\eqref{eq:mlS-sqrts} in the case of exponential random graphs (cf.\ Section~\ref{sec:Glauber}). The dependence on $\rho_1$ is not specified there, but one can easily derive it from the proof. Corollary~\eqref{cor:polynomials} allows to  relax the dependence on $\rho_1$ to dependence on $\rho_0$, leading to the following estimate (we skip the detailed calculations, which are based on those from~\cite{MR3383337} and present just the final result):
\begin{multline*}
  \PP(|T - \EE T| \ge t) \\
  \le 2\exp\Big(-\frac{1}{C}\min\Big(\frac{t^2}{n^3 (\rho_0^{-3} + \rho_0^{-2} A^2) + n^4\rho_0^{-1} B^2}, \frac{t}{\sqrt{n}\rho_0^{-3/2} + n \rho_0^{-1} A}, \frac{t^{2/3}}{\rho_0^{-1}}\Big)\Big)
\end{multline*}
for $t \ge 0$ and a universal constant $C$.

\appendix
\section{Auxiliary lemmas}\label{app-auxiliary-lemmas}
In this section we provide proofs of technical lemmas used in Section~\ref{sec:mLStoBec}.

\begin{proof}[Proof of Lemma~\ref{L:Bec=>finite-mean}]
	Consider a nonnegative  $f$ such that $f,f^{p-1}\in Dom(\calE)$.
	For $t > 0$, let $f_t$ denote $\min(f,t)$.
	The function $x\mapsto \min(x,t)$ is a contraction, whence $f_t,f_t^{p-1}\in Dom(\calE)$.
	Moreover for any $t>0$ and any non-decreasing function $\varphi$
	\[
		(f_t(x)-f_t(y))(\varphi(f_t(x))-\varphi(f_t(y))) \le (f(x)-f(y))(\varphi(f(x))-\varphi(f(y))),
	\]
	so by Assumption~\ref{a:I},
	\[
		\sup_{t>0}\calE(f_t,f_t^{p-1})\le \calE(f,f^{p-1}).
	\]
	It remains to show that $\mu(f^p)-\mu(f)^p$ is well-defined and is the limit of $\mu(f_t^p)-\mu(f_t)^p$ as $t\to\infty$. By the Lebesgue monotone convergence theorem, it suffices to show that $\mu(f)<\infty$.
	To this end it is enough to show that $\sup_t \mu(f_t) < \infty$. This inequality is trivially satisfied if $\sup_t \mu(f_t)^p \le \calE(f,f^{p-1})$. Assume thus that
	$\mu(f_{t})^p > \calE(f,f^{p-1})$ for some $t$.
	Applying Beckner's inequality~\eqref{eq:Beckner-GJ} to the function $f_{t}$, we get that $\frac{1+\alpha_p}{\alpha_p}\mu(f_{t})^p \ge \mu(f_{t}^p)$ (recall that $p/2\leq 1$).
	Whence, by the Paley--Zygmund inequality (see, e.g., \cite[Corollary 3.3.2]{MR1666908})
	\[
		\mu\Big(f > \frac{1}{2}\mu(f_{t})\Big)
		\ge
		\mu\Big(f_{t} > \frac{1}{2}\mu(f_{t})\Big)
		\ge
		\Bigl(\frac{1}{2^p}\frac{\mu(f_{t})^p}{\mu(f_{t}^p)}\Bigr)^{\frac{1}{p-1}}
		\ge
		2^{-\frac{p}{p-1}} \Bigl(\frac{\alpha_p}{1+\alpha_p}\Bigr)^{\frac{1}{p-1}}.
\]
	Thus $\frac{1}{2}\mu(f_t)$ must be bounded by an appropriate quantile of $f$ and so $\sup_{t>0} \mu(f_t)$ is finite as desired.
\end{proof}

\begin{proof}[Proof of Lemma~\ref{L:v'formula}]
	Fix any $s\ge 1$.	
	Lipschitz property of the appropriate maps on the interval $[\inf f,\sup f]$ and Assumption~\ref{a:I} imply that for any $u$, $f^{u-1},f^{u-1}\log f\in \Dom(\calE)$.
	Denote $h_\varepsilon(x) = \frac{x^{\varepsilon}-1}{\varepsilon}-\log x$ and assume without loss of generality that $1-s\le \varepsilon$.	
	Then $f^{s-1}h_\varepsilon(f)\in \Dom(\calE)$ (since $\Dom(\calE)$ is a linear space) and by the bilinearity of $\calE$
	\begin{equation*}\label{eq:v'}
			\frac{v(s+\varepsilon)-v(s)}{\varepsilon} = \calE(f,f^{s-1}\log f) + \calE(f,f^{s-1}h_\varepsilon(f)).
	\end{equation*}
	
	We will show that
	\[
		\abs{x^{s-1}h_\varepsilon(x) - y^{s-1}h_\varepsilon(y) }
		\le
		C\abs{\varepsilon}\cdot\abs{x-y}
	\]
	for all $x,y\in [\inf f,\sup f]$ and some positive constant $C$ dependent on $s$ and $f$.
	By Assumption~\ref{a:I}, this will allow to conclude that $\abs{\calE(f,f^{s-1}h_\varepsilon(f))}\le C\abs{\varepsilon}\cdot\calE(f,f)\to 0$ as $\varepsilon\to 0$.

	By the Taylor expansion of $x^\varepsilon$ in $\varepsilon$ with the integral form of the remainder
	\[
		h_\varepsilon(x) = \frac{1}{\varepsilon}\int_0^\varepsilon (\varepsilon-u)x^u\log^2x\,du,
	\]
	so
	\[
		x^{s-1}h_\varepsilon(x) - y^{s-1}h_\varepsilon(y)  =
		\frac{1}{\varepsilon}\int_0^\varepsilon ({g}_u(x)-{g}_u(y))(\varepsilon-u)\,du,
	\]
	where ${g}_u(x)=x^{u+s-1}\log^2(x)$. 
	We can also assume without loss of generality that $\varepsilon\le 1$.
	Recalling that $f$ is bounded and separated from zero, it follows that $\sup_{u,x}\frac{1}{2}{\abs{g_u'(x)}} =: C < \infty$.
	The proof is concluded by taking the absolute value, passing with it under the integral and estimating the increments of $g_u$.

\end{proof}

\section{Known implications between functional inequalities}
\label{app-implications}

In this section we provide sketches of proofs of previously known implications between functional inequalities discussed in Section~\ref{sec:SOA}. Although in the literature the results are commonly stated in the Markov kernel setting (sometimes only in the finite state space or continuous case), their proofs are mostly based on pointwise inequalities which imply comparison of Dirichlet forms. Hence by virtue of Assumption~\ref{a:I} they pass directly to our setting at the cost of adding just a few technical details (needed mostly in order to make sure that all quantities are well-defined). We remark that a section containing implications between the Poincar\'e, log-Sobolev, and modified log-Sobolev inequality in an abstract setting, but under a somewhat different set of assumptions than ours, can be found in the article~\cite{MR2283379} by Bobkov and Tetali. In all the statements below we assume that Assumption~\ref{a:I} is satisfied.

Let us start with the implication between~\eqref{eq:mlS-sqrts} and~\eqref{eq:mlS-log}.
\begin{proposition}
\label{prop:LSI-to-mLSI}
If the log-Sobolev inequality~\eqref{eq:mlS-sqrts} holds with some constant $\rho_1$, then the modified log-Sobolev inequality~\eqref{eq:mlS-log} holds with $\rho_0 \ge 4\rho_1$.
\end{proposition}

\begin{proof}
The argument is based on the following pointwise inequality introduced by Bobkov and Tetali~\cite{MR2283379}:
\begin{align}\label{eq:pointwise-LSI-mLSI}
 4\bigl(\sqrt{a}-\sqrt{b}\bigr)^2 \leq (a-b)(\log(a)-\log(b)), \quad a,b>0.
\end{align}

Assuming that~\eqref{eq:mlS-sqrts} holds for all $g \in \Dom(\calE)$, let us consider $f$, such that $f, \log f \in \Dom(\calE)$. Denote $f_\varepsilon = \max(f,\varepsilon)$. By Assumption~\ref{a:I}, the fact that $x \mapsto \max(x,\varepsilon)$ is a contraction and monotonicity of the $\log$, one can easily see that $\calE(f_\varepsilon,\log(f_
\varepsilon)) \le \calE(f,\log(f))$. Taking into account that $\Ent_\mu(f) = \lim_{\varepsilon \to 0} \Ent_\mu(f_\varepsilon)$, one can thus assume that $f$ is separated from zero. Thus, again by Assumption~\ref{a:I}, and Lipschitz property of the quare root away from zero, $g = \sqrt{f} \in \Dom(\calE)$, and $\rho_1 \Ent_\mu(f) \le \calE(\sqrt{f},\sqrt{f}) \le \frac{1}{4}\calE(f,\log(f))$, where in the last inequality we used~\eqref{eq:pointwise-LSI-mLSI} and Assumption~\ref{a:I}.
\end{proof}

Let us now pass to the relation between~\eqref{eq:Beckner-LO} and~\eqref{eq:mlS-sqrts}.

\begin{proposition}
If Beckner's inequality~\eqref{eq:Beckner-LO} holds for every $q\in[1,2)$ with $\beta_q$ bounded away from zero, then the log-Sobolev inequality~\eqref{eq:mlS-sqrts} holds as well with $\rho_1 \geq \frac{1}{2}\limsup_{q\to 2^-}\beta_q$. Conversely if~\eqref{eq:mlS-sqrts} holds, then so does~\eqref{eq:Beckner-LO} for every $q \in [1,2)$, with $\beta_q \ge q\rho_1$.
\end{proposition}

\begin{proof}To obtain the first part of the proposition, it is enough to apply~\eqref{eq:Beckner-LO} to $|g|$ (note that $|g| \in \Dom(\calE)$ by Assumption~\ref{a:I}), divide by $2-q$ and pass to the limit $q\to 2$, observing that
\[
  \frac{\mu(g^2) - \mu(g^q)^{2/q}}{2-q}   \longrightarrow \frac{1}{2}\Ent_\mu(g^2)
  \quad\text{as}\quad q\to 2^-.
  \]
One obtains $(2^{-1}\limsup_{q\to 2-} \beta_q)\Ent_\mu (g^2) \le \calE(|g|,|g|) \le \calE(g,g)$, where the last inequality follows from another application of Asumption~\ref{a:I}.

The second part follows from a lemma proved in~\cite{MR1796718}, which asserts that if $g \in L_2(\calX,\mu)$, then
the function
\[
 [1,2)\ni q \longmapsto  \frac{\mu(g^2) - \mu(g^q)^{2/q}}{1/q-1/2}
\]
is increasing.
Note that in our setting square integrability of $g$ for $g \in \Dom(\calE)$ is a part of the assumption~\eqref{eq:mlS-sqrts}.
\end{proof}

\begin{proposition}
If the Beckner inequality~\eqref{eq:Beckner-GJ} holds for every $p\in(1,2]$ with $\alpha_p$ bounded away from zero, then the modified log-Sobolev inequality~\eqref{eq:mlS-log} holds as well with $\rho_0 \geq 2\limsup_{p\to 1^+}\alpha_p$.
\end{proposition}

\begin{proof}
Let us consider $f$ such that $f,\log(f) \in \Dom(\calE)$. Additionally, let us assume that $f$ is bounded and separated from zero. In particular $f^p \in \Dom(\calE)$ for $p \ge 1$. Taking the right derivative of the function $p\mapsto \mu(f^p) - \mu(f)^p$ at $p=1$ and using~\eqref{eq:Beckner-GJ} together with Lemma~\ref{L:v'formula}, we obtain~\eqref{eq:mlS-log} with $\rho_0 \ge 2\limsup_{p\to 1+} \alpha_p$.

It remains to remove the additional assumptions on $f$. The assumption that $f$ is separated from zero can be removed in the same way as in the proof of Proposition~\ref{prop:LSI-to-mLSI}. Let us therefore focus on the boundedness assumption. The argument is a variation of the one used in Lemma~\ref{L:Bec=>finite-mean}. Setting $f_t = \min(f,t)$, by Assumption~\ref{a:I} we have $\calE(f_t,
\log(f_t)) \le \calE(f,\log(f)) < \infty$, so it remains to show that $\sup_{t>0} \mu(f_t) < \infty$, as it will prove integrability of $\mu$, which will allow to pass to the limit with $t \to \infty$ in~\eqref{eq:mlS-log} for $f_t$. Let $\varphi\colon [0,\infty)\to [0,\infty)$ be any increasing convex function, $\varphi(0) = 0$, such that for large $x$, $\varphi(x) = x\log x$. Set $g_t = f_t/\mu(f_t)$. Since $x\log x$ is bounded from below, and $\Ent_\mu(f_t) = \mu(g_t \log(g_t))$, we obtain $\sup_{t > 0} \mu(\varphi(g_t)) < \infty$. Thus $\lim_{\delta \to 0}\sup_{t>0} \varphi(\delta g_t) = 0$. Let $\delta < 1/4$ be such that for all $t > 0$, $\mu(\varphi(\delta g_t)) < 1/4$. Denoting by $\varphi^\ast$ the Legendre transform of $\varphi$, given by $\varphi^\ast(y) = \sup_{x \ge 0} (xy - \varphi(x))$, we have
\begin{align*}
1 = \mu(g_t) \le \delta + \mu(g_t\ind{\{g_t > \delta\}}) \le \delta + \mu(\varphi(\delta g_t)) + \varphi^\ast(\delta^{-1})\mu(g_t \ge \delta),
\end{align*}
which gives
\begin{displaymath}
  \mu(f \ge \delta\mu(f_t)) \ge \mu(f_t \ge \delta\mu(f_t)) = \mu(g_t \ge \delta) \ge \frac{1}{2\varphi^\ast(\delta^{-1})} > 0.
\end{displaymath}
This shows that $\sup_{t> 0} \mu(f_t)$ is dominated by an appropriate quantile of $f$ and is therefore finite, ending the proof.
\end{proof}

Let us now pass to the relation between Beckner's inequalities~\eqref{eq:Beckner-GJ} and~\eqref{eq:Beckner-LO}.

\begin{proposition}
Let $p \in (1,2]$ and $q = 2/p \in [1,2)$.
If Beckner's inequality~\eqref{eq:Beckner-LO} holds with constant $\beta_q$ then Beckner's inequality~\eqref{eq:Beckner-GJ} holds with constant $\alpha_p \ge \beta_q$. Conversely, if~\eqref{eq:Beckner-GJ} holds with constant $\alpha_p$, then~\eqref{eq:Beckner-LO} holds with constant
$\beta_q \ge q(2-q)\alpha_p$.
\end{proposition}

\begin{proof}
  Assume first~\eqref{eq:Beckner-LO}. By Lemma~\ref{L:Bec=>finite-mean} in order to prove~\eqref{eq:Beckner-GJ} it is enough to show that it holds for bounded $f$. Assume thus that $f$ is bounded and $f,f^{p-1} \in \Dom(\calE)$.  Set $g = f^{p/2}$.
  Since $p/2 \ge p-1$, $g \in \Dom(\calE)$ and thus
  \begin{align*}
    \beta_q(\mu(f^p) - \mu(f)^p) = \beta_q (\mu(g^2) - \mu(g^{q})^{2/q}) \le (2-q)\calE(g,g) = \frac{2(p-1)}{p} \calE(g,g)
  \end{align*}
  By Lemma~\ref{L:ab-inequality}, $\calE(g,g)\le \frac{p^2}{4(p-1)}\calE(f,f^{p-1})$, which implies~\eqref{eq:Beckner-GJ}.

The second part of the proposition follows by the first inequality of Lemma~\ref{L:ab-inequality} for functions $g$ separated from zero and infinity (the assumption is needed in order to assure that for $f = g^{2/p}$ we have $f,f^{p-1} \in \Dom(\calE)$. An extension to general functions can be obtained by appropriate truncations analogously as in the other implications we have considered so far. Since we do not use this implication in any part of this paper, we skip the details.
\end{proof}

Finally, let us show that the Poincar\'e inequality is implied both by the modified log-Sobolev inequality~\eqref{eq:mlS-log} and by Beckner's inequality~\eqref{eq:Beckner-GJ} (with fixed $p$)

\begin{proposition}
\label{prop:mLSI-to-Poinc}
If the modified log-Sobolev inequality~\eqref{eq:mlS-log} holds with constant $\rho_0$, then the Poincar\'e inequality~\eqref{eq:Poinc} holds with $\lambda \ge \rho_0/2$.
\end{proposition}

\begin{proof}
  Again the argument is well known and one just needs to adjust it to our setting. If $f \in \Dom(\calE)$ is bounded (say $\|f\|_\infty = M$), then set $g_\varepsilon = e^{\varepsilon f}$. Using Taylor's expansion we get
  \begin{displaymath}
    \Ent_\mu (g_\varepsilon) = \frac{1}{2}\varepsilon^2 \Var_\mu (f) + o(\varepsilon^2).
  \end{displaymath}
On the other hand, using the inequality
\begin{displaymath}
  (e^a - e^b)(a - b)\le e^{\varepsilon M}(a-b)^2
\end{displaymath}
valid for $a,b \in [-\varepsilon M,\varepsilon M]$, together with Assumption~\ref{a:I}, we obtain
\begin{displaymath}
  \calE(g_\varepsilon, \log g_\varepsilon) \le \varepsilon^2 e^{\varepsilon M} \calE(f,f).
\end{displaymath}
To obtain~\eqref{eq:Poinc} for bounded functions it is thus enough to apply~\eqref{eq:mlS-log} to $g_\varepsilon$ and let $\varepsilon \to 0^+$.
To extend this to $\Dom(\calE)$ note that the Poincar\'e inequality for bounded functions implies (in fact is equivalent to)~\eqref{eq:Beckner-GJ} with $p=2$ for bounded functions. Thus by Assumption~\ref{a:I} and Lemma~\ref{L:Bec=>finite-mean} if $f \in \Dom(\calE)$, then $f$ is square integrable. It is thus enough to set $f_t = \max(-t,\min(f,t))$ for $t > 0$, apply~\eqref{eq:Poinc} to $f_t$ and pass with $t\to \infty$, using the fact that by Assumption~\ref{a:I}, $\calE(f_t,f_t) \le \calE(f,f)$.
\end{proof}

\begin{proposition}
\label{prop:Bec-to-Poinc}
Let $p \in (1,2]$. If Beckner's inequality~\eqref{eq:Beckner-GJ} holds with constant $\alpha_p$, then the Poincar\'e inequality~\eqref{eq:Poinc} holds with constant $\lambda \ge \alpha_p$.
\end{proposition}

\begin{proof}
As in the proof of Proposition~\ref{prop:mLSI-to-Poinc}, it is enough to prove~\eqref{eq:Poinc} for bounded functions. Assume thus that $f \in \Dom(\calE)$ is bounded.
Then, for sufficiently small $\varepsilon$, $(1+\varepsilon f)^{p-1} \in \Dom(\calE)$. Thus
\begin{displaymath}
  \alpha_p\Big(\mu((1+\varepsilon f)^p) - (\mu(1+\varepsilon f))^p\Big) \le \frac{p}{2}\calE(1+ \varepsilon f,(1+\varepsilon f)^{p-1}) = \frac{\varepsilon p}{2} \calE( f,(1+\varepsilon f)^{p-1}).
\end{displaymath}
The Taylor expansion reveals that for $\varepsilon \to 0$,
\begin{align*}
\mu((1+\varepsilon f)^p) - (\mu(1+\varepsilon f))^p = \frac{1}{2}p(p-1)\varepsilon^2 \Var_\mu(f)  + o(\varepsilon^2).
\end{align*}
On the other hand
\begin{displaymath}
  \varepsilon\calE(f,(1+\varepsilon f)^{p-1}) = (p-1)\varepsilon^2\calE(f,f) + \varepsilon \calE( f, (1+\varepsilon f)^{p-1} - 1 - (p-1)\varepsilon f).
\end{displaymath}
To finish the proof it is thus enough to show that
\begin{displaymath}
  \calE( f, (1+\varepsilon f)^{p-1} - 1 - (p-1)\varepsilon f) = o(\varepsilon)
\end{displaymath}
for $\varepsilon \to 0+$. Denote $M = \|f\|_\infty$ and denote $g(x) = (1+x)^{p-1} - 1 - (p-1)x$. For $a,b \in [-M,M]$ we have
\begin{displaymath}
  |a - b||g(\varepsilon a)  - g(\varepsilon b)| \le \varepsilon (a-b)^2 A_\varepsilon.
\end{displaymath}
where $A_\varepsilon = \sup_{t \in [-\varepsilon M,\varepsilon M]} |g'(t)|$. Thus, by Assumption~\ref{a:I},
$\calE( f, (1+\varepsilon f)^{p-1} - 1 - (p-1)\varepsilon f) \le \varepsilon A_\varepsilon \calE(f,f)$ and it remains to show that $\lim_{\varepsilon \to 0+} A_\varepsilon = 0$. This is however true, since $g$ is continuously differentiable in the neighbourhood of $0$ and $g'(0) = 0$.
\end{proof}

Finally we address the question of the monotonicity of the constants in Beckner's inequalities.

\begin{lemma}
\label{L:ab-inequality_for_alpha_p}
 For $1 < p-\varepsilon < p \leq 2$ and $a,b,>0$,
 \begin{multline*}
  \bigl(a^{(p-\varepsilon)/p}-b^{(p-\varepsilon)/p}\bigr)\bigl(a^{(p-\varepsilon)(p-1)/p}-b^{(p-\varepsilon)(p-1)/p}\bigr)\\
  \leq
  \frac{(p-\varepsilon)^2(p-1)}{p^2(p-\varepsilon-1)} (a-b)(a^{p-\varepsilon - 1}-b^{p- \varepsilon - 1}).
 \end{multline*}
\end{lemma}

\begin{proof}
By the integral version of Chebyshev's sum inequality,
\begin{align*}
 \MoveEqLeft \frac{a^{(p-\varepsilon)/p}-b^{(p-\varepsilon)/p}}{a-b} \cdot \frac{a^{(p-\varepsilon)(p-1)/p}-b^{(p-\varepsilon)(p-1)/p}}{a-b}
  \\
&=
\frac{(p-\varepsilon)/p}{a-b} \int_a^b s^{(p-\varepsilon)/p -1}ds \cdot \frac{(p-\varepsilon)(p-1)/p}{a-b} \int_a^b s^{(p-\varepsilon)(p-1)/p -1}ds\\
&\leq \frac{(p-\varepsilon)^2(p-1)}{p^2} \frac{1}{a-b} \int_a^b s^{p-\varepsilon-2} ds  = \frac{(p-\varepsilon)^2(p-1)}{p^2(p-\varepsilon-1)} \frac{a^{p-\varepsilon - 1}-b^{p- \varepsilon - 1}}{a-b}
\end{align*}
(note that $s\mapsto s^{p-\varepsilon-2}$ and $s\mapsto s^{(p-\varepsilon)(p-1)/p -1}$ are both decreasing).
\end{proof}

\begin{proposition}
 Suppose that for some $p\in(1,2]$ Beckner's inequality \eqref{eq:Beckner-GJ} holds with constant $\alpha_p>0$. Let $0<\varepsilon<p-1$. Then,  Beckner's inequality  \eqref{eq:Beckner-GJ} holds for $p-\varepsilon$ (in place of $p$) with constant
 \[
\alpha_{p-\varepsilon}   \geq  \frac{(p-\varepsilon-1)p}{(p-\varepsilon)(p-1)}\alpha_p.
\]

In particular, if  Beckner's inequality \eqref{eq:Beckner-GJ} holds for all (equivalently: for some) $p\in(1,2]$, then the function $ p\mapsto \frac{p}{p-1}\alpha_p^\opt$, $p\in(1,2]$, is nonincreasing.
\end{proposition}

 \begin{proof}
 By Jensen's inequality, \eqref{eq:Beckner-GJ}, and Lemma~\ref{L:ab-inequality_for_alpha_p},
 \begin{align*}
  \alpha_p \bigl( \mu(f^{p-\varepsilon}) - \mu(f)^{p-\varepsilon}\bigr) &\leq \alpha_p \bigl( \mu(f^{p-\varepsilon}) - \mu(f^{(p-\varepsilon)/p})^{p}\bigr) \\
  &\leq \frac{p}{2} \calE(f^{(p-\varepsilon)/p},f^{(p-\varepsilon)(p-1)/p}) \\
  &\leq  \frac{p-\varepsilon}{2} \cdot \frac{(p-\varepsilon)(p-1)}{(p-\varepsilon-1)p} \calE(f,f^{p-\varepsilon-1}).
 \end{align*}
 This finishes the proof.
\end{proof}

\begin{remark}
 In the case of Beckner's inequality \eqref{eq:Beckner-LO} we have:
 \begin{itemize}
 \item $q\mapsto \frac{1}{2-q} \beta_q^\opt$, $q\in[1,2)$ is nondecreasing (just by the Jensen inequality),
 \item $q\mapsto \frac1q \beta_q^\opt$, $q\in[1,2)$ is nonincreasing (by a lemma proved in \cite{MR1796718} concerning the monotonity of the function $q\mapsto  \frac{\mu(g^2) - \mu(g^q)^{2/q}}{1/q-1/2}$, $q\in[1,2)$).
 \end{itemize}
\end{remark}

\section{Connections with Dirichlet forms}
\label{app-Markov}

In this section we will provide a link between our assumptions and the usual theory of Dirichlet forms associated with Markov semigroups, in particular showing that our main assumptions are satisfied in this setting.
As reference we suggest the monographs~\cite{MR1303354,MR3155209}.
Recall that we work on a probability space $(\calX,\calB,\mu)$.

Let $P\colon [0,\infty)\times \calX \times \calB \to [0,1]$ be a homogeneous Markov transition function for which $\mu$ is an invariant measure. We will assume that $P$ is reversible with respect to $\mu$. We will often write $P_t(x,B)$ for $P(t,x,B)$ and we will also denote by $(P_t)_{t\ge 0}$ the semigroup of operators on $L_2(\calX,\mu)$ related to the transition function and defined as
\begin{displaymath}
  P_t f(x) = \int_\calX f(y)P_t(x,dy).
\end{displaymath}
We will assume that this semigroup is strongly continuous.

It can be shown that for each $f\in L_2(\calX,\mu)$ the function
\begin{displaymath}
  t\mapsto \frac{1}{2t}\int_\calX\int_\calX (f(y) - f(x))^2P_t(x,dy)\mu(dx)
\end{displaymath}
is non-increasing. Denoting
\begin{displaymath}
  \Dom(\calE) = \Big\{f \in L_2(\calX,\mu)\colon \sup_{t\ge 0} \frac{1}{2t}\int_\calX\int_\calX (f(y) - f(x))^2P_t(x,dy)\mu(dx) < \infty \Big\}
\end{displaymath}
and defining for $f,g\in \calE$,
\begin{displaymath}
  \calE(f,g) = \lim_{t\to 0} \frac{1}{2t}\int_\calX\int_\calX (f(y)-f(x))(g(y)-g(x))P_t(x,dy)\mu(dx)
\end{displaymath}
we obtain a nonnegative definite symmetric quadratic form.

In particular, for $f \in \Dom(\calE)$ we obtain
\begin{displaymath}
  \calE(f,f) = \lim_{t\to 0} \frac{1}{2t}\int_\calX\int_\calX (f(y) - f(x))^2P_t(x,dy)\mu(dx)
\end{displaymath}
and it is straightforward to check that the Assumption~\ref{a:I} is satisfied in this case.

Let us now discuss the Assumption~\ref{a:II}. If $L$ is the infinitesimal operator of the semigroup $(P_t)_{t\ge 0}$, defined as
\begin{displaymath}
  Lf = \lim_{h\to 0^+} \frac{P_h f - f}{h}
\end{displaymath}
with the convergence in the $L_2$ sense, and $f,g \in \Dom(L)$, then
\begin{displaymath}
  \calE(f,g) = -\int_\calX fLg d\mu.
\end{displaymath}
If also $fg \in \Dom(L)$, then one obtains further the equality~\eqref{eq:champ-form} where the carr\'e du champ operator $\Gamma$ is given by~\eqref{eq:carre-du-champ}. One shows that $\Gamma(f,f) \ge 0$.
In most applications the operator $\Gamma$ is first defined on a suitable algebra of functions $\calA_0 \subseteq \Dom(L)$ and then extended to some larger class $\calA$. This is the situation described, e.g., in Chapter 1.14 of~\cite{MR3155209}. In the case of diffusions on Riemannian manifolds one usually takes $\calA_0$ to be the algebra of smooth compactly supported functions and $\calA$ -- the algebra of all smooth functions. However in the abstract setting there is no canonical choice of $\calA$, so we will stick here to the basic case of $\calA \subseteq \Dom(L)$ and refer to Chapter 3 of~\cite{MR3155209} for the axiomatic approach, which allows to introduce a general framework for an abstract counterpart of the theory of diffusions in $\RR^n$.

The following proposition shows that in our basic setting the first part of Assumption~\ref{a:II} is satisfied for every algebra $\calA \subseteq \Dom(L)$.

\begin{proposition}\label{prop:assumption-2-1} Assume that $f\colon \calX \to \RR$ is a bounded function and $f, f^2\in \Dom(L)$. Let $t,t' \ge 1$ satisfy $\frac{1}{t} + \frac{1}{t'} = 1$. Then for every differentiable convex, nondecreasing function $\varphi \colon [0,\infty) \to \RR$ and every $c \in \RR$
\begin{align*}
\calE(\varphi(|f+c|),|f+c|) \le 2\int_\calX \varphi'(|f+c|) \Gamma(f) d\mu \le 2\Big\|\varphi'(|f+c|)\Big\|_{t'}\Big\|\Gamma(f)\Big\|_t.
\end{align*}
\end{proposition}

\begin{proof}
Note that if $f^2, f \in \Dom(L)$ and $f$ is bounded then
\begin{multline}\label{eq:convergence}
\frac{1}{u} \int_\calX (f(x) - f(y))^2 P_u(x,dy)\\
 = \frac{1}{u} \int_\calX (f^2(y) - f^2(x)) P_u(x,dy) - 2f(x) \frac{1}{u} \int_\calX (f(y) - f(x))P_u(x,dy) \stackrel{u\to 0+}{\to} 2\Gamma(f)
\end{multline}
in $L_2$.

By boundedness of $f$, the fact that convex functions are locally Lipschitz, and Assumption~\ref{a:I} (which we know to be satisfied in the Markov case), $\varphi(|f+c|), |f+c| \in \Dom(\calE)$. Moreover, denoting $g=f+c$,
\begin{align*}
&\calE(\varphi(|g|),|g|) \\
&= \lim_{u\to 0+} \frac{1}{2u}\int_\calX\int_\calX (\varphi(|g(x)|) - \varphi(|g(y)|))(|g(x)| - |g(y)|)P_u(x,dy)\mu(dx) \\
&= \lim_{u\to 0+} \frac{1}{u}\int_\calX\int_\calX (\varphi(|g(x)|) - \varphi(|g(y)|))_+(|g(x)| - |g(y)|)_+ P_{u}(x,dy)\mu(dx) \\
&\le \liminf_{u\to 0+}  \int_\calX \varphi'(|g(x)|) \frac{1}{u} \int_\calX (f(x) - f(y))_+^2P_{u}(x,dy)\mu(dx)\\
&\le \liminf_{u\to 0+}  \int_\calX \varphi'(|g(x)|) \frac{1}{u} \int_\calX (f(x) - f(y))^2P_{u}(x,dy)\mu(dx)\\
&= 2 \int_\calX \varphi'(|f(x)+c|) \Gamma(f)d\mu.
\end{align*}
where in the second equality we used reversibility of the semigroup together with monotonicity of $\varphi$, and in the first inequality -- convexity of $\varphi$. The last equality follows by boundedness of $f$ and~\eqref{eq:convergence}.

This proves the first inequality of the proposition. The second one follows by H\"older's inequality.
\end{proof}

As for the second part of Assumption~\ref{a:II}, it is satisfied, e.g., if $\calA \subseteq \Dom(L)$ is an algebra stable under compositions with smooth functions vanishing at zero, which is a common assumption in this context (see Chapter 1.13 of~\cite{MR3155209}). Indeed, in this case one can define an appropriate sequence of smooth contractions $\psi_n\colon \RR\to \RR$ with $\psi_n(0) = 0$, converging to $\psi(x) = x$ pointwise
and take $f_n = \psi_n(f)$. Then clearly $f_n \to f$ and $|f_n|\le |f|$ $\mu$-a.s.
 Moreover, for general $f$ such that $f,f^2\in \Dom(L)$ one still has~\eqref{eq:convergence} but this time in $L_1$. Thus using the contraction property of $\psi_n$ one can easily show that
\begin{displaymath}
  \Gamma(f_n) = \Gamma(\psi_n(f)) \le \Gamma(f)
  \quad
  \mu\text{-a.s.}
\end{displaymath}

Combining this observation with Proposition~\ref{prop:assumption-2-1} one obtains.

\begin{proposition}\label{prop:assumption-2}
If $\calA \subseteq \Dom(L)$ is an algebra stable under composition with smooth functions vanishing at zero, then the Assumption~\ref{a:II} is satisfied.
\end{proposition}

This gives the basic setting for applying moment estimates of Proposition~\ref{prop:moments-abstract} in the Markovian case. Clearly, in concrete applications the moment inequalities can be extended to larger classes of functions -- the details of such an extension and the choice of the class of functions may however depend on the particular case. We again refer to Chapter 3 of~\cite{MR3155209} for an extensive discussion of this issue.

\bibliographystyle{amsplain}
\bibliography{mLS_and_Bec}

\end{document}